\documentclass[reqno,11pt]{article}
\usepackage{a4wide,color,eucal,enumerate,mathrsfs}
\usepackage[normalem]{ulem}
\usepackage{amsmath,amssymb,epsfig,bbm}
\usepackage[latin1]{inputenc}
\usepackage{pdfsync}
\numberwithin{equation}{section}

\newcommand{\C}{\mathbb{C}}
\newcommand{\D}{\mathbb{D}}

\newcommand{\N}{\mathbb{N}}
\newcommand{\Q}{\mathbb{Q}}
\newcommand{\R}{\mathbb{R}}

\newcommand{\BB}{\mathscr{B}}

\newcommand{\cE}{{\ensuremath{\mathcal E}}}

\newcommand{\mm}{{\mbox{\boldmath$m$}}}

\newcommand{\ggamma}{{\mbox{\boldmath$\gamma$}}}

\newcommand{\ppi}{{\mbox{\boldmath$\pi$}}}

\newcommand{\sggamma}{{\mbox{\scriptsize\boldmath$\gamma$}}}

\newcommand{\sfd}{{\sf d}}

\newcommand{\sfh}{{\sf h}}

\newcommand{\sfC}{{\sf C}}
\newcommand{\sfD}{{\sf D}}

\newcommand{\sfM}{{\sf M}}

\newcommand{\sfS}{{\sf S}}

\newcommand{\rme}{{\mathrm e}}
\newcommand{\Kliminf}{K\kern-3pt-\kern-2pt\mathop{\rm lim\,inf}\limits}  % Kuratowski liminf di insiemi
\newcommand{\supp}{\mathop{\rm supp}\nolimits}   % supporto 
\newcommand{\diam}{\mathop{\rm diam}\nolimits}   % diametro
\newcommand{\Lip}{\mathop{\rm Lip}\nolimits}          %Lipschitz constant
\renewcommand{\d}{{\mathrm d}}
\newcommand{\dt}{{\d t}}

\newcommand{\restr}[1]{\lower3pt\hbox{$|_{#1}$}}

\newcommand{\down}{\downarrow}              %frecce in su e in giu nei limiti

\newcommand{\eps}{\varepsilon}  
\newcommand{\nchi}{{\raise.3ex\hbox{$\chi$}}}

\newcommand{\media}{\mkern12mu\hbox{\vrule height4pt           %
          depth-3.2pt                                 %  da usare
          width5pt}\mkern-16.5mu\int\nolimits}        %  nelle formule
\newcommand{\forevery}{\text{for every }}
\newcommand{\Pc}[2]{\overline{#1}\kern-2pt^{\vphantom 0}_{#2}}

\newcommand{\EVI}[4]{\mathrm{EVI}_{#4}(#1,{#2}_W,#3)}

\newcommand{\BorelSets}[1]{\BB(#1)}

\newcommand{\Probabilities}[1]{\mathscr P(#1)}          % misure di probabilita'
\newcommand{\ProbabilitiesTwo}[1]{\mathscr P_2(#1)}     % misure di probabilita' con momento quadratico finito
\newenvironment{proof}{\removelastskip\par\medskip   % inizio e fine dimostrazione
\noindent{\em Proof.}
\rm}{\penalty-20\null\hfill$\square$\par\medbreak}

\newtheorem{theorem}{Theorem}[section]

\newtheorem{corollary}[theorem]{Corollary}
\newtheorem{lemma}[theorem]{Lemma}
\newtheorem{proposition}[theorem]{Proposition}
\newtheorem{definition}[theorem]{Definition}

\newtheorem{example}[theorem]{Example}
\newtheorem{remark}[theorem]{Remark}

\newcommand{\ent}[1]{{\rm Ent}_{\mm}(#1)}              % entropia. Ha come argomento la misura di cui si calcola l'entropia (quella di riferimento e' sempre $\mm$)
\newcommand{\entv}{{\rm Ent}_{\mm}}                    % entropia senza argomento. La uso principalmente quando scrivo la slope
\newcommand{\prob}{\Probabilities}
\newcommand{\probt}{\ProbabilitiesTwo}
\newcommand{\geo}{{\rm{Geo}}}                       % insieme delle geodetiche
\newcommand{\e}{{\rm{e}}}                           % mappa di valutazione, bisogna mettere `a mano' il tempo t
\newcommand{\gopt}{{\rm{GeoOpt}}}                   % displacement plans ottimali
\newcommand{\opt}[2]{{\rm{Opt}}(#1,#2)}
\newcommand{\adm}[2]{{\rm{Adm}}(#1,#2)}
\newcommand{\sppi}{{\mbox{\scriptsize\boldmath$\pi$}}}      % $\pi$ bold e piccolo (per i subscript). Mancava dalla lista.
\newcommand{\fr}{\hfill$\blacksquare$}                      %quadratino nero alla fine del remark, se non vi piace, la cosa migliore e' `svuotare' la macro, cosi' non bisogna intervenire sul testo

\newcommand{\X}{\mathbb X}

\newcommand{\relgrad}[1]{|\nabla #1|_*}                     % gradiente ottenuto per rilassamento
\newcommand{\weakgrad}[1]{|\nabla #1|_w}                % gradient ottenuto integrando lungo "quasi tutte" le geodetiche
\newcommand{\cartgrad}[1]{|\nabla #1|_c}

\newcommand{\res}{\mathop{\hbox{\vrule height 7pt width .5pt depth 0pt
\vrule height .5pt width 6pt depth 0pt}}\nolimits} % macro per la restrizione

\renewcommand{\C}{{\rm Ch}}

\newcommand{\entr}[2]{{\rm Ent}_{#2}(#1)}              % entropia con misura di riferimento da assegnare

\newcommand{\rcd}[2]{RCD(#1,#2)}

\newcommand{\muF}[1]{\mu_{F,#1}}
\newcommand{\heatl}{{\sf H}}
\newcommand{\heatw}{{\mathscr H}}
\newcommand{\weakgrado}[2]{|\nabla #1|_{w,#2}}                % gradient ottenuto integrando lungo "quasi tutte" le geodetiche

\newcommand{\ke}[2]{\heatw_{#2}(\delta_{#1})}
\newcommand{\ked}[2]{\rho_{#2}[#1]}\renewcommand{\C}{{\sf Ch}}

\newcommand{\Deltam}{\Delta}
\newcommand{\Deltamc}{\Delta^c}

\newcommand{\AC}[3]{\mathrm{AC}^{#1}(#2;#3)}
\newcommand{\CC}[2]{\mathrm{C}(#1;#2)}
\newcommand{\Cb}{\mathrm C_b}
\newcommand{\calT}{{\mathcal T}}
\renewcommand{\EVI}{\ensuremath{\mathrm{EVI}}}
\newcommand{\Di}{\mathsf{Di}}
\renewcommand{\mm}{\mathfrak m}

\title{Metric measure spaces with Riemannian Ricci \\ curvature bounded from below}
\begin{document}

\author{Luigi Ambrosio\
   \thanks{Scuola Normale Superiore, Pisa. email: \textsf{l.ambrosio@sns.it}}
   \and
   Nicola Gigli\
   \thanks{Nice University. email: \textsf{nicola.gigli@unice.fr}}
 \and
   Giuseppe Savar\'e\
   \thanks{Universit\`a di Pavia. email: \textsf{giuseppe.savare@unipv.it}}
   }

\maketitle

\begin{abstract} In this paper we introduce a synthetic notion of
Riemannian Ricci bounds from below for metric measure spaces
$(X,\sfd,\mm)$ which is stable under measured Gromov-Hausdorff
convergence and rules out Finsler geometries. It can be given in
terms of an enforcement of the Lott, Sturm and Villani geodesic
convexity condition for the entropy coupled with the linearity of
the heat flow. Besides stability, it enjoys the same tensorization,
global-to-local and local-to-global properties. In these spaces,
that we call $RCD(K,\infty)$ spaces, we prove that the heat flow
(which can be equivalently characterized either as the flow
associated to the Dirichlet form, or as the Wasserstein gradient
flow of the entropy) satisfies Wasserstein contraction estimates and
several regularity properties, in particular Bakry-Emery estimates
and the $L^\infty-{\rm Lip}$ Feller regularization. We also prove
that the distance induced by the Dirichlet form coincides with
$\sfd$, that the local energy measure has density given by the
square of Cheeger's relaxed slope and, as a consequence, that the
underlying Brownian motion has continuous paths. All these results
are obtained independently of Poincar\'e and doubling assumptions on
the metric measure structure and therefore apply also to spaces
which are not locally compact, as the infinite-dimensional ones.
\end{abstract}

\tableofcontents

\section{Introduction}

The problem of finding synthetic notions of Ricci curvature bounds
from below has been a central object of investigation in the last
few years. What became clear over time (see in particular
\cite{Fukaya87} and \cite[Appendix 2]{Cheeger-Colding97}), is that
the correct class of spaces where such a synthetic notion can be
given, is that of metric measure spaces, i.e. metric spaces equipped
with a reference measure which one might think of as volume measure.
The goal is then to find a notion consistent with the smooth
Riemannian case, which is sufficiently weak to be stable under
measured Gromov-Hausdorff limits. The problem of having stability is
of course in competition with the necessity to find a condition as
restrictive as possible, to describe efficiently the closure of the
class of Riemannian manifolds with Ricci curvature uniformly bounded
from below.

In their seminal papers Lott-Villani \cite{Lott-Villani09} and Sturm
\cite{Sturm06I} independently attacked these questions with tools
based on the theory of optimal transportation, devising stable and
consistent notions. In these papers, a metric measure space
$(X,\sfd,\mm)$ is said to have Ricci curvature bounded from below by
$K\in\R$ (in short: it is a $CD(K,\infty)$ space) if the relative
entropy functional
$$
\entr{\mu}{\mm}:=\int\rho\log\rho\,\d\mm\qquad\qquad\text{with
$\mu=\rho\mm$}
$$
is $K$ geodesically convex on the Wasserstein space $(\probt
X,W_2)$.

Also, in \cite{Lott-Villani09}, \cite{Sturm06II} a synthetic notion
$CD(K,N)$ of having Ricci curvature bounded from below  by $K$ and
dimension bounded above by $N$ was given (in \cite{Lott-Villani09}
only the case $CD(0,N)$ was considered, for $N<\infty$), and a
number of geometric consequences of these notions, like
Brunn-Minkowski and Bishop-Gromov inequalities, have been derived.
In \cite{Lott-Villani-Poincare} it was also proved that, at least
under the nonbranching assumption, the $CD(K,N)$ condition implies
also the Poincar\'e inequaliy, see also \cite{Rajala11} for some
recent progress in this direction.

An interesting fact, proved by Cordero-Erasquin, Sturm and Villani
(see the conclusions of \cite{Villani09}), is that $\R^d$ equipped
with any norm and with the Lebesgue measure, is a $CD(0,N)$ space.
More generally, Ohta showed in \cite{Ohta09} that any smooth compact
Finsler manifold is a $CD(K,N)$ space for appropriate finite $K,N$.
However, a consequence of the analysis of tangent spaces done in
\cite{Cheeger-Colding97}, is that a Finsler manifold arises as limit
of Riemannian manifolds with Ricci curvature uniformly bounded below
and dimension uniformly bounded from above, if and only if it is
Riemannian (the case of possibly unbounded dimension of the
approximating sequence is covered by the stability of the heat flow
proved in \cite{Gigli10} in conjunction with the fact that the heat
flow on a Finsler manifold is linear if and only if it is Riemannian
\cite{Sturm-Ohta-CPAM}).

Therefore it is natural to look for a synthetic and stable notion of
Ricci curvature bound which rules Finsler spaces out. This is the
scope of this paper. What we do, roughly said, is to add to the
$CD(K,\infty)$ condition the linearity of the heat flow, see below
for the precise definition.

Before passing to the description of the results of this paper, we
recall the main results of the ``calculus'' in the first paper of
ours \cite{Ambrosio-Gigli-Savare11}, needed for the development and
the understanding of this one. The main goals of
\cite{Ambrosio-Gigli-Savare11} have been the identification of two
notions of gradient and two gradient flows. \\ The first notion of
gradient, that we call \emph{minimal relaxed gradient} and denote by
$\relgrad{f}$, is inspired by Cheeger's work \cite{Cheeger00}: it is
the local quantity that provides integral representation to the
functional $\C(f)$ given by
\begin{equation}
\C(f):=\frac 12\inf\left\{\liminf_{n\to\infty}\int|\nabla
f_h|^2\,\d\mm:\ f_h\in {\rm Lip}(X),\,\,\int_X|f_h-f|^2\,\d\mm\to
0\right\}\label{eq:33}
\end{equation}
(here $|\nabla f_h|$ is the so-called local Lipschitz constant of
$f_h$), so that $\C(f)=\tfrac12\int\relgrad{f}^2\,\d\mm$. The second
notion of gradient, that we call \emph{minimal weak upper gradient}
and denote by $\weakgrad{f}$ is, instead, inspired by
Shanmugalingam's work \cite{Shanmugalingam00} and based on the
validity of the upper gradient property
$$
|f(\gamma_1)-f(\gamma_0)|\leq\int_\gamma\weakgrad{f}
$$
on ``almost all'' curves $\gamma$. We proved that the minimal weak
upper gradient and the minimal relaxed gradient coincide. In
addition, although our notions of null set of curves differs from
\cite{Shanmugalingam00} and the definition of $\C$ differs from
\cite{Cheeger00}, we prove, a posteriori, that the gradients
coincide with those in \cite{Cheeger00}, \cite{Shanmugalingam00}.
Since an approximation by Lipschitz functions is implicit in the
formulation \eqref{eq:33}, this provides a density result of
Lipschitz function in the weak Sobolev topology without any doubling
and Poincar\'e assumption on $(X,\sfd,\mm)$. In the context of the
present paper, where $\C$ will be a quadratic form even when
extended to $\mm$-measurable functions using weak upper gradients,
this approximation result yields
the density of Lipschitz functions in the strong Sobolev topology.\\
The concept of minimal relaxed gradient can be used in connection
with ``vertical'' variations of the form $\eps\mapsto f+\eps g$,
which occur in the study of the $L^2(X,\mm)$-gradient flow of $\C$,
whose semigroup we shall denote by $\heatl_t$. On the other hand,
the concept of minimal weak upper gradient is relevant in connection
with ``horizontal'' variations of the form $t\mapsto f(\gamma_t)$,
which play an important role when study the derivative of $\entv$
along geodesics. For this reason their identification is crucial, as
we will see in Section~\ref{se:formule}. Given this identification
for granted, in the present paper most results will be presented and
used at the
level of minimal weak upper gradients, in order to unify the exposition. \\
Finally, in $CD(K,\infty)$ spaces we identified the $L^2$-gradient
flow $\heatl_t$ of $\C$ (in the sense of the Hilbertian theory
\cite{Brezis73}) with the $W_2$-gradient flow of $\entv$ in the
Wasserstein space of probability measures $\mathscr P_2(X)$ (in the
sense of De Giorgi's metric theory, see
\cite{Ambrosio-Gigli-Savare08} and, at this level of generality,
\cite{Gigli10}), which we shall denote by $\heatw_t$. A byproduct of
this identification is an equivalent description of the entropy
dissipation rate along the flow, equal to
$4\int\weakgrad{\sqrt{\heatl_tf}}^2\,\d\mm$ and to the square of the
metric derivative of $\mu_t=\heatw_t (f\mm)$ w.r.t. $W_2$.\\ All
these results have been obtained in \cite{Ambrosio-Gigli-Savare11}
under very mild assumptions on $\mm$, which include all measures
such that $ \rme^{-c\,\sfd^2(x,x_0)}\mm$ is finite for some $c>0$
and $x_0\in X$. In this paper, in order to minimize the
technicalities, we assume that $\mm$ is a probability measure with
finite second moment. On the other hand, no local compactness
assumption on $(X,\sfd)$ will be needed, so that
infinite-dimensional spaces fit well into this theory.

Coming back to this paper, we say that a metric measure space
$(X,\sfd,\mm)$ has \emph{Riemannian Ricci} curvature bounded from
below by $K\in\R$, and write that $(X,\sfd,\mm)$ is $\rcd K\infty$,
if one of the following equivalent conditions hold:
\begin{itemize}
\item[(i)] $(X,\sfd,\mm)$ is a strong $CD(K,\infty)$ space and the
$W_2$-gradient flow of $\entv$ is additive.
\item[(ii)] $(X,\sfd,\mm)$ is a strong $CD(K,\infty)$ space and
$\C$ is a quadratic form in $L^2(X,\mm)$, so that the $L^2$-heat
flow of $\C$ is linear.
\item[(iii)] The gradient flow of $\entv$ exists for all initial
data $\mu$ with $\supp\mu\subset\supp\mm$ and satisfies the $\EVI_K$
condition.
\end{itemize}
The equivalence of these conditions is not at all obvious, and its
proof is actually one of the main results of this paper.

Observe that in $(i)$ and $(ii)$ the $CD(K,\infty)$ is enforced on
the one hand considering a stronger convexity condition (we describe
this condition in the end of the introduction, being this aspect
less relevant), on the other hand adding linearity of the heat flow.
A remarkable fact is that this combination of properties can be
encoded in a \emph{single} one, namely the $\EVI_K$ property. This
latter property can be expressed by saying that for all
$\nu\in\probt{X}$ with finite entropy the gradient flow
$\heatw_t(\mu)$ starting from $\mu$ satisfies
\begin{equation}\label{eq:EVIintro}
\frac \d{\d t}\frac{W_2^2(\heatw_t(\mu),\nu)}2+\frac
K2W^2_2(\heatw_t(\mu),\nu)+\entr{\heatw_t(\mu)}{\mm}\leq\entr{\nu}{\mm}\qquad\text{for
a.e. $t\in (0,\infty)$.}
\end{equation}

It is immediate to see that the $\rcd K\infty$ notion is consistent
with the Riemannian case: indeed, uniqueness of geodesics in
$(\probt M,W_2)$ between absolutely continuous measures and the
consistency of the $CD(K,\infty)$ notion, going back to
\cite{Cordero-McCann-Schmuckenschlager01,Sturm-VonRenesse05}, yield
that manifolds are strong $CD(K,\infty)$ spaces (see below for the
definition), and the fact that $\C$ is quadratic is directly encoded
in the Riemannian metric tensor, yielding the linearity of the heat
flow. On the other hand, the stability of $\rcd K\infty$ bounds with
respect to the measured Gromov-Hausdorff convergence introduced by
Sturm \cite{Sturm06I} is a consequence, not too difficult, of
condition $(iii)$ and the general stability properties of $\EVI_K$
flows (see also \cite{Savare07,Savare10} for a similar statement).
We remark also that thanks to the results in
\cite{Gigli-Ohta10,Ohta09,Petrunin,ZZ1}, compact and
finite-dimensional spaces with Alexandrov curvature bounded from
below are $\rcd K\infty$ spaces.

Besides this, we prove many additional properties of $\rcd K\infty$
spaces. Having \eqref{eq:EVIintro} at our disposal at the level of
measures, it is easy to obtain fundamental solutions, integral
representation formulas, regularizing and contractivity properties
of the heat flow, which exhibits a strong Feller regularization from
$L^\infty(X,\mm)$ to Lipschitz. Denoting by
$W^{1,2}(X,\sfd,\mm)\subset L^2(X,\mm)$ the finiteness domain of
$\C$, the identification of the $L^2$-gradient flow of $\C$ and of
the $W_2$-gradient flow of $\entv$ in conjunction with the
$K$-contractivity of $W_2$ along the heat flow, yields, as in
\cite{GigliKuwadaOhta10}, the Bakry-Emery estimate
$$
\weakgrad{(\heatl_t f)}^2\leq
e^{-2Kt}\heatl_t(\weakgrad{f}^2)\qquad\text{$\mm$-a.e. in $X$}
$$
for all $f\in W^{1,2}(X,\sfd,\mm)$. As a consequence of this, we
prove that functions $f$ whose minimal weak upper gradient
$\weakgrad{f}$ belongs to $L^\infty(X,\mm)$ have a Lipschitz version
$\tilde f$, with ${\rm Lip}(\tilde f)\leq\|\weakgrad{f}\|_\infty$.

In connection with the tensorization property, namely the stability
of $\rcd K\infty$ metric measure spaces with respect to product
(with squared product distance given by the sum of the squares of
the distances in the base spaces), we are able to achieve it
assuming that the base spaces are nonbranching. This limitation is
due to the fact that also for the tensorization of $CD(K,\infty)$
spaces the nonbranching assumption has not been ruled out so far
(see \cite[Theorem 4.17]{Sturm06I}). On the other hand, we are able
to show that the linearity of the heat flow tensorizes, when coupled
just with the strong $CD(K,\infty)$ condition. The nonbranching
assumption on the base spaces could be avoided with a proof of the
tensorization property directly at the level of $\EVI_K$, but we did
not succeed so far in tensorizing the $\EVI_K$ condition.

Since $\C$ is a quadratic form for $\rcd K\infty $ spaces, it is
tempting to take the point of view of Dirichlet forms and to
describe the objects appearing in Fukushima's theory
\cite{Fukushima80} of Dirichlet forms. In this direction, see also
the recent work \cite{Koskela-Zhou11} and
Remark~\ref{rem:KoskelaZhou}. Independently of any curvature bound
we show that, whenever $\C$ is quadratic, a Leibnitz formula holds
and there exists a ``local'' bilinear map $(f,g)\mapsto\nabla
f\cdot\nabla g$ from $[W^{1,2}(X,\sfd,\mm)]^2$ to $L^1(X,\mm)$, that
provides an integral representation to the Dirichlet form ${\cal
E}(u,v)$ associated to $\C$. This allows us to show that the local
energy measure $[u]$ of Fukushima's theory coincides precisely with
$\weakgrad{u}^2\mm$. If the space is $\rcd K\infty$ then the
intrinsic distance $\sfd_{{\mathcal E}}$, associated to the
Dirichlet form by duality with functions $u$ satisfying $[u]\leq\mm$
is precisely $\sfd$. The theory of Dirichlet forms can also be
applied to obtain the existence of a continuous \emph{Brownian
motion} in $\rcd K\infty$ spaces, i.e. a Markov process with
continuous sample paths and transition probabilities given by
$\heatw_t(\delta_x)$.

Besides the extension to more general classes of reference measures
$\mm$, we believe that this paper opens the door to many potential
developments: among them we would like to mention the dimensional
theory, namely finding appropriate ``Riemannian'' versions of the
$CD(K,N)$ condition, and the study of the tangent space. In
connection with the former question, since $CD(K,N)$ spaces are
$CD(K,\infty)$, a first step could be analyzing them with the
calculus tools we developed and to see the impact of the linearity
of the heat flow and of the $\EVI_K$ condition stated at the level
of $\entv$. Concerning  the latter question, it is pretty natural to
expect $\rcd K\infty$ spaces to have Hilbertian tangent space for
$\mm$-a.e. point. While the proof of this result in the genuine
infinite dimensional case seems quite hard to get, if the space is
doubling and supports a local Poincar\'e inequality, one can hope to
refine Cheeger's analysis (\cite[Section~11]{Cheeger00}) in order to
achieve it.

The paper is organized as follows. In Section~\ref{sec:prem} we
introduce our main notation and the preliminary results needed for
the development of the paper. With the exception of
Section~\ref{sub4}, where we quote from
\cite{Ambrosio-Gigli-Savare11} the basic results we already alluded
to, namely the identification of weak gradients and relaxed
gradients and the identification of $L^2$-gradient flow of $\C$ and
$W_2$-gradient flow of $\entv$, the material is basically known.
Particularly relevant for us will be the $\EVI_K$ formulation of
gradient flows, discussed in Section~\ref{sec:EVI}.

In Section~\ref{se:strongcd} we introduce a convexity condition,
that we call \emph{strong} $CD(K,\infty)$, intermediate between the
$CD(K,\infty)$ condition, where convexity is required along some
geodesic, and convexity along all geodesics. It can be stated by
saying that, given any two measures, there is always an optimal
geodesic plan $\ppi$ joining them such that the $K$-convexity holds
along all the geodesics induced by weighted plans of the form
$F\ppi$, where $F$ is a bounded, non negative function with $\int
F\,\d\ppi=1$. We also know from \cite{Daneri-Savare08} that the
$\EVI_K$ condition implies convexity along all geodesics supported
in $\supp\mm$, and therefore the strong $CD(K,\infty)$ property.\\
This enforcement of the $CD(K,\infty)$ condition is needed to derive
strong $L^\infty$ bounds on the interpolating measures induced by
the ``good'' geodesic plan. These ``good'' interpolating measures
provide large class of test plans and are used to show, in this
framework, that the ``metric Brenier'' theorem
\cite[Theorem~10.3]{Ambrosio-Gigli-Savare11} holds. Roughly
speaking, this theorem states that, when one transports in an
optimal way $\mu$ to $\nu$, the transportation distance
$\sfd(\gamma_0,\gamma_1)$ depends $\ppi$-almost surely only on the
initial point $\gamma_0$ (and in particular it is independent on the
final point $\gamma_1$). Furthermore, the proof of this result
provides the equality
$\sfd(\gamma_0,\gamma_1)=|\nabla^+\varphi|(\gamma_0)=\weakgrad{\varphi}(\gamma_1)$
for $\ppi$-a.e. $\gamma$, where $\varphi$ is any Kantorovich
potential relative to $(\mu,\nu)$. This equality will be crucial for
us when proving optimal bounds for the derivative of $\entv$ along
geodesics.

In Section~\ref{se:formule} we enter the core of the paper with two
basic formulas, one for the derivative of the Wasserstein distance
along the heat flow $(\rho_t\mm)$ (Theorem~\ref{thm:derw2}),
obviously important for a deeper understanding of
\eqref{eq:EVIintro}, the other one for the derivative of the entropy
(Theorem~\ref{thm:derentr}) along a geodesic $\mu_s=\rho_s\mm$. The
proof of the first one uses the classical duality method and relates
the derivative of $W_2^2(\rho_t\mm,\nu)$ to the ``vertical''
derivative of the density $\rho_t$ in the direction given by
Kantorovich potential from $\rho_t\mm$ to $\nu$. The second one
involves much more the calculus tools we developed. The key idea is
to start from the (classical) convexity inequality for the entropy,
written in terms of the optimal geodesic plan $\ppi$ from
$\rho_0\mm$ to $\rho_1\mm$
\begin{equation}\label{eq:Alica1}
\entr{\rho_s\mm}{\mm}-\entr{\rho_0\mm}{\mm}\geq\int\log\rho_0(\rho_s-\rho_0)\,\d\mm=
\int
(\log\rho_0(\gamma_s)-\log\rho_0(\gamma_0)\bigr)\,\d\ppi(\gamma)
\end{equation}
and then use the crucial Lemma~\ref{le:horver}, relating the
``horizontal'' derivatives appearing in \eqref{eq:Alica1} to the
``vertical'' ones. In section \ref{sec:quadratic} the same lemma,
applied to suitable plans generated by the heat flow, is the key to
deduce a local quadratic structure from a globally quadratic Cheeger
energy and to develop useful calculus tools, leading in particular
to the identification of $\weakgrad u^2\mm$ with the energy measure
$[u]$ provided by the general theory of Dirichlet forms.

Section~\ref{sec:rcd} is devoted to the proof of the equivalence of
the three conditions defining $\rcd K \infty$ spaces, while the
final Section~\ref{sec:rcdprop} treats all properties of $\rcd
K\infty$ spaces we already discussed: representation, contraction,
and regularizing properties of the heat flow, relations with the
theory of Dirichlet forms and existence of the Brownian motion,
stability, tensorization. We also discuss, in the last section, the
so-called global-to-local and local-to-global implications. We prove
that the first one always holds if the subset under consideration is
convex, with positive $\mm$-measure and $\mm$-negligible boundary.
We also prove a partial result in the other direction, from local to
global, comparable to those available within the $CD(K,\infty)$
theory.

\smallskip
\noindent {\bf Acknowledgement.} The authors acknowledge the support
of the ERC ADG GeMeThNES.

\section{Preliminaries}\label{sec:prem}

\subsection{Basic notation, metric and measure theoretic
concepts}\label{sub1}

Unless otherwise stated, all metric spaces $(Y,\sfd_Y)$ we will be
dealing with are complete and separable. Given a function
$E:Y\to\R\cup\{\pm\infty\}$, we shall denote its domain $\{y:\
E(y)\in\R\}$ by $D(E)$. The slope (also called local Lipschitz
constant) $|\nabla E|(y)$ of $E$ at $y\in D(E)$ is defined by
\begin{equation}\label{eq:loclip}
|\nabla E|(y):=\limsup_{z\to y}\frac{|E(y)-E(z)|}{\sfd_Y(y,z)}.
\end{equation}
By convention we put $|\nabla E|(y)=+\infty$ if $y\notin D(E)$ and
$|\nabla E|(y)=0$ if $y\in D(E)$ is isolated.

We shall also need the one-sided counterparts of this concept,
namely the \emph{descending slope} (in the theory of gradient flows)
and the \emph{ascending slope} (in the theory of Kantorovich
potentials). Their are defined at $y\in D(E)$ by
\[
|\nabla^-E|(y):=\limsup_{z\to y}\frac{(E(z)-E(y))^-}{\sfd_Y(y,z)},
\quad
|\nabla^+E|(y):=\limsup_{z\to y}\frac{(E(z)-E(y))^+}{\sfd_Y(z,y)},
\]
with the usual conventions if either $y$ is isolated or it does not
belong to $D(E)$.

We will denote by $\CC{[0,1]}{Y}$ the space of continuous curves on
$(Y,\sfd_Y)$; it is a complete and separable metric space when
endowed with the sup norm. We also denote with $\e_t: \CC{[0,1]}Y\to
Y$, $t\in[0,1]$, the evaluation maps
\[
\e_t(\gamma):=\gamma_t\qquad\qquad\forall\gamma\in \CC{[0,1]}Y.
\]
A curve $\gamma:[0,1]\to Y$ is said to be absolutely continuous if
\begin{equation}
\label{eq:curvaac} \sfd_Y(\gamma_t,\gamma_s)\leq\int_s^tg(r)\,\d r
\qquad\forall s,\,t\in [0,1],\,\,s\leq t,
\end{equation}
for some $g\in L^1(0,1)$. If $\gamma$ is absolutely continuous, the
\emph{metric speed} $|\dot\gamma|:[0,1]\to[0,\infty]$ is defined by
\[
|\dot\gamma|:=\lim_{h\to0}\frac{\sfd_Y(\gamma_{t+h},\gamma_t)}{|h|},
\]
and it is possible to prove that the limit exists for a.e. $t$, that
$|\dot\gamma|\in L^1(0,1)$, and that it is the minimal $L^1$
function (up to Lebesgue negligible sets) for which the bound
\eqref{eq:curvaac} holds (see
\cite[Theorem~1.1.2]{Ambrosio-Gigli-Savare08} for the simple proof).

We shall denote by $\AC2{[0,1]}Y$ the class of absolutely
continuous curves with metric derivative in $L^2(0,1)$; it is easily
seen to be a countable union of closed sets in $\CC{[0,1]}Y$ and in
particular a Borel subset.

A curve $\gamma\in \CC{[0,1]}Y$ is called constant speed geodesic if
$\sfd_Y(\gamma_t,\gamma_s)=|t-s|\sfd_Y(\gamma_0,\gamma_1)$ for all
$s,\, t\in [0,1]$. We shall denote by $\geo(Y)$ the space of
constant speed geodesics, which is a closed (thus complete and
separable) subset of $\CC{[0,1]}Y$.

$(Y,\sfd_Y)$ is called a \emph{length} space if
for any
$y_0,\,y_1\in Y$ and $\eps>0$ there exists $\gamma\in\AC{}{[0,1]}Y$
such that
\begin{equation}
\gamma_0=y_0,\
\gamma_1=y_1\quad\text{and}\quad
\mathrm{Length}(\gamma):=\int_0^1 |\dot \gamma_t|\,\d t\le
\sfd_Y(y_0,y_1)+\eps.\label{eq:2}
\end{equation}
If for any $y_0,\,y_1\in Y$ one can find $\gamma$ satisfying
\eqref{eq:2} with $\eps=0$ (and thus, up to a reparameterization,
$\gamma\in\geo(Y)$), we say that $(Y,\sfd_Y)$ is a \emph{geodesic}
space. We also apply the above definitions to (even non closed)
subsets $Z\subset Y$, always endowed with the distance $\sfd_Y$
induced by $Y$. It is worth noticing that if $Z$ is a length space
in $Y$, then $\overline Z$ is a length space in $Y$ as well
\cite[Ex.~2.4.18]{Burago-Burago-Ivanov01}.

We use standard measure theoretic notation, as $\Cb(X)$ for bounded
continuous maps, $f_\sharp$ for the push forward operator induced by
a Borel map $f$, namely $f_\sharp \mu(A):=\mu(f^{-1}(A))$, $\mu\res
A$ for the restriction operator, namely $\mu\res A(B)=\mu(A\cap B)$.

\subsection{Reminders on optimal transport}\label{sub2} We assume that the reader
is familiar with optimal transport, here we just recall the notation
we are going to use in this paper and some potentially less known
constructions. Standard references are
\cite{Ambrosio-Gigli11,Ambrosio-Gigli-Savare08,Villani09} and
occasionally we give precise references for the facts stated here.

Given a complete and separable space
 $(X,\sfd)$, $\probt X$ is the set of Borel
probability measures with finite second moment,  which we endow with
the Wasserstein distance $W_2$ defined by
\begin{equation}
\label{eq:w2} W_2^2(\mu,\nu):=\min\int\sfd^2(x,y)\,\d\ggamma(x,y),
\end{equation}
the minimum being taken among the collection $\adm\mu\nu$ of all
admissible plans (also called couplings) $\ggamma$ from $\mu$ to
$\nu$, i.e. all measures $\ggamma\in\prob{X\times X}$ such that
$\pi^1_\sharp\ggamma=\mu$, $\pi^2_\sharp\ggamma=\nu$. All the
minimizers of \eqref{eq:w2} are called optimal plans and their
collection (always non empty, since $\mu,\,\nu\in \probt X$) is
denoted by $\opt\mu\nu$. The metric space $(\probt X,W_2)$ is
complete and separable; it is also a length or a geodesic space if
and only if $X$ is, see for instance \cite[Theorem~2.10,
Remark~2.14]{Ambrosio-Gigli11}.

Given a reference measure $\mm$, we shall also use the notation
$$
\probt{X,\mm}:=\big\{\mu\in\probt{X}:\
\supp\mu\subset\supp\mm\big\}.
$$
The $c$-transform of a function $\psi:X\to\R\cup\{-\infty\}$,
relative to the cost $c=\tfrac{1}{2}\sfd^2$, is defined by
\[
\psi^c(x):=\inf_{y\in X}\frac{\sfd^2(x,y)}{2}-\psi(y).
\]
Notice that still $\psi^c$ takes its values in $\R\cup\{-\infty\}$,
unless $\psi\equiv -\infty$. A function
$\varphi:X\to\R\cup\{-\infty\}$ is said to be $c$-concave if
$\varphi=\psi^c$ for some $\psi:X\to\R\cup\{-\infty\}$. A set
$\Gamma\subset X\times X$ is $c$-cyclically monotone if
\[
\sum_{i=1}^n c(x_i,y_i)\leq\sum_{i=1}^n c(x_i,y_{\sigma(i)})
\qquad\forall n\geq 1,\,\,(x_i,y_i)\in\Gamma,\,\,\text{$\sigma$
permutation.}
\]
Given $\mu,\,\nu\in\probt X$ there exists a $c$-cyclically monotone
closed set $\Gamma$ containing the support of all optimal plans
$\ggamma$. In addition, there exists a (possibly non unique)
$c$-concave function $\varphi\in L^1(X,\mu)$ such that $\varphi^c\in
L^1(X,\nu)$ and $\varphi(x)+\varphi^c(y)=c(x,y)$ on $\Gamma$. Such
functions are called Kantorovich potentials. We remark that the
typical construction of $\varphi$ (see for instance
\cite[Theorem~6.1.4]{Ambrosio-Gigli-Savare08}) gives that $\varphi$
is locally Lipschitz in $X$ if the target measure $\nu$ has bounded
support. Conversely, it can be proved that $\ggamma\in\adm\mu\nu$
and $\supp\ggamma$ $c$-cyclically monotone imply that $\ggamma$ is
an optimal plan.

It is not hard to check that (see for instance
\cite[Proposition~3.9]{Ambrosio-Gigli-Savare11})
\begin{equation}\label{eq:propkant}
|\nabla^+\varphi|(x)\leq \sfd(x,y)\qquad\text{for $\ggamma$-a.e.
$(x,y)$,}
\end{equation}
for any optimal plan $\ggamma$ and Kantorovich potential $\varphi$
from $\mu$ to $\nu$.

If $\mu$ and $\nu$ are joined by a geodesic in $(\probt X,W_2)$, the
distance $W_2$ can be equivalently characterized by
\begin{equation}
\label{eq:w2din} W_2^2(\mu,\nu)=\min\int\int_0^1|\dot\gamma_t|^2\,\d
t\,\d\ppi(\gamma),
\end{equation}
among all measures $\ppi\in\prob{\CC{[0,1]}X}$ such that
$(\e_0)_\sharp \ppi=\mu$, $(\e_1)_\sharp\ppi=\nu$, where the
2-action $\int_0^1|\dot\gamma_t|^2\,\dt$ is taken by definition
$+\infty$ is $\gamma$ is not absolutely continuous. The set of
minimizing plans $\ppi$ in \eqref{eq:w2din} will be denoted by
$\gopt(\mu,\nu)$. It is not difficult to see that
$\ppi\in\gopt(\mu,\nu)$ if and only if
$\ggamma:=(\e_0,\e_1)_\sharp\ppi\in\prob{X\times X}$ is a minimizer
in \eqref{eq:w2} and $\ppi$ is concentrated on $\geo(X)$.
Furthermore, a curve $(\mu_t)$ is a constant speed geodesic from
$\mu_0$ to $\mu_1$ if and only if there exists
$\ppi\in\gopt(\mu_0,\mu_1)$ such that
\[
\mu_t=(\e_t)_\sharp\ppi\qquad\quad\forall t\in [0,1],
\]
see for instance \cite[Theorem~2.10]{Ambrosio-Gigli11} and notice
that the assumption that $(X,\sfd)$ is geodesic is never used in the
proof of $(i)\Leftrightarrow (ii)$.

The linearity of the transport problem immediately yields that the
squared Wasserstein distance $W_2^2(\cdot,\cdot)$ is jointly convex.
This fact easily implies that if $(\mu^1_t),\,(\mu^2_t)\subset\probt
X$ are two absolutely continuous curves, so is
$t\mapsto\mu_t:=(1-\lambda)\mu^1_t+\lambda\mu^2_t$ for any
$\lambda\in [0,1]$, with an explicit bound on its metric speed:
\begin{equation}
\label{eq:convw2}
|\dot\mu_t|^2\leq(1-\lambda)|\dot\mu^1_t|^2+\lambda|\dot\mu^2_t|^2\qquad\text{for
a.e. $t\in [0,1]$.}
\end{equation}

Finally, we recall the definition of \emph{push forward via a plan}, introduced in \cite{Sturm06I}
(with a different notation) and further studied in \cite{Gigli10}, \cite{Ambrosio-Gigli-Savare11}.

\begin{definition}[Push forward via a plan]\label{def:pushplan}
Let $\ggamma\in\prob{X\times Y}$. For $\mu\in\prob X$ such that
$\mu=\rho\,\big(\pi^X_\sharp\ggamma\big)\ll\pi^X_\sharp\ggamma$, the
push forward $\ggamma_\sharp\mu\in\prob Y$ of $\mu$ via $\ggamma$ is
defined by
$$\ggamma_\sharp\mu:=\pi^Y_\sharp\big((\rho\circ\pi^X)\ggamma\big).$$
\end{definition}

An equivalent representation of $\ggamma_\sharp\mu$ is
\begin{equation}\label{eq:equivalentpushplan}
\ggamma_\sharp\mu=\eta\,\pi^Y_\sharp\ggamma\qquad\text{where}\qquad
\eta(y):=\int_X\rho(x)\,\d\ggamma_y(x)
\end{equation}
and $\{\ggamma_y\}_{y\in Y}\subset\Probabilities{X}$ is the
disintegration of $\ggamma$ w.r.t. the projection on $Y$.

Defining $\ggamma^{-1}:=(\pi^Y,\pi^X)_\sharp\ggamma\in\prob{Y\times
X}$, we can define in a symmetric way the map
$\nu\mapsto\ggamma^{-1}_\sharp\nu\in\prob{X}$ for any
$\nu\ll\pi^Y_\sharp\ggamma^{-1}=\pi^Y_\sharp\ggamma$.

Notice that if $\ggamma$ is concentrated on the graph of a map
$T:X\to Y$, it holds $\ggamma_\sharp\mu=T_\sharp\mu$ for any
$\mu\ll\pi^X_\sharp\ggamma$, and that typically
$\ggamma^{-1}_\sharp(\ggamma_\sharp\mu)\neq\mu$. We collect in the
following proposition the basic properties of $\ggamma_\sharp$ in
connection with the Wasserstein distance.

\begin{proposition}\label{prop:basegammasharp}
The following properties hold:
\begin{itemize}
\item[(i)] $\mu\leq C\,\pi^1_\sharp\ggamma$ for some $C>0$ implies $\ggamma_\sharp\mu\leq C\,\pi^2_\sharp\ggamma$.
\item[(ii)] Let
  $\mu,\,\nu\in\probt X$ and $\ggamma\in\opt\mu\nu$. Then for every
$\tilde\mu\in\probt X$ such that $\tilde\mu\ll\mu$ it holds
\begin{equation}
\label{eq:optpush} W_2^2(\tilde\mu,\ggamma_\sharp\tilde\mu)=\int
\sfd^2(x,y)\frac{\d \tilde\mu}{\d\mu}(x)\,\d\ggamma(x,y)
\end{equation}
and, in particular, $\ggamma_\sharp\tilde\mu\in\probt{Y}$ and
$\frac{\d\tilde\mu}{\d\mu}\circ\pi^1\ggamma\in\opt{\tilde\mu}{\ggamma_\sharp\tilde\mu}$
if any of the two terms is finite.
\item[(iii)] Let $\ggamma\in\probt{X\times Y}$, $C>0$ and
$A_C:=\bigl\{\mu\in\probt X:\ \mu\leq C\,\pi^1_\sharp\ggamma\bigr\}$.
Then
\begin{equation}
\label{eq:uniforme} \mu\mapsto\ggamma_\sharp\mu \quad\text{is
uniformly continuous in $A_C$ w.r.t. the $W_2$ distances.}
\end{equation}
\item[(iv)] Let $\ggamma\in\probt{X\times X}$ and $\mu\leq C\,\pi^1_\sharp\ggamma$ for some constant $C$.
Then
\begin{equation}
\label{eq:boundlimitato}
W_2^2(\mu,\ggamma_\sharp\mu)\leq C\int\sfd^2(x,y)\,\d\ggamma(x,y).
\end{equation}
\end{itemize}
\end{proposition}
\begin{proof} (i) is obvious. \\* (ii)
Since $\ggamma$ is optimal, $\supp\big(\tfrac{\d
\tilde\mu}{\d\mu}\circ\pi^1\ggamma\big)\subset\supp\ggamma$ is
$c$-cyclically monotone. Moreover $\tfrac{\d
\tilde\mu}{\d\mu}\circ\pi^1\ggamma$ is an admissible plan from
$\tilde\mu$ to $\ggamma_\sharp\tilde\mu$, with cost equal to the
right hand side of \eqref{eq:optpush}. Hence, if the cost is finite,
from the finiteness of $W_2(\tilde\mu,\ggamma_\sharp\tilde\mu)$ we
infer that $\ggamma_\sharp\tilde\mu\in\probt{X}$, hence $c$-cyclical
monotonicity implies optimality and equality in \eqref{eq:optpush}.
The same argument works if we assume that
$W_2(\tilde\mu,\ggamma_\sharp\tilde\mu)$ is finite.\\*  (iii) Since
the singleton $\{\pi^1_\sharp\ggamma\}$ is both tight and
2-uniformly integrable, the same is true for the set $A_C$, which,
being $W_2$-closed, is compact (see
\cite[Section~5.1]{Ambrosio-Gigli-Savare08} for the relevant
definitions and simple proofs). Hence it is sufficient to prove the
continuity of the map. Let $(\mu_n)\subset A_C$ be $W_2$-converging
to $\mu\in A_C$ and let $\rho_n$, $\rho$ be the respective densities
w.r.t. $\pi^1_\sharp\ggamma$. Since $(\mu_n)$ converges to $\mu$ in
duality with $\Cb(X)$ and since the densities are equibounded, we
get that $\rho_n$ converge to $\rho$ weakly$^*$ in
$L^\infty(X,\pi^1_\sharp\ggamma)$. By $(i)$ and the same argument
just used we know that $(\ggamma_\sharp\mu_n)\subset \probt Y$ is
relatively compact w.r.t. the Wasserstein topology, hence to
conclude it is sufficient to show that $(\ggamma_\sharp\mu_n)$
converges to $\ggamma_\sharp\mu$ in duality with $\Cb(Y)$. To this
aim, fix $\varphi\in \Cb(Y)$ and notice that it holds
\[
\begin{split}
\int_Y\varphi(y)\,\d\ggamma_\sharp\mu_n(y)=
\int_{X\times Y}\varphi(y)\rho_n(x)\,\d\ggamma(x,y)=
\int_X\left(\int_Y\varphi(y)\,\d\ggamma_x(y)\right)\rho_n(x)\,\d\pi^1_\sharp\ggamma(x),
\end{split}
\]
where $\{\ggamma_x\}$ is the disintegration of $\ggamma$ w.r.t. the
projection on the first component. Since $\varphi$ is bounded, so is
the map $x\mapsto\int\varphi\,\d\ggamma_x$, and the claim
follows.\\* \noindent (iv) Just notice that
$\frac{\d\mu}{\d\pi^1_\sharp\sggamma}\circ\pi^1\ggamma\in\adm\mu{\ggamma_\sharp\mu}$.
\end{proof}

The operation of push forward via a plan has also interesting
properties in connection with the relative entropy functional
$\entv$. We recall that, given $\mm\in\prob X$, the functional
$\entv:\prob X\to[0,\infty]$ is defined by
\[
\entr\mu\mm:=\left\{
\begin{array}{ll}
\displaystyle{\int\frac{\d\mu}{\d\mm}\log\left(\frac{\d\mu}{\d\mm}\right)
\,\d\mm}&\qquad\textrm{ if }\mu\ll\mm,\\
&\\
+\infty&\qquad\text{ otherwise}.
\end{array}
\right.
\]

\begin{proposition}\label{prop:basegammapf}
For all $\ggamma\in\prob{X\times Y}$ the following properties hold:
\begin{itemize}
\item[(i)] For any $\mm,\mu\ll\pi^X_\sharp\ggamma$ it holds
$\entr{\ggamma_\sharp\mu}{\sggamma_\sharp\mm}\leq \entr\mu\mm$.
\item[(ii)] For any $\mm\ll\pi^X_\sharp\ggamma$, $C>0$, the map $\mu\mapsto\entr\mu\mm-
\entr{\ggamma_\sharp\mu}{\sggamma_\sharp\mm}$ is convex in
$\left\{\mu\in\Probabilities{X}:\ \mu\leq C\mm\right\}$.
\end{itemize}
\end{proposition}
\begin{proof} (i) We follow \cite[Lemma~7.4]{Ambrosio-Gigli11} and
\cite[Lemma~4.19]{Sturm06I}. We can assume $\mu\ll\mm$, otherwise
there is nothing to prove. Then it is immediate to check from the
definition that $\ggamma_\sharp\mu\ll\ggamma_\sharp\mm$. Let
$\mu=\rho\mm$, $\mm=\theta\,\pi_\sharp^X\ggamma$,
$\ggamma_\sharp\mu=\eta\,\ggamma_\sharp\mm$, and $e(z):=z\log z$. By
disintegrating $\ggamma$ as in \eqref{eq:equivalentpushplan}, we
have that
$$
\eta(y)=\int {\rho(x)}\,\d\tilde\ggamma_y(x),\qquad \tilde\ggamma_y=
\biggl(\int {\theta(x)}\,\d\ggamma_y(x)\biggr)^{-1}
{\theta}\,\ggamma_y.
$$
Using the convexity of $e$ and Jensen's inequality with the
probability measures $\tilde\ggamma_y$ we get
$$
e(\eta(y))\leq\int e(\rho(x))\,\d\tilde\ggamma_y(x).
$$
Now, since $\tilde\ggamma_y$ are the conditional probability
measures of
$\tilde\ggamma:=[\mm/\pi^X_\sharp\gamma]\circ\pi^X\ggamma=
\theta\circ\pi^X\ggamma$, whose first marginal is $\mm$, by
integration of both sides with respect to the second marginal of
$\tilde\ggamma$, namely $\ggamma_\sharp\mm$, we get
\begin{eqnarray*}
\entr{\ggamma_\sharp\mu}{\sggamma_\sharp\mm}&=&\int
e(\eta(y))\,\d\ggamma_\sharp\mm(y)\leq \int\int
e(\rho(x))\,\d\tilde\ggamma_y(x)\d\ggamma_\sharp\mm(y)\\
&=& \int\int e(\rho(x))\,\d\tilde\ggamma(x,y)=\int e(\rho)\,\d\mm.
\end{eqnarray*}
%
% {\color{red} (i) We follow \cite[Lemma~7.4]{Ambrosio-Gigli11} and
% \cite[Lemma~4.19]{Sturm06I}. We can assume $\mu\ll\mm$, otherwise
% there is nothing to prove. Then it is immediate to check from the
% definition that $\ggamma_\sharp\mu\ll\ggamma_\sharp\mm$. Let
% $\mu=\rho\mm$, $\ggamma_\sharp\mu=\eta\ggamma_\sharp\mm$, and
% $u(z):=z\log z$. By disintegrating $\ggamma$ as in
% \eqref{eq:equivalentpushplan}, we have
% \[
% \begin{split}
% \entr{\ggamma_\sharp\mu}{\sggamma_\sharp\mm}&=
% \int u(\eta(y))\,\d\ggamma_\sharp\mm(y)=\int u\left(\int \rho(x)\,\d\ggamma_y(x)\right)\,\d\ggamma_\sharp\mm(y)\\
% &\leq \int\int
% u(\rho(x))\,\d\ggamma_y(x)\,\d\ggamma_\sharp\mm(y)=\int u(\rho(x))
% \frac{\d\mm}{\d\pi^1_\sharp\ggamma}(x)\,\d\ggamma(x,y)\\&=\int
% u(\rho(x))\,\d\mm(x)=\entr{\mu}{\mm}.
% \end{split}
% \]}
\noindent (ii) This is proved in
\cite[Lemma~7.7]{Ambrosio-Gigli-Savare11} (see also
\cite[Proposition~11]{Gigli10}). Notice that in
\cite{Ambrosio-Gigli-Savare11} we worked under the assumption $X=Y$,
but this makes no difference,
since the as one can work on the disjoint union
$X\sqcup Y$ endowed with a distance which extends those of $X,\,Y$.
\end{proof}

\begin{remark}{\rm
We remark that the property $(i)$ above is true for any internal
energy kind functional: as the proof shows, under the same
assumptions on $\mm,\,\mu,\,\ggamma$ it holds
\[
U_{\sggamma_\sharp\mm}(\ggamma_\sharp\mu)\leq U_\mm(\mu),
\]
where $U_\mm(\mu)$ is given by $\int
u(\rho)\,\d\mm+u'(\infty)\mu^s(X)$ for some convex continuous
function $U:[0,\infty)\to\R\cup\{+\infty\}$ and $\mu=\rho\mm+\mu^s$,
with $\mu^s\perp\mm$.

On the other hand, part $(ii)$ does \emph{not} always hold for these
functionals: in \cite{Gigli10} it has been shown that for
$U(z):=\frac{z^\alpha}{\alpha-1}$ one has that
\[
\mu\quad\mapsto\quad U_\mm(\mu)-U_{\sggamma_\sharp\mm}(\ggamma_\sharp\mu)
\]
is convex on $\left\{\mu\in\Probabilities{X}:\ \mu\leq C\mm\right\}$
for any $C>0$ if and only if $1<\alpha\leq 2$. In particular, convexity does
not hold for the functionals appearing in the definition of $CD(K,N)$ bounds.}\fr
\end{remark}

\subsection{Metric measure spaces and Sturm's distance $\D$}\label{sub3}

Throughout this paper we will always consider normalized metric
measure spaces with finite variance, according to \cite[\S
3.1]{Sturm06I}: in short, we will denote by $\X$ the set of
(isomorphism classes of) metric measure spaces that we will
consider, namely
\begin{equation}
  \label{eq:1}
\X:=\Big\{(X,\sfd,\mm)\ :\ (X,\sfd)\textrm{ is complete,
  separable, and }\mm\in\probt X\Big\}.
\end{equation}
We say that the two metric measure spaces $(X,\sfd_X,\mm_X)$ and
$(Y,\sfd_Y,\mm_Y)$ are \emph{isomorphic} if there exists a
\begin{equation}\label{eq:isosturm}
\text{bijective isometry\quad $f:\supp\mm_X\to\supp\mm_Y$\quad such
that}\quad
 f_\sharp\mm_X=\mm_Y.
\end{equation}
We say that $(X,\sfd,\mm)$ is \emph{length} or \emph{geodesic} if
$(\supp\mm,\sfd)$ is so, and these notions are invariant in the
isomorphism class.

Notice that $(X,\sfd,\mm)$ is always isomorphic to
$(\supp\mm,\sfd,\mm)$, so that it will often be not restrictive to
assume the non-degeneracy condition $\supp\mm=X$.

In this section we recall the definition of the distance $\D$
between metric measure spaces, introduced by Sturm in
\cite{Sturm06I}, and its basic properties.

\begin{definition}[Coupling between metric measure spaces]
Given two metric measure spaces $(X,\sfd_X,\mm_X)$, $(Y,\sfd_Y,\mm_Y)$, we consider the product space
$(X\times Y,\sfd_{XY})$, where $\sfd_{XY}$ is the distance defined by
\[
\sfd_{XY}\big((x_1,y_1),(x_2,y_2)\big):=\sqrt{\sfd_X^2(x_1,x_2)+\sfd_Y^2(y_1,y_2)},
\]
We say that a pair $({\bf d},\ggamma)$ is an admissible coupling
between $(X,\sfd_X,\mm_X),\,(Y,\sfd_Y,\mm_Y)$, and we write $({\bf
d},\ggamma)\in\adm{(\sfd_X,\mm_X)}{(\sfd_Y,\mm_Y)}$, if:
\begin{itemize}
\item[(a)] ${\bf d}$ is a pseudo distance on $X\sqcup Y$ (i.e. points at 0
${\bf d}$-distance are not necessarily equal) which coincides with
$\sfd_X$ (resp. $\sfd_Y$) when restricted to
$\supp\mm_X\times\supp\mm_X$ (resp. $\supp\mm_Y\times\supp\mm_Y$).
\item[(b)] $\ggamma$ is a Borel measure
on $X\times Y$ such that $\pi^X_\sharp\ggamma=\mm_X$ and
$\pi^Y_\sharp\ggamma=\mm_Y$.
\end{itemize}
\end{definition}
It is not hard to see that the set of admissible couplings is always
non empty. Notice that the restriction of $\bf d$ to $X\times Y$ is
Lipschitz continuous and therefore Borel (with respect to the
product topology), as a simple application of the triangle
inequality.

The \emph{cost} $C({\bf d},\ggamma)$ of a coupling is given by
\[
C({\bf d},\ggamma):=\int_{X\times Y} {\bf d}^2(x,y)\,\d\ggamma(x,y).
\]
In analogy with the definition of $W_2$, the distance
$\D\big((X,\sfd_X,\mm_X),(Y,\sfd_Y,\mm_Y)\big)$ is then defined as
\begin{equation}
\label{eq:defD}
\D^2\bigl((X,\sfd_X,\mm_X),(Y,\sfd_Y,\mm_Y)\bigr):=\inf C({\bf
d},\ggamma),
\end{equation}
the infimum being taken among all couplings $({\bf d},\ggamma)$ of
$(X,\sfd_X,\mm_X)$ and $(Y,\sfd_Y,\mm_Y)$.
Since one can use the isometries in \eqref{eq:isosturm} to transfer
couplings between two spaces to couplings between isomorphic spaces,
a trivial consequence of the definition is that $\D$ actually
depends only on the isomorphism class. In the next proposition we
collect the main properties of $\D$, see
\cite[Section~3.1]{Sturm06I}.

\begin{proposition}[Properties of $\D$]
The infimum in \eqref{eq:defD} is attained and a minimizing coupling
will be called optimal. Also, $\D$ is a distance on $\X$, and in
particular $\D$ vanishes only on pairs of isomorphic metric measure
spaces.\\
Finally, $(\X,\D)$ is a complete, separable and length metric space.
\end{proposition}

Also, it can be shown \cite[Lemma~3.7]{Sturm06I} that $\D$ metrizes
the measured Gromov-Hausdorff convergence, when restricted to
compact metric spaces with controlled diameter. We also remark that,
in line with what happens with the Gromov-Hausdorff distance, a
$\D$-convergent sequence of metric measure spaces can be embedded
into a common metric space: in this case the possibility to work in
spaces where $\supp\mm$ is not equal to the whole space $X$ turns
out to be useful.
\begin{proposition}
Let $(X_n,\sfd_n,\mm_n)\in\X$, $n\in\N$, and $(X,\sfd,\mm)\in\X$.
Then the following two properties are equivalent.
\begin{itemize}
\item[(i)] $(X_n,\sfd_n,\mm_n)\stackrel{\D}\to(X,\sfd,\mm)$ as $n\to\infty$.
\item[(ii)] There exist a complete and separable metric
space $(Y,\sfd_Y)$ and isometries $f_n:\supp\mm_n\to Y$, $n\in\N$,
$f:\supp\mm\to Y$, such that $W_2((f_n)_\sharp\mm_n,f_\sharp\mm)\to
0$ as $n\to\infty$.
\end{itemize}
\end{proposition}
\begin{proof} $(i)\Rightarrow (ii)$. Let $({\bf d}_n,\ggamma_n)$ be optimal
couplings for $(X,\sfd,\mm)$, $(X_n,\sfd_n,\mm_n)$, $n\in\N$. Define
$Y:=\Big(\bigsqcup_i\supp\mm_i\Big)\sqcup \supp\mm$ and the
pseudo distance $\sfd_Y$ on $Y$ by
\[
\sfd_Y(x,x'):=\left\{\begin{array}{ll}
{\bf d}_n(x,x') ,&\qquad\textrm{ if }x,x'\in \supp\mm\sqcup \supp\mm_n,\\
\inf\limits_{x''\in X} {\bf d}_n(x,x'')+{\bf d}_m(x'',x')&
\qquad\textrm{ if }x\in \supp\mm_n,\ x'\in \supp\mm_m.
\end{array}
\right.
\]
By construction the quotient metric space $(Y,\sfd_Y)$ induced by
the equivalence relation $x\sim y\Leftrightarrow \sfd_Y(x,y)=0$ is
separable. Possibly replacing it by its abstract completion we can
also assume that it is complete. Denoting by $f_n,\,f$ the isometric
embeddings of $X_n,\,X$ into $Y$,
\[
\D\big((X,\sfd,\mm),(X_n,\sfd_n,\mm_n)\big)= \sqrt{\int_{X\sqcup
X_n}{\bf d}_n^2(x,y)\,\d\ggamma_n(x,y)}\geq
W_2((f_n)_\sharp\mm_n,f_\sharp\mm),
\]
so the conclusion follows.\\* $(ii)\Rightarrow (i)$.
Straightforward.
\end{proof}

\subsection{Calculus and heat flow in metric measure
spaces}\label{sub4}

\subsubsection{Upper gradients}

Recall that the slope of a Lipschitz function $\psi$ is an upper
gradient, namely $|\psi(\gamma_1)-\psi(\gamma_0)|$ can be bounded
from above by $\int_\gamma|\nabla\psi|$ for any absolutely
continuous curve $\gamma:[0,1]\to X$.
\begin{lemma}
  \label{le:sug} Let $(X,\sfd,\mm)\in\X$ and $\psi:X\to\R$ Lipschitz. For all
  $\mu_t\in \AC2{[0,1]}{\probt{X}}$ it holds
  \begin{equation}
      \label{eq:3}
        \left|\int_X\psi\,\d\mu_1-\int_X\psi\,\d\mu_0\right|\le
        \int_0^1\Big(\int_X
        |\nabla\psi|^2\,\d\mu_t\Big)^{1/2}|\dot\mu_t|\,\dt.
    \end{equation}
  \end{lemma}
  \begin{proof}
    %Taking Lemma~\ref{le:approxGiuseppe} into account we can assume
    %with no loss of generality that $\mu_t\leq C\mm$ for all $t\in
    %[0,1]$, for some constant $C$.
    Applying \cite{Lisini07}, we can find a probability
    measure $\ppi$ in $\CC{[0,1]}X$ concentrated on
    $\AC2{[0,1]}X =0$ and satisfying
    \begin{equation}
      \label{eq:4}
      \mu_t=(\e_t)_\sharp\ppi\quad\text{for all $t\in [0,1]$},\qquad
      |\dot \mu_t|^2=\int |\dot\gamma_t|^2\,\d\ppi(\gamma)\quad
      \text{for a.e.\ }t\in (0,1).
    \end{equation}
    %Since $\mu_t\leq C\mm$ we obtain that $\ppi$ is a test plan,
    %hence $\psi$ is Sobolev along $\ppi$-a.e.~curve.
    By the upper gradient property of $|\nabla\psi|$ we get
    \begin{eqnarray*}
    \biggl|\int_X\psi\,\d\mu_1-\int_X\psi\,\d\mu_0\biggr|
    &=&
    \biggl|\int(\psi\circ\e_1-\psi\circ\e_0))\,\d\ppi\biggr|
    \le\int
    \Big(\int_0^1|\nabla\psi|(\gamma_t)|\dot\gamma_t|\,\d t\Big)\,\d\ppi(\gamma)
    \\&=&
    \int_0^1\Big(\int|\nabla\psi|(\gamma_t)|\dot\gamma_t|\,\d\ppi(\gamma)\Big) \,\d t
    \\&\le&
    \int_0^1\Big(\int|\nabla\psi|^2(\gamma_t)\,\d\ppi(\gamma)\Big)^{1/2}
    \Big(\int|\dot\gamma_t|^2\,\d\ppi(\gamma)\Big)^{1/2}\,\d t\\
    &=&\int_0^1 \Big(\int|\nabla\psi|^2\,\d\mu_t\Big)^{1/2}
    |\dot\mu_t|\,\d t.
  \end{eqnarray*}
  \end{proof}

\subsubsection{Weak upper gradients and gradient flow of Cheeger's energy}

Here we recall the definition and basic properties of weak upper
gradients of real functions in the metric measure space
$(X,\sfd,\mm)$. All the concepts and statements that we consider
here have been introduced and proven in
\cite{Ambrosio-Gigli-Savare11}, see \S~5. In particular, here  we
shall consider measures concentrated in $\AC2{[0,1]}X$ (see
\S~\ref{sub1}).

\begin{definition}[Test plans and negligible collection of
  curves]
\label{def:testplan}
We say that\\
 $\ppi\in \prob{\AC2{[0,1]}X}$ is a test plan
 with bounded compression if there exists $C=C(\ppi)>0$ such
that
\[
(\e_t)_\sharp\ppi\leq C\,\mm\qquad\text{for every } t\in[0,1].
\]
We will denote by $\calT$ the collection of all the test plans with
bounded compression. We say that a Borel set $A\subset \AC2{[0,1]}X
$ is $\calT$-\emph{negligible} if $\ppi(A)=0$ for any test plan
$\ppi\in\calT$.
\end{definition}
Since we will always deal with test plans in $\calT$, we will often
omit to mention exsplicitly $\calT$ and the words ``bounded
compression'', and we will refer to them simply as \emph{test
plans.}

A property which holds for every curve of $\AC2{[0,1]}X $, except
possibly for a subset of a negligible set, is said to hold for
almost every curve.

\begin{definition}[Functions which are Sobolev along almost all curves]
We say that $f:X\to\overline{\R}$ is Sobolev along almost all curves
if, for a.e.~curve $\gamma$, $f\circ\gamma$ coincides a.e.~in
$[0,1]$ and in $\{0,1\}$ with an absolutely continuous map
$f_\gamma:[0,1]\to\R$.
\end{definition}

Notice that the choice of the trivial test plan $\ppi:=\iota_\sharp
\mm$, where $\iota:X\to \AC 2{[0,1]}X$ maps any point $x\in X$ to
the constant curve $\gamma\equiv x$, yields that any Sobolev
function along almost all curves is finite $\mm$-a.e.~in $X$. In
this class of functions we can define the notion of weak upper
gradient and of minimal weak upper gradient.

\begin{definition}[Weak upper gradients]\label{def:weak_upper_gradient}
Given $f:X\to\overline{\R}$ Sobolev along a.e.~curve, a
$\mm$-measurable function $G:X\to[0,\infty]$ is a weak upper
gradient of $f$ if
\begin{equation}
\label{eq:inweak} \biggl|\int_{\partial\gamma}f\biggr|\leq
\int_\gamma G\qquad\text{for a.e.~curve }\gamma.
\end{equation}
Here and in the following, we write $\int_{\partial\gamma}f$ for
$f(\gamma_1)-f(\gamma_0)$ and $\int_\gamma G$ for $\int_0^1
G(\gamma_t)|\dot\gamma_t|\d t$.
\end{definition}

It turns out (see \cite[Proposition~5.7,
Definition~5.9]{Ambrosio-Gigli-Savare11}) that if and $G_1,\,G_2$
are weak upper gradients of $f$, then so is $\min\{G_1,G_2\}$. It
follows that there exists a $\mm$-measurable function $\weakgrad
f:X\to[0,\infty]$ weak upper gradient having the property that
\[
\weakgrad f\leq G\qquad\text{$\mm$-a.e.~in $X$}
\]
for any other weak upper gradient $G$. Because of this $\mm$-a.e.
minimality property, the function $\weakgrad f$ will be called the
\emph{minimal weak upper gradient} of $f$. Also, the property of
being Sobolev along a.e.~curve and the minimal weak upper gradient
are invariant under modifications of $f$ in $\mm$-negligible sets
(\cite[Proposition~5.8]{Ambrosio-Gigli-Savare11}). In addition, the
minimal weak gradient is local in the following sense: if both
$f,\,g$ are Sobolev along a.e.~curve then it holds
\begin{equation}
\label{eq:weaklocal} \weakgrad f=\weakgrad g\qquad\text{$\mm$-a.e.
on the set $\{f=g\}$.}
\end{equation}

Other useful and natural properties are: the restriction inequality
\cite[Remark~5.6]{Ambrosio-Gigli-Savare11}
\begin{equation}
\label{eq:restrizione}
|f(\gamma_s)-f(\gamma_s)|\leq\int_t^s\weakgrad
f(\gamma_r)|\dot\gamma_r|\,dr\quad\text{for a.e.\ $\gamma$, for all
$[s,t]\subset [0,1]$}
\end{equation}
the chain rule \cite[Proposition~5.14(b)]{Ambrosio-Gigli-Savare11})
\begin{equation}
\label{eq:chainrule}
\begin{aligned}
\weakgrad{(\phi\circ f)}&=\phi'\circ f\weakgrad
f&\quad&\text{$\mm$-a.e.~in $X$, if $\phi$ is Lipschitz and nondecreasing,}\\
\weakgrad{(\phi\circ f)}&\leq |\phi'\circ f|\weakgrad
f&\quad&\text{$\mm$-a.e.~in $X$, if $\phi$ is Lipschitz,}
\end{aligned}
\end{equation}
and the weak Leibnitz rule
\begin{equation} \label{eq:weakleib}
\weakgrad{(fg)}\leq|f|\weakgrad g+|g|\weakgrad
f\qquad\text{$\mm$-a.e.~in $X$.}
\end{equation}

The \emph{Cheeger energy} is the functional defined in the class of
Borel functions $f:X\to\overline{\R}$ by
\[
\C(f):=\left\{
\begin{array}{ll}
\displaystyle{\frac12\int\weakgrad f^2\,\d\mm}&\qquad\text{if }f\text{ is Sobolev along a.e.~curve},\\
&\\
+\infty&\qquad\text{otherwise}.
\end{array}
\right.
\]
Using the stability properties of weak upper gradients under weak
convergence (\cite[Theorem~5.12]{Ambrosio-Gigli-Savare11}) it can be
proved that $\C$ is convex and lower semicontinuous w.r.t.
convergence in $\mm$-measure (in particular w.r.t. $\mm$-a.e.
convergence). For the domain of $\C$ in $L^2(X,\mm)$ we shall also
use the traditional notation $W^{1,2}(X,\sfd,\mm)$, see \cite[Remark
4.7]{Ambrosio-Gigli-Savare11}: it is a Banach space when endowed
with the norm $\|f\|_{W^{1,2}}^2:=\|f\|_2^2+2\C(f)$. A nontrivial
approximation theorem (see
\cite[Theorem~6.2]{Ambrosio-Gigli-Savare11}) shows that
\begin{equation}\label{eq:nontrivialatall}
\C(f)=\frac 12\inf\left\{\liminf_{h\to\infty}\int |\nabla
f_h|^2\,\d\mm:\ f_h\in {\rm Lip}(X),\,\,\|f_h-f\|_2\to
0\right\}\qquad\forall f\in L^2(X,\mm),
\end{equation}
where $|\nabla f|$ is the local Lipschitz constant of $f$ defined in
\eqref{eq:loclip}.

Given $f\in W^{1,2}(X,\sfd,\mm)$, we write $\partial^-\C(f)\subset
L^2(X,\mm)$ for the subdifferential at $f$ of the restriction to
$L^2(X,\mm)$ of Cheeger's energy, namely $\xi\in\partial^-\C(f)$ iff
$$
\C(g)\geq\C(f)+\int_X\xi(g-f)\,\d\mm\qquad\forall g\in L^2(X,\mm).
$$
We say that $f\in L^2(X,\mm)$ is in the domain of the
$(\sfd,\mm)$-Laplacian, and write $f\in D(\Deltam)$ (in
\cite{Ambrosio-Gigli-Savare11} we used the notation
$\Delta_{\sfd,\mm}$ to emphasize the dependence on the metric
measure structure), if $\partial^-\C(f)\neq\emptyset$. In this case
we define $\Deltam f\in L^2(X,\mm)$ by $\Deltam f:=-v$, where $v$ is
the element of minimal $L^2(X,\mm)$ norm in $\partial^-\C(f)$.

We remark that in this generality $\C$ is not necessarily a
quadratic form, which is the same as to say that its restriction to
$L^2(X,\mm)$ is not a Dirichlet form. This means that the Laplacian,
though 1-homogeneous \cite[Remark~4.14]{Ambrosio-Gigli-Savare11}, is
not necessarily linear.

For the Laplacian we just defined the following rough integration by
parts formula holds:
\begin{equation}
\label{eq:partilapl} \left|\int g\Deltam
f\,\d\mm\right|\leq\int\weakgrad g\weakgrad f\,\d\mm,
\end{equation}
for all $f,\,g\in L^2(X,\mm)$ with  $f\in D(\Deltam)$ and $g\in
D(\C)$, see \cite[Proposition~4.15]{Ambrosio-Gigli-Savare11}.

The following result is a consequence of the by now classical theory
of gradient flows of convex lower semicontinuous functionals on
Hilbert spaces.

\begin{theorem}[Gradient flow of $\C$ in $L^2(X,\mm)$]\label{thm:gfc}
For all $f\in L^2(X,\mm)$ there exists a unique locally absolutely
continuous curve $(0,\infty)\ni t\mapsto f_t\in L^2(X,\mm)$ such
that $f_t\to f$ in $L^2(X,\mm)$ as $t\downarrow 0$ and
\[
\frac{\d}{\dt}f_t\in -\partial^-\C(f_t)\qquad\text{for a.e.~$t>0$,}
\]
the derivative being understood in $L^2(X,\mm)$. This curve is also
locally Lipschitz, it satisfies $f_t\in D(\Deltam)$ for any $t>0$
and
\[
\frac{\d^+}{\dt}f_t=\Deltam f_t\qquad\forall t>0.
\]
Finally, $t\mapsto \C(f_t)$ is locally Lipschitz in $(0,\infty)$,
infinitesimal at $\infty$ and, if $f\in D(\C)$, continuous in $0$.
Its right derivative is given by $-\|\Deltam f_t\|^2_2$ for every
$t>0$.
\end{theorem}

Finally, we recall a property of the minimal weak gradient of
Kantorovich potentials \cite[Lemma~10.1]{Ambrosio-Gigli-Savare11}:
\begin{proposition}\label{prop:potkant}
Let $(X,\sfd,\mm)\in\X$, $\mu,\,\nu\in\probt X$ with $\mu\geq c\mm$
for some $c>0$ and let $\varphi$ be a Kantorovich potential relative
to $(\mu,\nu)$. Then $\varphi$ is finite and absolutely continuous
(in particular, Sobolev) along a.e.~curve and
\[
\weakgrad\varphi\leq|\nabla^+\varphi|\qquad\text{$\mm$-a.e.~in $X$.}
\]
\end{proposition}

As a consequence of the previous proposition, since
\eqref{eq:propkant} yields $|\nabla^+\varphi|\in L^2(X,\mu)$, the
lower bound on $\mu$ yields $\weakgrad\varphi\in L^2(X,\mm)$.

\subsubsection{Convex functionals: gradient flows, entropy, and the
$CD(K,\infty)$ condition}

Let $(Y,\sfd_Y)$ be a complete and separable metric space,
$E:Y\to\R\cup\{+\infty\}$, and $K\in\R$. We say that $E$ is
$K$-geodesically convex if for any $y_0,\,y_1\in D(E)$ there exists
$\gamma\in\geo(Y)$ satisfying $\gamma_0=y_0$, $\gamma_1=y_1$ and
\[
E(\gamma_t)\leq (1-t)E(y_0)+tE(y_1)-\frac K2t(1-t)
\sfd_Y^2(y_0,y_1)\qquad\text{for every } t\in[0,1].
\]
Notice that if $E$ is $K$-geodesically convex, then $D(E)$ is
geodesic in $Y$ and therefore $\overline{D(E)}$ is a length space.

A consequence of $K$-geodesic convexity is that the descending slope
$|\nabla^- E|$ can be calculated at all $y\in D(E)$ as
\begin{equation}
\label{eq:slopesup} |\nabla^-E|(y)=\sup_{z\in
D(E)\setminus\{y\}}\left(\frac{E(y)-E(z)}{\sfd_Y (y,z)}+\frac K2
\sfd_Y(y,z)\right)^+.
\end{equation}
We recall (see \cite[Corollary~2.4.10]{Ambrosio-Gigli-Savare08})
that for $K$-geodesically convex and l.s.c.~functionals
the descending slope is an upper gradient, in particular
the property we shall need is
\begin{equation}
\label{eq:boundtuttecurve} E(y_s)\le E(y_t)+\int_s^t |\dot
y_r|\,|\nabla^-E|(y_r)\,\d r \qquad\text{for every }s,\,t\in
[0,\infty),\ s<t,
\end{equation}
for \emph{all} locally absolutely continuous curves $y:[0,\infty)\to
D(E)$. A metric gradient flow for the $K$-geodesically convex
functional $E$ is a locally absolutely continuous curve
$y:[0,\infty)\to D(E)$ along which \eqref{eq:boundtuttecurve} holds
as an equality and moreover $|\dot y_t|=|\nabla^- E|(y_t)$ for a.e.
$t>0$, so that the energy dissipation rate $\tfrac{\d}{\dt}E(y_t)$
is equal to $-|\dot y_t|^2=-|\nabla^- E|^2(y_t)$ for a.e.~$t>0$.

An application of Young inequality shows that metric gradient flows
for $K$-geodesically convex and l.s.c. functionals can equivalently
be defined as follows.

\begin{definition}[Metric  formulation of gradient flow]\label{def:dissKconv}
Let $E:Y\to\R\cup\{+\infty\}$ be a $K$-geodesically convex and
l.s.c. functional. We say that a locally absolutely continuous curve
$[0,\infty)\ni t\mapsto y_t\in D(E)$ is a gradient flow of $E$
starting from $y_0\in D(E)$ if
\begin{equation}\label{eq:ede}
E(y_0)= E(y_t)+\int_0^t\frac12 |\dot y_r|^2+\frac12|\nabla^-
E|^2(y_r)\,\d r\qquad\forall t\geq 0.
\end{equation}
\end{definition}
We now recall the definition of metric measure space with Ricci
curvature bounded from below by $K\in\R$, following \cite[\S
4.2]{Sturm06I} and \cite[\S 5]{Lott-Villani09}. More precisely, we
consider here the weaker definition of \cite{Sturm06I} and we will
discuss a stronger version in Section~\ref{se:strongcd}: see the
bibliographical references of \cite[Chapter 17]{Villani09} for a
comparison between the two approaches.
\begin{definition}[$CD(K,\infty)$ spaces]\label{def:cdk}
We say that $(X,\sfd,\mm)\in\X$ has Ricci curvature bounded from
below by $K\in\R$ (in short: it is a $CD(K,\infty)$ space) if the
relative entropy functional $\entv$ is $K$-geodesically convex on
$(\probt X,W_2)$, i.e. for any pair of measures $\mu,\,\nu\in
D(\entv)\cap\probt X$ there exists a constant speed geodesic
$(\mu_t)\subset\probt X$ such that $\mu_0=\mu$, $\mu_1=\nu$ and
\[
\entr{\mu_t}\mm\leq(1-t)\entr{\mu_0}\mm+t\entr{\mu_1}\mm-\frac
K2t(1-t)W_2^2(\mu_0,\mu_1)\qquad\forevery t\in[0,1].
\]
\end{definition}
Notice that, in comparison with the definition given in
\cite{Sturm06I} and \cite{Lott-Villani09} we are restricting the
analysis to the case of a probability reference measure  $\mm$ with
finite second moment (but we do not assume local compactness). This
is actually unneeded from the ``Ricci bound'' point of view (see
also \cite[Definition~9.1]{Ambrosio-Gigli-Savare11}), however in
this paper we want to focus more on the geometrical aspect, rather
than on the - non trivial - analytic tools needed to work in higher
generality: the assumption $\mm\in\probt X$ serves to this scope.

Let us also remark that a $CD(K,\infty)$ space $(X,\sfd,\mm)$
satisfies the length property, i.e.~$\supp\mm$ is a length space if
it is endowed with the distance $\sfd$ \cite[Remark
4.6(iii)]{Sturm06I} (the proof therein, based on an approximate
midpoint construction, does not use the local compactness).

Now let $(X,\sfd,\mm)$ be a $CD(K,\infty)$ space. Then, by
assumption, the relative entropy functional $\entv$ is
$K$-geodesically convex on $(\probt X,W_2)$, so that we could ask
about the existence and the uniqueness of its gradient flow. The
following theorem, proved in \cite{Gigli10} for the locally compact
case and generalized in
\cite[Theorem~9.3(ii)]{Ambrosio-Gigli-Savare11} holds:

\begin{theorem}[Gradient flow of the relative entropy]
\label{thm:gfe} Let $(X,\sfd,\mm)$ be a $CD(K,\infty)$ space. Then
for any $\mu\in D(\entv)\cap\probt X$ there exists a unique gradient
flow of $\entv$ starting from $\mu$.
\end{theorem}

Notice that the theorem says nothing about contractivity of the
Wasserstein distance along the flow, a property which we address in
Section~\ref{sec:EVI}. Actually, Ohta and Sturm proved in
\cite{Sturm-Ohta10} that contractivity \emph{fails} if
$(X,\sfd,\mm)$ is $\R^d$ endowed with the Lebesgue measure and with
a distance coming from a norm not induced by a scalar product.

\subsubsection{The heat flow as gradient flow in $L^2(X,\mm)$ and in
$\probt X$}

One of the main result of \cite{Ambrosio-Gigli-Savare11} has been
the following identification theorem, see formula (8.5) and
Theorem~9.3(iii) therein.

\begin{theorem}[The heat flow as gradient flow]\label{thm:heatgf}
Let $(X,\sfd,\mm)$ be a $CD(K,\infty)$ space and let $f\in
L^2(X,\mm)$ be such that $\mu=f\mm\in\probt{X}$. Let $(f_t)$ be the
gradient flow of $\C$ in $L^2(X,\mm)$ starting from $f$ as in
Theorem~\ref{thm:gfc}, and let $(\mu_t)$ be the gradient flow of
$\entv$ in $\probt X$ starting from $\mu$, as in
Theorem~\ref{thm:gfe}.\\
Then $\mu_t=f_t\mm$ for all $t\geq 0$, $t\mapsto\entv(\mu_t)$ is
locally absolutely continuous in $[0,\infty)$, and
\begin{equation}\label{eq:edissrateflow}
-\frac{\d}{\d
t}\entv(\mu_t)=|\dot\mu_t|^2=\int_{\{f_t>0\}}\frac{\weakgrad
{f_t}^2}{f_t}\,\d\mm\qquad\text{for a.e.~$t>0$.}
\end{equation}
\end{theorem}

In other words, we can unambiguously define the heat flow on a
$CD(K,\infty)$ space either as the gradient flow of Cheeger's energy
in $L^2(X,\mm)$ or as the gradient flow of the relative entropy in
$(\probt X,W_2)$. A byproduct of this proof is also (see
\cite[Theorem~9.3(i)]{Ambrosio-Gigli-Savare11}) the equality between
slope and the so-called Fisher information functional:
\begin{equation}\label{eq:slopeFisher}
|\nabla^-\entv|^2(\rho\mm)=4\int\weakgrad{\sqrt{\rho}}^2\,\d\mm
\end{equation}
for all probability densities $\rho$ such that $\sqrt{\rho}\in
D(\C)$. Choosing $f=\sqrt{\rho}$ this identity, in conjunction with
the $HWI$ inequality relating entropy, Wasserstein distance and
Fisher information (see \cite{Lott-Villani-Poincare} or
\cite[Proposition~7.18]{Ambrosio-Gigli11}) gives the log-Sobolev
inequality
\begin{equation}\label{eq:logSobolev}
\int f^2\log f^2\,\d\mm\leq\frac{2}{K}\int\weakgrad{f}^2\,\d\mm
\qquad\text{whenever $f\in D(\C)$ and $\int f^2\,\d\mm=1$.}
\end{equation}

We will denote by $\heatl_t:L^2(X,\mm)\to L^2(X,\mm)$ the heat
semigroup in $L^2(X,\mm)$ and by $\heatw_t:\probt X\to\probt X$ the
gradient flow of the entropy on $\probt X$. A distinct notation is
useful not only for conceptual reasons, but also because the domains
of the two gradient flows don't match, even if we identify
absolutely continuous measures with their densities.

Some basic properties of the heat flow that we will need later on
are collected in the following proposition, see
\cite[Theorem~4.16]{Ambrosio-Gigli-Savare11} also for further
details.

\begin{proposition}[Some properties of the heat flow]\label{prop:proprheat}
Let $(X,\sfd,\mm)\in\X$ and $f\in L^2(X,\mm)$. Then the following
statements hold:
\begin{itemize}
\item[(i)] (Maximum principle) If $f\leq C$ (resp. $f\geq C$) $\mm$-a.e.~in $X$ for some
$C\in\R$, then $\heatl_t(f)\leq C$ (resp. $\heatl_t(f)\geq C$)
$\mm$-a.e.~in $X$ for any $t\geq 0$.
\item[(ii)] ($1$-homogeneity)
$\heatl_t(\lambda f)=\lambda\heatl_t(f)$ for any $\lambda\in\R$,
$t\geq 0$.
\end{itemize}
\end{proposition}

\subsection{EVI formulation of gradient flows}\label{sec:EVI}

Here we recall a stronger formulation of gradient flows in a
complete and separable metric space $(Y,\sfd_Y)$, introduced and
extensively studied in \cite{Ambrosio-Gigli-Savare08},
\cite{Daneri-Savare08}, \cite{Savare10}, which will play a key role
in our analysis.
\begin{definition}[Gradient flows in the $\EVI$ sense]\label{def:EVIK}
Let $E:Y\to\R\cup\{+\infty\}$ be a lower semicontinuous functional,
$K\in\R$ and $(0,\infty)\ni t\mapsto y_t\in D(E)$ be a locally
absolutely continuous curve. We say that $(y_t)$ is a $K$-gradient
flow for $E$ in the Evolution Variational Inequalities sense (or,
simply, it is an $\EVI_K$ gradient flow) if for any $z\in Y$ it
holds
\begin{equation}
\label{eq:defevi} \frac \d{\d t}\frac{\sfd_Y^2(y_t,z)}2+\frac
K2\sfd_Y^2(y_t,z)+E(y_t)\leq E(z)\qquad\text{for a.e.~$t\in
(0,\infty)$.}
\end{equation}
If $\lim\limits_{t\downarrow 0}y_t=y_0 \in \overline {D(E)}$, we say
that the gradient flow starts from $y_0$.
\end{definition}
Notice that the derivative in \eqref{eq:defevi} exists for
a.e.~$t>0$, since $t\mapsto\sfd_Y(y_t,z)$ is locally absolutely
continuous in $(0,\infty)$.

In the next proposition we will consider equivalent formulations of
\eqref{eq:defevi} involving subsets $D\subset D(E)$ \emph{dense in
energy}: it means that for any $y\in D(E)$ there exists a sequence
$(y_n)\subset D$ such that $\sfd_Y(y_n,y)\to 0$ and $E(y_n)\to E(y)$
as $n\to\infty$.

\begin{proposition}[Equivalent formulations of $\EVI$]\label{prop:eviequiv}
Let $E$, $K$ be as in Definition~\ref{def:EVIK}, $D\subset D(E)$
dense in energy, and $y:(0,\infty)\to D(E)$ be a locally absolutely
continuous curve with $\lim_{t\downarrow0}y_t=y_0\in
\overline{D(E)}$. Then, $(y_t)$ is an $\EVI_K$ gradient flow if and
only if one of the following properties is satisfied:
\begin{itemize}
\item[(i)] (Dense version)
The differential inequality \eqref{eq:defevi} holds for all $z\in
D$.
\item[(ii)] (Integral version)
For all $z\in D$ it holds
\begin{equation}
\label{eq:eviint} \frac{\rme^{K(t-s)}}2
\sfd^2_Y(y_t,z)-\frac{\sfd^2_Y(y_s,z)}{2}\le \mathrm
I_K(t-s)\Big(E(z)-E(y_t)\Big)\qquad \forevery 0\le s\leq t,
\end{equation}
where $\mathrm I_K(t):=\int_0^t \rme^{K r}\,\d r$.
\item[(iii)] (Pointwise version) For all $z\in D$ it holds
\begin{equation}
\label{eq:evipoint} \limsup_{h\downarrow
0}\frac{\sfd^2_Y(y_{t+h},z)-\sfd^2_Y(y_t,z)}2+\frac
K2\sfd^2_Y(y_t,z)+E(y_t)\leq E(z)\qquad\forevery t>0.
\end{equation}
\end{itemize}
\end{proposition}
\begin{proof}
To get \eqref{eq:eviint} for all $z\in D(E)$ from \eqref{eq:defevi},
just multiply by $\rme^{Kt}$ and integrate in time, using the fact
that $t\mapsto E(y_t)$ is nonincreasing (see e.g.\
\cite{Clement-Desch10} and the next Proposition); a differentiation
provides the equivalence, since $y$ is absolutely continuous. The
fact that \eqref{eq:eviint} holds for any $z$ if and only if it
holds in a set dense in energy is trivial, so that the equivalence
of $(ii)$ and Definition~\ref{def:EVIK} is proved. The equivalences
with $(i)$ and $(iii)$ follow by similar arguments.
\end{proof}

We recall some basic and useful properties of gradient flows in the
$\EVI$ sense; we give here the essential sketch of the proofs,
referring to \cite[Chap.\ 4]{Ambrosio-Gigli-Savare08} and
\cite{Savare10} for more details and results. In particular, we
emphasize that the maps $\sfS_t:y_0\mapsto y_t$ that at every $y_0$
associate the value at time $t\ge0$ of the unique $K$-gradient flow
starting from $y_0$ give raise to a continuous semigroup of
$K$-contractions according to \eqref{eq:21} in a closed (possibly
empty) subset of $Y$.
\begin{proposition}[Properties of gradient flows in the $\EVI$ sense]\label{prop:evipropr}
Let $Y$, $E$, $K$, $y_t$ be as in Definition~\ref{def:EVIK}
 and suppose that $(y_t)$ is an $\EVI_K$ gradient flow of $E$ starting from
$y_0$. Then:
\begin{itemize}
\item[(i)] If $y_0\in D(E)$, then $y_t$ is also a metric gradient flow,
i.e.~\eqref{eq:ede} holds.
\item[(ii)] If $(\tilde y_t)$ is another $\EVI_K$ gradient flow for $E$ starting
from $\tilde{y}_0$, it holds
\begin{equation}
\sfd_Y(y_t,\tilde y_t)\leq e^{-Kt}\sfd_Y(y_0,\tilde y_0).\label{eq:21}
\end{equation}
In particular, $\EVI_K$ gradient flows uniquely depend on the
initial condition.
\item[(iii)]
Existence of $\EVI_K$ gradient flows starting from any point in
$D\subset Y$ implies existence starting from any point in $\overline
D$.
\item[(iv)] $(y_t)$ is locally Lipschitz in $(0,\infty)$,
  $y_t\in D(|\nabla^- E|)$ for every $t>0$, the map $t\mapsto \rme^{K
    t}\,|\nabla^-E|(y_t)$ is nonincreasing, and
  we have the regularization estimate
  \begin{equation}\label{eq:ultraEVI}
    \mathrm I_K(t)\,E(y_t)+\frac{\big(\mathrm I_K(t)\big)^2}2|\nabla^- E|^2(y_t)\leq
    \mathrm I_K(t)\,E(z)+\frac12
    \sfd_Y^2(z,y_0)\quad
    \forall t>0,\ z\in D(E).
  \end{equation}
\end{itemize}
\end{proposition}
\begin{proof}
The fact that $\EVI_K$ gradient flows satisfy \eqref{eq:ede} has
been proved by the third author in \cite{Savare10} (see also
\cite[Proposition~3.9]{Ambrosio-Gigli11}). The contractivity
property $(ii)$ has been proved in
\cite[Chap.~4]{Ambrosio-Gigli-Savare08}. Statement $(iii)$ follows
trivially from contractivity and integral formulation
\eqref{eq:eviint} of the $\EVI$. The fact that $t\mapsto \rme^{K
t}|\nabla^- E|(y_t)$ is nonincreasing follows from the energy
identity, which shows that $|\nabla^- E|(y_t)=|\dot y_t|$, and the
$K$-contraction estimate \eqref{eq:21}, which in particular yields
that $t\mapsto \rme^{Kt}\sfd_Y(y_t,y_{t+h})$ is nonincreasing as
well as $t\mapsto \rme^{Kt}|\dot y_t|$.

An easier regularization formula for $t\mapsto E(y_t)$ follows
immediately by \eqref{eq:eviint} by choosing $s=0$ and neglecting
the term proportional to $\sfd_Y^2(y_t,z)$. Inequality
\eqref{eq:ultraEVI} is a consequence of the $\EVI_K$, the identity
$\frac\d{\dt}E(y_t)=-|\nabla^- E|^2(y_t)$, the previous monotonicity
property and the following calculations:
  \begin{align*}
    \frac 12&\big(\mathrm I_K(t)\big)^2|\nabla^-E|^2(y_t)
    =\frac12\big(\mathrm I_{-K}(t)\big)^2\rme^{2 K t}|\nabla^- E|^2(y_t)
    \le
    \int_0^t \mathrm I_{-K}(s) \rme^{-K s}
    \rme^{2K s}|\nabla^- E|^2(y_s)\,\d s
    \\&
    =-\int_0^t  \mathrm I_{-K}( s) \rme^{K s}
    \big(E(y_s)-E(y_t)\big)'\,\d s
    =
    \int_0^t \rme^{K s}\big(E(y_s)-E(y_t)\big)\,\d s
    \\&\le
    \int_0^t -\frac 12\big(\rme^{K s}\sfd^2_Y(y_s,z)\big)'
    +\rme^{K s}\big(E(z)-E(y_t)\big)\,\d s
    %\\&
    \le
    \frac 12\sfd^2_Y(y_0,z)+\mathrm I_K(t)\big(E(z)-E(y_t)\big).
  \end{align*}
\end{proof}

We point out that in general existence of $\EVI_K$ gradient flows is
a consequence of the $K$-geodesic convexity of $E$ and of strong
geometric assumptions on the metric space $(Y,\sfd_Y)$: it is well
known when $Y$ is a convex set of an Hilbert space, but existence
holds even when $(Y,\sfd_Y)$ satisfies suitable lower sectional
curvature bounds in the sense of Alexandrov
\cite{Gigli-Ohta10,Ohta09,Savare07,Savare10}, or when suitable
compatibility conditions between $E$ and $\sfd$ hold
\cite[Chapter~4]{Ambrosio-Gigli-Savare08} which include also spaces
with nonpositive Alexandrov curvature. In the present paper we will
study an important situation where $\EVI_K$ gradient flows arise
without any assumption on sectional curvature.

In any case, $\EVI_K$ gradient flows have the following interesting
geometric consequence on the functional $E$
\cite[Theorem~3.2]{Daneri-Savare08}: if $\EVI_K$ gradient flows
exist for any initial data, then the functional is $K$-convex along
\emph{any} geodesic contained in $\overline {D(E)}$. Recall that the
standard definition of geodesic convexity, e.g.~the one involved in
Definition~\ref{def:cdk} of $CD(K,\infty)$ metric measure spaces,
requires convexity along \emph{some} geodesic; this choice is
usually motivated by stability properties
w.r.t.~$\Gamma$-convergence
\cite[Thm.~9.1.4]{Ambrosio-Gigli-Savare08} and
Sturm-Gromov-Hausdorff convergence in the case of metric measure
spaces (see also the next section). We state this property in a
quantitative way, which will turn out to be useful in the following.
\begin{proposition}\label{prop:danerisavare}
Let $E$, $K$, $y_t$ be as in Definition~\ref{def:EVIK} and assume
that for every $y_0\in\overline{D(E)}$ there exists the $\EVI_K$
gradient flow $y_t:=\sfS_t(y_0)$ for $E$ starting from $y_0$. If
$\eps\ge0$ and $\gamma:[0,1]\to\overline{D(E)}$ is a Lipschitz curve
satisfying
\begin{equation}
  \label{eq:5}
  \sfd_Y(\gamma_{s_1},\gamma_{s_2})\le L\,|s_1-s_2|,\quad
  L^2\le \sfd^2_Y(\gamma_0,\gamma_1)+\eps^2\qquad\forevery s_1,\,s_2\in [0,1],
\end{equation}
then for every $t>0$ and $s\in [0,1]$
\begin{equation}
  \label{eq:6}
  E(\sfS_t(\gamma_s))\le (1-s)E(y_0)+sE(y_1)-\frac
  K2s(1-s)\sfd^2_Y(y_0,y_1)+\frac{\eps^2}{2\mathrm I_K(t)} s(1-s).
\end{equation}
In particular $E$ is $K$-convex along all geodesics contained in
$\overline{D(E)}$.
\end{proposition}

The last statement is an immediate consequence of \eqref{eq:6} by
choosing $\eps=0$ and letting $t\down0$.

\section{Strong $CD(K,\infty)$ spaces}\label{se:strongcd}
\begin{definition}[Strong $CD(K,\infty)$ spaces]\label{def:strongcd}
We say that $(X,\sfd,\mm)$ is a strong $CD(K,\infty)$ space if for
every $\mu_0,\,\mu_1\in D(\entv)\cap\probt{X}$ there exists an
optimal geodesic plan $\ppi$ from $\mu_0$ to $\mu_1$ such that
$K$-convexity of the entropy holds along all weighted plans
$\ppi_F:=F\ppi$, where $F:\geo(X)\to\R$ is any Borel, bounded, non
negative function such that $\int F\,\d\ppi=1$. More precisely, for
any such $F$, the interpolated measures $\muF
t:=(\e_t)_\sharp\ppi_F$ satisfy:
\[
\entr{\muF t }\mm\leq(1-t)\entr{\muF0}\mm +t\entr{\muF1}\mm-\frac K2
t(1-t)W_2^2(\muF0,\muF1)\qquad\forall t\in [0,1].
\]
\end{definition}

It is unclear to us whether this notion is stable w.r.t
$\D$-convergence or not. As such, it should be handled with care. We
introduced this definition for two reasons. The first one is that
applying Proposition~\ref{prop:danerisavare} we will show in
Lemma~\ref{le:EVISCD} that if a length metric measure space
$(X,\sfd,\mm)$ admits existence of $\EVI_K$ gradient flows of
$\entv$ for any initial measure $\mu\in D(\entv)\cap\probt X$, then
it is a strong $CD(K,\infty)$ space. Given that spaces admitting
$\EVI_K$ gradient flows for $\entv$ are the main subject of
investigation of this paper, it is interesting to study a priori the
properties of strong $CD(K,\infty)$ spaces. The other reason is due
to the fact that the proof that linearity of the heat flow implies
the existence of $\EVI_K$ gradient flows of the entropy requires
additional $L^\infty$-estimates for displacement interpolations
which looks unavailable in general $CD(K,\infty)$ spaces.

\begin{remark}[The nonbranching case]\label{re:nonbr}{\rm If a space $(X,\sfd,\mm)$
is $CD(K,\infty)$ and nonbranching, then it is also strong
$CD(K,\infty)$ according to the previous definition. \\
Indeed, pick $\mu_0,\,\mu_1\in D(\entv)$, let
$\ppi\in\gopt(\mu_0,\mu_1)$ be such that the relative entropy is
$K$-convex along $((\e_t)_\sharp\ppi)$, so that
$\entr{(\e_t)_\sharp\ppi}{\mm}$ is bounded in $[0,1]$. Now, pick $F$
as in Definition~\ref{def:strongcd}, let $\muF
t:=(\e_t)_\sharp\ppi_F$ and notice that the real function $s\mapsto
\phi(s):=\entr{\muF s}{\mm}$ is bounded (thus in particular $\muF
s\in D(\entv)$) in $[0,1]$, since $\muF t\leq\sup|F|
(\e_t)_\sharp\ppi$.\\
The nonbranching assumption ensures that for any $t\in(0,1)$ there
is a unique geodesic connecting $\muF t$ to $\muF 0$ (and similarly
to $\muF 1$). Hence, since $(X,\sfd,\mm)$ is a $CD(K,\infty)$ space,
the restriction of $\phi$ to all the intervals of the form $[0,t]$
and $[t,1]$ for $t\in(0,1)$ is $K$-convex and finite. It follows
that $\phi$ is $K$-convex in $[0,1]$. \fr }\end{remark}

In order to better understand the next basic interpolation estimate,
let us consider the simpler case of an otimal geodesic plan $\ppi\in
\gopt(\mu_0,\mu_1)$ in a nonbranching $CD(K,\infty)$ space
$(X,\sfd,\mm)$. Assuming that $\mu_i=\rho_i\mm\in D(\entv)$ and
setting $\mu_t:=\rho_t\mm$, along $\ppi$-a.e.\ geodesic $\gamma$ the
real map $t\mapsto\log\rho_t(\gamma_t)$ is $K$-convex and therefore
$\rho_t(\gamma_t)$ can be pointwise estimated by \cite[Thm.~30.32,
(30.51)]{Villani09}
\begin{equation}
  \label{eq:7}
  \rho_t(\gamma_t)\le \rme^{-\frac
    K2\,t(1-t)\,\sfd^2(\gamma_0,\gamma_1)}\rho_0(\gamma_0)^{1-t}\,\rho_1(\gamma_1)^t\quad
  \forevery t\in [0,1],\ \text{for $\ppi$-a.e.\ $\gamma$.}
\end{equation}
Inequality \eqref{eq:7} for smooth Riemannian manifolds goes back to
\cite{Cordero-McCann-Schmuckenschlager01}. If $\mu_i$ have bounded
supports, one immediately gets the uniform $L^\infty$-bound:
\begin{equation}
  \label{eq:8}
  \|\rho_t\|_\infty
  \le \rme^{\frac {K^-}2    t(1-t)S^2}
  \|\rho_0\|_\infty
  ^{1-t}\,\|\rho_1\|_\infty^t,\quad\text{with }
  S:=\sup\big\{\sfd(x_0,x_1):\ x_i\in\supp(\mu_i)\big\}.
\end{equation}
When $K\ge0$ \eqref{eq:8} is also a consequence of the definition (stronger than
\eqref{def:cdk}) of spaces with
non-negative Ricci curvature given by \cite{Lott-Villani09}, which in
particular yields the geodesic convexity of all the
functionals $U_{p}(\mu):=\int \rho\,^p\,\d\mm$ whenever $\mu=\rho\mm$
and $p>1$.

If we know that only $\rho_1$ is supported in a bounded set, we can
still get a weighted $L^\infty$-bound on $\rho_t$. Let us assume
that
\begin{equation}
  \label{eq:9}
  \supp\rho_1\subset \sfC,\quad\text{with}\quad
  \Di:=\diam(\sfC)<\infty,\quad
  \sfD(x):=\mathrm{dist}(x,\sfC)\quad\text{for }x\in X,
\end{equation}
and let us observe that for $\ppi$-a.e.~$\gamma$ we have
$\gamma_1\in\supp\mu_1$, so that for every $t\in [0,1)$ it holds
\begin{align}
\label{eq:10}
\sfd(\gamma_0,\gamma_1)&=\frac{\sfd(\gamma_t,\gamma_1)}{1-t}\leq\frac{\sfD(\gamma_t)+{\sf
Di}}{1-t},
\\
\label{eq:11}
\sfD(\gamma_0)&\geq \sfd(\gamma_t,\gamma_1)-{\sf
Di}-\sfd(\gamma_0,\gamma_t)= (1-2t)\sfd(\gamma_0,\gamma_1)-{\sf
Di}\geq\frac{1-2t}{1-t}\sfD(\gamma_t)-{\sf Di}.
\end{align}
Substituting the above bounds in \eqref{eq:7} we get
\begin{equation}
\label{eq:boundinterpolate} \rho_t(x)\leq
\rme^{\frac{K^-}2\frac{t}{1-t}(\sfD(x)+{\sf Di})^2}
\|\rho_0\|^{1-t}_{L^\infty(R(\sfD(x),t);\mm)}\, \|\rho_1\|^t_\infty
\qquad\text{$\mm$-a.e.~in $X$},
\end{equation}
where
\begin{equation}\label{eq:defM}
R(D,t):= \Big\{y\in X:\sfD(y)\ge \frac{1-2t}{1-t}D-\Di\Big\},\quad
D\ge 0,\ t\in [0,1).
\end{equation}

The next lemma shows that the strong $CD(K,\infty)$ condition is
sufficient to obtain the same estimates.

\begin{proposition}[Interpolation properties]\label{prop:intpropr}
Let $(X,\sfd,\mm)$ be a strong $CD(K,\infty)$ space and let
$\rho_0,\,\rho_1$ be probability densities such that
$\mu_i=\rho_i\mm\in D(\entv)\cap\probt X$. Assume that $\rho_1$ is
bounded and with support contained in a bounded set ${\sf C}$ as in
\eqref{eq:9} and let $\ppi\in\gopt(\mu_0,\mu_1)$ as in
Definition~\ref{def:strongcd}. Then for all $t\in [0,1)$ the density
$\rho_t$ of $\mu_t=(\e_t)_\sharp\ppi$ satisfies
\eqref{eq:boundinterpolate}. Furthermore, if also $\rho_0$ is
bounded with bounded support, then \eqref{eq:8} holds and
$\sup_t\|\rho_t\|_\infty<\infty$.
\end{proposition}
\begin{proof} Let $\ppi$ be given by the strong $CD(K,\infty)$ condition. Fix
$t\in(0,1)$ and assume that \eqref{eq:boundinterpolate} does not
hold on a Borel set $B$ of positive $\mm$-measure. Then we can find
a Borel set $A\subset B$ with $\mm(A)>0$ such that
\[
\rho_t(x)>\rme^{\frac{K^-}2\frac{t}{1-t}\left(\sfD_1+{\sf
Di}\right)^2} \sfM^{1-t} \|\rho_1\|_\infty^t\qquad\forall x\in A,
\]
where
\begin{equation}
  \label{eq:12}
  \sfM:=\|\rho_0\|_{L^\infty(R(\sfD_2,t);\mm)},\quad
  \sfD_1:=\sup_{x\in A}\sfD(x),\quad
  \sfD_2:=\inf_{x\in A} \sfD(x).
\end{equation}
To build $A$, it suffices to slice $B$ in countably many pieces
where the oscillation of $D$ is sufficiently small. We have
$\ppi((\e_t)^{-1}(A))=\mu_t(A)>0$, thus the plan
$\tilde\ppi:=c\,\ppi\res{\e_t^{-1}( A)}$, where
$c:=\bigl[\mu_t(A)\bigr]^{-1}$ is the normalizing constant, is well
defined. Let $\tilde\rho_s$ be the density of
$\tilde\mu_s=(\e_s)_\sharp\tilde\ppi$. By definition it holds
$\tilde\rho_t=c\rho_t$ on $A$ and $\tilde\rho_t=0$ on $X\setminus
A$, thus we have:
\begin{equation}
\label{eq:entrrhot}
%\begin{split}&
\entr{\tilde\mu_t%(\e_t)_\sharp\tilde\ppi
}\mm=\int\tilde\rho_t\log\tilde\rho_t\,\d\mm
%\\&
> \log c+\frac{K^-}2\frac t{1-t}\bigl(\sfD_1+{\sf Di}\bigr)^2+(1-t)
\log\sfM+t\log\|\rho_1\|_\infty.
\end{equation}
On the other hand, we have $\tilde\rho_0\leq c\rho_0$ and
$\tilde\rho_1\leq c\rho_1$ hence
\begin{equation}
\label{eq:entrrho01}
\entr{\tilde\mu_0}\mm
=\int\log(\tilde\rho_0\circ\e_0)\,\d\tilde\ppi
\leq\log c+
\log\Big(\|\rho_0\circ\e_0\|_{L^\infty(\geo(X),\tilde\sppi)}\Big),\qquad
\end{equation}
\begin{equation}\label{eq:entrrho010}
\entr{\tilde\mu_1}\mm=\int\log(\tilde\rho_1\circ\e_1)\,
\d\tilde\ppi\leq\log c+\log\|\rho_1\|_\infty.
\end{equation}
Now observe that $\tilde\ppi$-a.e.~geodesic $\gamma$ satisfies
$\gamma_t\in A$ and $\gamma_1\in \supp\rho_1\subset {\sf C}$, so
that \eqref{eq:10} and \eqref{eq:11} yield
\[
\sfd(\gamma_0,\gamma_1) \leq\frac{\sfD_1+{\sf Di}}{1-t},\qquad
\sfD(\gamma_0) \geq\frac{1-2t}{1-t}\sfD_2-{\sf Di},\quad
\text{i.e.}\quad \gamma_0\in R(\sfD_2,t).
\]
Integrating the squared first inequality w.r.t.~$\tilde\ppi$ and
combining the second one, \eqref{eq:entrrho01}, and \eqref{eq:12} we
get
\begin{equation}
\label{eq:entrrho0} W_2^2(\tilde\mu_0,\tilde\mu_1)\leq
\left(\frac{\sfD_1+{\sf Di}}{1-t}\right)^2, \qquad
\entr{\tilde\mu_0}\mm\leq \log c+(1-t)\log \sfM.
\end{equation}
Inequalities \eqref{eq:entrrhot}, \eqref{eq:entrrho010}, and
\eqref{eq:entrrho0} contradict the $K$-convexity of the entropy
along $((\e_s)_\sharp\tilde\ppi)$, so the proof of the first claim
is concluded.

The proof of \eqref{eq:8} when also $\rho_0$ has bounded support
follows the same lines just used. Let $t\in(0,1)$ and assume that
\eqref{eq:8} does not hold. Thus there exists a Borel set $A$ of
positive $\mm$-measure such that $\rho_t>
e^{K^-t(1-t)S^2/2}\|\rho_0\|^{1-t}_\infty\,\|\rho_1\|^t_\infty$ in
$A$. As before, we define $\tilde\ppi:=c\ppi\res{\e_t^{-1}(A)}$,
where $c$ is the normalizing constant, and $\tilde\rho_s$ as the
density of $(\e_s)_\sharp\tilde\ppi$: the inequalities
\[
\begin{split}
&\entr{(\e_t)_\sharp\tilde\ppi}\mm>\log c+
\frac{K^-}{2}t(1-t)S^2+(1-t)\log\|\rho_0\|_\infty+t\log\|\rho_1\|_\infty,\\
&\entr{(\e_0)_\sharp\tilde\ppi}\mm\leq \log
c+\log\|\rho_0\|_\infty,\,\,\,\,
\entr{(\e_1)_\sharp\tilde\ppi}\mm\leq\log c+\log\|\rho_1\|_\infty,\\
&W_2^2\big((\e_0)_\sharp\tilde\ppi,(\e_1)_\sharp\tilde\ppi\big)\leq
S^2,
\end{split}
\]
contradict the $K$-convexity of the entropy along  $((\e_s)_\sharp\tilde\ppi)$.
\end{proof}

In the sequel we will occasionally use the stretching/restriction
operator ${\rm restr}_0^s$ in $\CC{[0,1]}X$, defined for all $s\in
[0,1]$ by
$$
{\rm restr}_0^s(\gamma)_t:=\gamma_{ts}\qquad t\in [0,1].
$$
\begin{proposition}[Existence of test plans]\label{prop:testplan}
Let $(X,\sfd,\mm)$ be a strong $CD(K,\infty)$ space and let
$\rho_0,\,\rho_1$ be probability densities. Assume that $\rho_1$ is
bounded with bounded support as in \eqref{eq:9}, that $\rho_0$ is
bounded and satisfies
\begin{equation}\label{eq:boundrho0}
\rho_0(x)\leq c\,\rme^{-9K^-(\sfD(x)-C)^2}\qquad\text{whenever
$\sfD(x):=\mathrm{dist}(x,\supp\rho_1)>R$},
\end{equation}
for some nonnegative constants $c,\,C,\,R$. Then,  for
$\ppi\in\gopt(\rho_0\mm,\rho_1\mm)$ as in
Definition~\ref{def:strongcd}, $({\rm restr}_0^{1/3})_\sharp\ppi$ is
a test plan (recall Definition~\ref{def:testplan}).
\end{proposition}
\begin{proof}
In order to avoid cumbersome formulas, in this proof we switch to
the exp notation. We need to prove that $\sup_X\rho_t$ is uniformly
bounded in $[0,1/3]$. Let ${\sf Di}={\rm diam}(\supp\rho_1)$, $M$
the function defined in \eqref{eq:defM}, $L$ a constant to be
specified later, $A:=\{y:\ \sfD(y)\leq L\}$ and set
$\ppi^1:=\ppi\res{\e_0^{-1}(A)}$,
$\ppi^2:=\ppi\res{\e_0^{-1}(X\setminus A)}$. Choosing $L$ large
enough we have $\alpha:=\ppi(\e_0^{-1}(A))>0$ and we can also assume
that $\alpha<1$ (otherwise, $\rho_0$ has bounded support and the
second part of Proposition~\ref{prop:intpropr} applies). Also,
possibly increasing $R$ and taking \eqref{eq:boundrho0} into
account, we can assume that $\rho_0(x)\leq 1$ wherever $\sfD(x)\geq R$.

Denoting by $\tilde{\ppi}^1,\,\tilde{\ppi}^2$ the corresponding
renormalized plans, it suffices to show that both have bounded
densities in the time interval $[0,1/3]$, because $\ppi$ is a convex
combination of them. Concerning $\tilde{\ppi}^1$, notice that both
$(\e_0)_\sharp\tilde{\ppi}^1$ and $(\e_1)_\sharp\tilde{\ppi}^1$ have
bounded support and bounded density, so that the conclusion follows
from the second part of Proposition~\ref{prop:intpropr}.

For $\tilde{\ppi}^2$ we argue as follows. Pick
$\gamma\in\supp\tilde{\ppi}^2$ and notice that
$\gamma_1\in\supp\rho_1$ and $t\leq \frac13$ give the inequality
\[
\sfD(\gamma_t) \geq \sfd(\gamma_t,\gamma_1)-{\sf
  Di}=(1-t)\sfd(\gamma_0,\gamma_1)-{\sf Di}
\geq (1-t)\sfD(\gamma_0) -{\sf Di}\geq \frac23D(\gamma_0)- {\sf Di}.
\]
So, choosing $L$ sufficiently large (depending only on ${\sf Di}$
and $R$), we have
\[
\gamma_0\in X\setminus A\qquad\Rightarrow\qquad
\frac{\sfD(\gamma_t)}{2}-{\sf Di}>R.
\]
Recalling the definition \eqref{eq:defM} of $R(D,t)$ and using the
fact that $t\in [0,1/3]$, we get that
\begin{displaymath}
  y\in R(\sfD(\gamma_t),t)\quad\Rightarrow\quad
  \sfD(y)\ge \frac{1-2t}{1-t}\sfD(\gamma_t)-\Di\ge
  \frac{\sfD(\gamma_t)}{2}-{\sf Di}>R
  \quad
  \text{for all $\gamma\in\supp\tilde\ppi^2$,}
\end{displaymath}
and therefore \eqref{eq:boundrho0} gives
\begin{displaymath}
  \sup_{R(\sfD(\gamma_t),t)}\rho_0\le
  c\exp\Bigl(-9K^-\big(\sfD(\gamma_t)-\Di-C\big)^2\Bigr)
  \qquad
  \text{for all $\gamma\in\supp\tilde\ppi^2$.}
\end{displaymath}
Now, by applying \eqref{eq:boundinterpolate} to $\tilde{\ppi}^2$, we
get that the density $\eta_t$ of $(\e_t)_\sharp{\tilde\ppi}^2$
satisfies
\begin{align*}
  \eta_t(\gamma_t)&
  \leq \frac{1}{1-\alpha}\exp\Bigl(\frac{K^-}4(\sfD(\gamma_t)+{\sf
    Di})^2\Bigr)
  \|\rho_0\|_{L^\infty(R(\sfD(\gamma_t),t),\mm)}^{1-t}\|\rho_1\|_\infty^t
  \qquad
  \text{for $\tilde\ppi^2$-a.e.~$\gamma$.}
\end{align*}
Using the fact that $t$ varies in $[0,1/3]$ and $\rho_0\le 1$ in
$R(\sfD(\gamma_t),t)$, we eventually get
$$
\eta_t(\gamma_t)\leq c \frac{\|\rho_1\|^{1/3}_\infty}{1-\alpha}\exp
\biggl(\frac{K^-}4(\sfD(\gamma_t)+{\sf
  Di})^2-6K^-\Big(\frac{\sfD(\gamma_t)}{2}-{\sf
Di}-C\Big)^2\biggr)
  \qquad
  \text{for $\tilde\ppi^2$-a.e.~$\gamma$.}
$$
Since $-\tfrac 54 K^-\sfD^2(\gamma_t)$ is the leading term in the
exponential, the right-hand side is bounded and we deduce that
$\|\eta_t\|_{\infty}=\|\eta_t\circ\rme_t\|_{L^\infty(\geo(X),\tilde\sppi^2)}$
is uniformly bounded.
\end{proof}

\begin{proposition}[Metric Brenier theorem for strong $CD(K,\infty)$ spaces]\label{prop:brenierbip}
Let $(X,\sfd,\mm)\in\X$ be a strong $CD(K,\infty)$ space, $x_0\in
X$, $\mu_0=\rho_0\mm\in\probt X$ with
\begin{equation}
  \label{eq:13}
  0<c_R\le \rho_0\le c_R^{-1}\qquad\text{$\mm$-a.e.\ in
  }B_R(x_0)\quad\forevery R>0,
\end{equation}
and $\mu_1\in\probt X$ with bounded support and bounded density.
Then, for $\ppi\in\gopt{(\mu_0,\mu_1)}$ as in
Definition~\ref{def:strongcd}, there exists $L\in L^2(X,\mu_0)$ such
that
\[
L(\gamma_0)=\sfd(\gamma_0,\gamma_1)\qquad\text{for $\ppi$-a.e.
$\gamma\in\geo(X)$.}
\]
Furthermore,
$$L(x)=\weakgrad\varphi(x)=|\nabla^+\varphi|(x)\qquad\text{for $\mu_0$-a.e.~$x\in X$,}$$
where $\varphi$ is any Kantorovich potential relative to
$(\mu_0,\mu_1)$.
\end{proposition}
\begin{proof}
We apply the metric Brenier Theorem~10.3 of
\cite{Ambrosio-Gigli-Savare11} with $V(x)=\sfd(x,x_0)$. To this aim,
we need only to show that
\begin{equation}\label{eq:boundlinfty1}
(\e_t)_\sharp\ppi(B\cap B_R(x_0))\leq C(R)\mm(B)\qquad\forevery t\in
[0, 1/2],\,\,B\in\BorelSets{X},\ R>0.
\end{equation}
Denoting by $R_1$ the radius of a ball containing the
support of $\mu_1$, notice that if a curve $\gamma$ in the support of
$\ppi$ hits $B_R(x_0)$ at some time $s\in [0,1/2]$, then
$$
\sfd(\gamma_0,\gamma_1)\leq 2\sfd(\gamma_s,\gamma_1)\leq 2(R+R_1)
$$
because $\gamma_1\in B_{R_1}(x_0)$. Possibly restricting $\ppi$ to
the set of $\gamma$'s hitting $B_R(x_0)$ at some $s\in[0,1/2]$, an
operation which does not affect $((\e_t)_\sharp\ppi)\res B_R(x_0)$
for $t\in [0,1/2]$, we get that $(\e_0)_\sharp\ppi$,
$(\e_1)_\sharp\ppi$ have bounded support and bounded densities, thus
the conclusion follows from the second part of
Proposition~\ref{prop:intpropr}.
\end{proof}

\section{Key formulas}\label{se:formule}

\subsection{Derivative of the squared Wasserstein
distance}\label{sub5}

In this short section we compute the derivative of the squared
Wasserstein distance along a heat flow.

\begin{theorem}[Derivative of squared Wasserstein distance]
\label{thm:derw2} Let $(X,\sfd,\mm)$ be a $CD(K,\infty)$ space,
$\mu=\rho\mm\in\probt X$ such that $0<c\leq \rho\leq C<\infty$ and
define $\mu_t:=\heatw_t(\mu)=\rho_t\mm$. Let $\sigma\in\probt X$ and
for any $t>0$ let $\varphi_t$ be a Kantorovich potential relative to
$(\mu_t,\sigma)$. Then for a.e.~$t>0$ it holds
\begin{equation}
\label{eq:derw2} \frac{\d}{\d
t}\frac12W_2^2(\mu_t,\sigma)\leq\frac{\C(\rho_t-\eps\varphi_t)-\C(\rho_t)}\eps
\qquad\forall\eps>0.
\end{equation}
\end{theorem}
\begin{proof} Since $t\mapsto\rho_t\mm$ is a locally absolutely continuous curve in
$\probt X$, the derivative at the left hand side of \eqref{eq:derw2}
exists for a.e.~$t>0$. Also, the derivative of
$t\mapsto\rho_t=\heatl_t(\rho)\in L^2(X,\mm)$ exists for a.e.~$t>0$.
Fix $t_0>0$ where both derivatives exist and notice that since
$\varphi_{t_0}$ is a Kantorovich potential for $(\mu_{t_0},\sigma)$
it holds
\[
\begin{split}
\frac 12 W_2^2(\mu_{t_0},\sigma)&=\int_X\varphi_{t_0}\,\d\mu_{t_0}+\int\varphi_{t_0}^c\,\d\sigma\\
\frac 12
W_2^2(\mu_{t_0-h},\sigma)&\geq\int_X\varphi_{t_0}\,\d\mu_{t_0-h}+
\int\varphi_{t_0}^c\,\d\sigma\qquad\text{for all $h$ such that
$t_0-h> 0$.}
\end{split}
\]
Taking the difference between the first identity and the second
inequality, dividing by $h>0$, and letting $h\to 0$ we obtain
\[
\frac{\d}{\d t}\frac 12 W_2^2(\mu_{t},\sigma)\restr{t=t_0}
\leq\liminf_{h\downarrow 0}\int_X\varphi_{t_0}\frac{
\rho_{t_0}-\rho_{t_0-h}}{h}\,\d\mm.
\]
Now recall that $\varphi_{t_0}\in L^1(X,\mu_{t_0})$, so that by our
assumption on $\rho$ and the maximum principle
(Proposition~\ref{prop:proprheat}) we deduce that $\varphi_{t_0}\in
L^1(X,\mm)$. By Proposition~\ref{prop:potkant} we have $\weakgrad
{\varphi_{t_0}}\in L^2(X,\mm)$.

Now, if $\varphi_{t_0}\in L^2(X,\mm)$ the estimate of the $\liminf$
with the difference quotient of $\C$ is just a consequence of the
following three facts: the first one is that, for all $t>0$ we have
$h^{-1}(\rho_{t+h}-\rho_t)\to\Deltam\rho_t$ as $h\downarrow 0$ in
$L^2(X,\mm)$; the second one is that we have chosen $t_0$ such that
the full limit exists; the third one is the inequality
\[
\C(\rho_{t_0})+\eps\int\varphi_{t_0}\Deltam\rho_{t_0}\,\d\mm\leq\C(\rho_{t_0}-\eps\varphi_{t_0})
\qquad\forall\eps>0
\]
provided by the inclusion $-\Deltam
\rho_{t_0}\in\partial^-\C(\rho_{t_0})$.

For the general case, fix $t_0>0$ as before, $\eps>0$ and let
$\varphi^N:=\max\{\min\{\varphi_{t_0},N\},-N\}\in L^2(X,\mm)$ be the
truncated functions. Since the chain rule \eqref{eq:chainrule} gives
$\weakgrad{\varphi^N}\leq\weakgrad{\varphi_{t_0}}$, the locality of
the minimal weak gradient \eqref{eq:weaklocal} and the dominated
convergence theorem ensures that $\C(\rho_t-\eps \varphi^N)\to
\C(\rho_t-\eps \varphi_{t_0})$ as $N\to\infty$. Applying
Lemma~\ref{le:luigi} below with $f:=\varphi_{t_0}-\varphi^N$ we get
\[
\begin{split}
\sup_{h\in(0,t_0/2)}&\left|\int
(\varphi_{t_0}-\varphi^N)\frac{\rho_{t_0}-\rho_{t_0-h}}{h}\,\d\mm\right|^2\\
&\leq \sup_{h\in(0,t_0/2)}\frac1h\int\limits_{t_0-h}^{t_0}
\biggl(\int\limits_{\{|\varphi_{t_0}|>N\}}\weakgrad
{\varphi_{t_0}}^2\,\rho_s\,\d\mm\,\int\frac{\weakgrad{\rho_s}^2}{\rho_s}\,\d\mm\biggr)\,\d
s,
\end{split}
\]
and hence
\[
\limsup_{N\to\infty}\sup_{h\in(0,t_0/2)}\left|\int
(\varphi_{t_0}-\varphi^N)\frac{
\rho_{t_0}-\rho_{t_0-h}}{h}\,\d\mm\right|=0,
\]
which is sufficient to conclude, applying the liminf estimate to all
functions $\varphi^N$ and then passing to the limit.
\end{proof}

\begin{lemma}\label{le:luigi} With the same notation and assumptions
of the previous theorem, for every $f\in L^1(X,\mm)$ and
$[s, t]\subset (0,\infty)$ it holds
\begin{equation}
\label{eq:trucco} \left|\int
f\frac{\rho_{t}-\rho_s}{t-s}\,\d\mm\right|^2\leq
\frac1{t-s}\int_s^{t}\biggl(\int\weakgrad
f^2\,\rho_r\,\d\mm\,\int\frac{\weakgrad{\rho_s}^2}{\rho_s}\,\d\mm\biggr)\,\d r.
\end{equation}
\end{lemma}
\begin{proof}
Assume first that $f\in L^2(X,\mm)$. Then from \eqref{eq:partilapl}
we get
\[
\left|\int f\Deltam\rho_r\,\d\mm\right|^2\leq\left(\int\weakgrad
f\,\weakgrad{\rho_r}\,\d\mm\right)^2\leq\int\weakgrad
f^2\,\rho_r\,\d\mm\,\int\frac{\weakgrad{\rho_r}^2}{\rho_r}\,\d\mm,
\]
for all $r>0$, and the thesis follows by integration in $(s,t)$.

For the general case, let $f^N:=\max\{\min\{f,N\},-N\}\in
L^2(X,\mm)$ be the truncated functions. By
Proposition~\ref{prop:proprheat}(i) we know that
$\rho_{t}-\rho_s\in L^\infty(X,\mm)$, so that
\[
\lim_{N\to\infty}\int f^N\,\frac{\rho_t-\rho_s}{t-s}\,\d\mm=\int
f\,\frac{\rho_t-\rho_s}{t-s}\,\d\mm,
\]
by dominated convergence. Also, by the chain rule
\eqref{eq:chainrule} we have $\weakgrad{f^N}\leq\weakgrad{f}$
$\mm$-a.e.~in $X$. The conclusion follows.
\end{proof}

\subsection{Derivative of the entropy along a geodesic}\label{sub6}

We now look for a formula to bound from below the derivative of the
entropy along a geodesic, which is going to be a much harder task
compared to Theorem~\ref{thm:derw2}, due to the lack of a change of
variable formula. From the technical point of view, we will need to
assume that $(X,\sfd,\mm)$ is a strong $CD(K,\infty)$ space, it
order to apply the metric Brenier theorem \ref{prop:brenierbip}.
From the geometric point of view, the key property that we will use
is given by Lemma~\ref{le:horver} where we relate ``horizontal'' to
``vertical'' derivatives. In order to better understand the point,
we propose the following simple example.

\begin{example}{\rm
Let $\|\cdot\|$ be a smooth, strictly convex norm on $\R^d$ and let
$\|\cdot\|_*$ be the dual norm. Let $\mathcal L$ be the duality map
from $(\R^d,\|\cdot\|)$ to $(\R^d,\|\cdot\|_*)$ and let $\mathcal
L^*$ be its inverse (respectively, the differentials of the maps
$\frac12\|\cdot\|^2$ and $\frac 12\|\cdot\|_*^2$). For a smooth map
$f:\R^d\to\R$ its differential $Df(x)$ at any point $x$ is
intrinsically defined as cotangent vector. To define the gradient
$\nabla g(x)$ of a function $g:\R^d\to\R$ (which is a tangent
vector), the norm comes into play via the formula $\nabla
g(x):=\mathcal L^*(Dg(x))$. Notice that the gradient can be
characterized without invoking the duality map: first of all one
evaluates the slope
\begin{equation}
\label{eq:15}
  |\nabla g|(x):=\limsup_{y\to x}\frac{|g(x)-g(y)|}{\|x-y\|}=\|Dg(x)\|_*;
\end{equation}
then one looks for smooth curves $\gamma:(-\delta,\delta)\to\R^d$ such that
\begin{equation}\label{eq:14}
  \begin{gathered}
  \gamma_0=x,\quad
  \frac \d{\d t}g(\gamma_t)\restr{t=0}
  =\|\dot \gamma_0\|^2=|\nabla g|^2(x).\\
  \text{In this case}\quad
  \nabla g(x)=\dot\gamma_0\quad\text{and}\quad
  |\nabla g|(x)=\|\nabla g(x)\|.
\end{gathered}
\end{equation}
Now, given two smooth functions $f,\,g$, the real number $Df(\nabla
g)(x)$ is well defined as the application of the cotangent vector
$Df(x)$ to the tangent vector $\nabla g(x)$.

What we want to point out, is that there are in principle and in
practice two very different ways of obtaining $Df(\nabla g)(x)$ from
a derivation. The first one, maybe more conventional, is the ``horizontal
derivative'':
\begin{displaymath}
Df(\nabla g)(x)= Df(\dot \gamma_0)=\lim_{t\to
  0}\frac{f(\gamma(t))-f(\gamma_0)}t,\qquad
\text{where $\gamma$ is any curve as in \eqref{eq:14}.}
\end{displaymath}
The second one is the ``vertical derivative'', where we consider
perturbations of the slope
\begin{displaymath}
Df(\nabla g)(x)=\lim_{\eps\to0}\frac{\frac12\,\|\nabla(g+\eps
  f)\|^2(x)-\frac12\|\nabla g\|^2(x)}{\eps}.
\end{displaymath}
It coincides with the previous quantity thanks to the ``dual''
representation \eqref{eq:15}.\fr}\end{example}

We emphasize that this relation between horizontal and vertical
derivation holds in a purely metric setting: compare the statement
of the example with that of Lemma~\ref{le:horver} below (the plan
$\ppi$ playing the role of a curve $\gamma$ as in \eqref{eq:14},
moving points in the direction of $-\nabla g$).

For $\gamma\in \AC2{[0,1]}X $ we set
\begin{equation}\label{eq:moregenerald}
E_t(\gamma):=\sqrt{t\int_0^t|\dot\gamma_s|^2\,\d s}.
\end{equation}
Notice that $E_t(\gamma)$ reduces to $\sfd(\gamma_0,\gamma_t)$ if
$\gamma\in\geo(X)$. In the sequel it is tacitly understood that the
undetermined ratios of the form
$$
\frac{f(\gamma_t)-f(\gamma_0)}{E_t(\gamma)}
$$
are set equal to 0 whenever $E_t(\gamma)=0$, i.e.~$\gamma$ is
constant in $[0,t]$.

Recall that he notion of negligible collections of curves in
$\AC2{[0,1]}X$ has been introduced in Definition~\ref{def:testplan}.

\begin{lemma}\label{le:perhorver}
Let $f:X\to\overline{\R}$ be a Borel function, Sobolev on almost
every curve, such that $\weakgrad f\in L^2(X,\mm)$, and let $\ppi$
be a test plan. Then
\begin{equation}\label{eq:aprioriEt}
\limsup_{t\downarrow
0}\int\left|\frac{f(\gamma_t)-f(\gamma_0)}{E_t(\gamma)}\right|^2\,\d\ppi(\gamma)\leq
\int\weakgrad f^2(\gamma_0)\,\d\ppi(\gamma).
\end{equation}
In particular, assume that  $\ppi\in\gopt(\mu,\nu)$ with
$\mu,\,\nu\ll\mm$ with bounded densities, $\nu$ with bounded
support, $\mu\geq c\mm$ for some $c>0$ and let $\varphi$ be a
Kantorovich potential relative to it. Then it holds
\begin{equation}
\label{eq:limitepot}
\lim_{t\downarrow0}\frac{\varphi(\gamma_0)-\varphi(\gamma_t)}
{E_t(\gamma)}=\weakgrad\varphi(\gamma_0) \qquad\text{in
$L^2(\geo(X),\ppi)$.}
\end{equation}
\end{lemma}
\begin{proof}
For any $t\in(0,1)$ and $\ppi$-a.e.~$\gamma$ it holds
\begin{equation}\label{eq:BtDt1}
\begin{split}
\left|\frac{f(\gamma_{t})-f(\gamma_0)}{E_t(\gamma)}\right|^2\leq\frac{\left(\int_0^t\weakgrad
f(\gamma_s)|\dot\gamma_s|\,\d s\right)^2}{E_t^2(\gamma)}\leq \frac
1t\int_0^t \weakgrad f^2(\gamma_s)\,\d s.
\end{split}
\end{equation}
Hence
\[
\int\left|\frac{f(\gamma_{t})-f(\gamma_0)}{E_t(\gamma)}\right|^2\,\d\ppi(\gamma)\leq
\frac1t\iint_0^t\weakgrad f^2(\gamma_s)\,\d s\,\d\ppi(\gamma)=\int
\left(\frac1t\int_0^t\rho_s\,\d s\right)\weakgrad f^2\,\d\mm,
\]
where $\rho_s$ is the density of $(\e_s)_\sharp\ppi$. Now notice
that $\rho_t\mm\to\rho_0\mm$ as $t\downarrow 0$ in duality with
continuous and bounded functions and that
$\sup_t\|\rho_t\|_\infty<\infty$. Hence $\rho_t\to\rho_0$ weakly$^*$
in $L^\infty(X,\mm)$ and the conclusion follows from the fact that
$\weakgrad f^2\in L^1(X,\mm)$.

The second part of the statement follows by \cite[Theorem
10.3]{Ambrosio-Gigli-Savare11}, Proposition~\ref{prop:brenierbip},
and the identity $E_t(\gamma)=\sfd(\gamma_0,\gamma_t)$ since $\ppi$
in this case is concentrated on $\geo(X)$.
\end{proof}

\begin{lemma}[Horizontal and vertical derivatives]\label{le:horver}
Let $f,\,g:X\to\overline{\R}$ be Borel functions, Sobolev on almost
every curve, such that both $\weakgrad f$ and $\weakgrad g$ belong
to $L^2(X,\mm)$, and let $\ppi$ be a test plan. Assume that
\begin{equation}
\label{eq:asshorvert}
\lim_{t\downarrow0}\frac{g(\gamma_0)-g(\gamma_t)}{E_t(\gamma)}=\lim_{t\downarrow0}\frac{E_t(\gamma)}t=\weakgrad
g(\gamma_0)\qquad\text{in $L^2\bigl(\AC2{[0,1]}X ,\ppi\bigr)$.}
\end{equation}
Then
\begin{equation}\label{eq:derhorvert}
\liminf_{t\downarrow0}\int\frac{f(\gamma_{t})-f(\gamma_0)}{t}
\,\d\ppi(\gamma)\geq\limsup_{\eps\downarrow 0}\int\frac{\weakgrad
g^2(\gamma_0)-\weakgrad{(g+\eps
f)}^2(\gamma_0)}{2\eps}\,\d\ppi(\gamma).
\end{equation}
\end{lemma}
\begin{proof}
Define functions $F_t,\,G_t:\AC2{[0,1]}X \to\R$ by
\[
F_t(\gamma):=\frac{f(\gamma_0)-f(\gamma_t)}{E_t(\gamma)},\qquad
G_t(\gamma):=\frac{g(\gamma_0)-g(\gamma_t)}{E_t(\gamma)}.
\]
By \eqref{eq:asshorvert} we get
\begin{equation}
\label{eq:uguale} \lim_{t\downarrow 0}\int
G_t^2\,\d\ppi=\int\weakgrad g^2(\gamma_0)\,\d\ppi(\gamma).
\end{equation}
Applying Lemma~\ref{le:perhorver} to the function $g+\eps f$ we
obtain
\begin{equation}
\label{eq:loscambio}
\begin{split}
\int\weakgrad{(g+\eps f)}^2(\gamma_0)\,\d\ppi(\gamma)&\geq
\limsup_{t\downarrow 0}\int\left|\frac{(g+\eps f)(\gamma_0)-(g+\eps f)(\gamma_{t})}
{E_t(\gamma)}\right|^2\,\d\ppi(\gamma)\\
&\geq \limsup_{t\downarrow 0}\int \bigl(G_t^2(\gamma)+2\eps
G_tF_t\bigr) \,\d\ppi(\gamma).
\end{split}
\end{equation}
Subtracting this inequality from \eqref{eq:uguale} we get
\[
\frac12\int\frac{\weakgrad g^2(\gamma_0)-\weakgrad{(g+\eps
f)}^2(\gamma_0)}\eps  \,\d\ppi(\gamma)\leq\liminf_{t\downarrow
0}-\int G_t(\gamma)F_t(\gamma)\,\d\ppi(\gamma).
\]
By assumption, we know that $\|G_t-E_t/t\|_2\to 0$ as
$t\downarrow0$. Also, by Lemma~\ref{le:perhorver}, we have
$\sup_t\|F_t\|_2<\infty$. Thus it holds
\[
\liminf_{t\downarrow 0}-\int
G_t(\gamma)F_t(\gamma)\,\d\ppi(\gamma)=\liminf_{t\downarrow
0}-\int\frac{E_t(\gamma)}tF_t(\gamma)\,\d\ppi(\gamma)=\liminf_{t\downarrow
0}\int\frac{f(\gamma_t)-f(\gamma_0)}{t}\,\d\ppi(\gamma).
\]
\end{proof}

Before turning to the proof of the estimate of the derivative of the
entropy along a geodesic, we need two more lemmas.

\begin{lemma}\label{le:apprentr}
Let $(X,\sfd,\mm)$ be a metric measure space, $\ppi$ a test plan and
$\chi_n:X\to[0,1]$ monotonically convergent to $1$. Define the plans
 $\ppi^n:=c_n\,(\chi_n\circ\rme_0)\,\ppi$, where $c_n$ is the
normalizing constant. Then
\[
\lim_{n\to\infty}\entr{(\e_t)_\sharp\ppi^n}\mm=\entr{(\e_t)_\sharp\ppi}\mm
\qquad\forall t\in [0,1].
\]
\end{lemma}
\begin{proof} If $\rho_{n,t}$ are the densities w.r.t.\ $\mm$ of
$(\e_t)_\sharp(\chi_n\circ\rme_0\,\ppi_n)$, by monotone convergence
we have
$\int\rho_{n,t}\log\rho_{n,t}\,\d\mm\to\entr{(\e_t)_\sharp\ppi}\mm$.
Since $c_n\downarrow 1$ the thesis follows.
\end{proof}

\begin{lemma}\label{le:chains}
Let $f,\,g:X\to\R$ be Sobolev functions along a.e.~curve, let $J$ be
an interval containing $g(X)$ and let $\phi:J\to\R$ be
nondecreasing, Lipschitz and $C^1$. Then
\[
\lim_{\eps\downarrow 0}\frac{\weakgrad{(f+\eps\phi(g))}^2-\weakgrad
f^2}{\eps} =\phi'(g)\lim_{\eps\downarrow 0}\frac{\weakgrad{(f+\eps
g)}^2-\weakgrad f^2}{\eps}\qquad\text{$\mm$-a.e.~in $X$.}
\]
Similarly, under the same assumptions on $\phi$, if $\C$ is a
quadratic form and $\weakgrad{f}\in L^2(X,\mm)$, it holds
\[
\lim_{\eps\downarrow 0}\int\frac{\weakgrad{(\phi(g)+\eps
f)}^2-\weakgrad{\phi(g)}^2}{\eps}\,\d\mm=\int
\phi'(g)\lim_{\eps\downarrow 0}\frac{\weakgrad{(g+\eps
f)}^2-\weakgrad g^2}{\eps}\,\d\mm.
\]
\end{lemma}
\begin{proof} Let us consider the first equality. Notice that it is
invariant under addition of constants to $\phi$ and multiplication
of $\phi$ by positive constants, hence if $\phi$ is affine the
thesis is obvious. In addition, since $\weakgrad{h}=0$ $\mm$-a.e.~in
all level sets $h^{-1}(c)$, the formula holds $\mm$-a.e.~on any
level set of $g$. Then, by locality, the formula holds if $\phi$ is
countably piecewise affine, i.e. if there is a partition of $J$ in
countably many intervals where $\phi$ is affine. In the general
case, thanks to the $C^1$ regularity of $\phi$, for any $\delta>0$
we can find a countably piecewise affine $\phi_\delta$ such that
$\|(\phi-\phi_\delta)'\|_\infty<\delta$ and use the estimate
\[
\big|\weakgrad{(f+\eps\phi(g))}-\weakgrad{(f+\eps\phi_\delta(g))}\big|
\leq\eps \weakgrad{(\phi-\phi_\delta)(g)}\leq\eps\delta\weakgrad{g}
\]
to conclude the proof of the first equality.

The second integral equality immediately follows by the integrating
the first one, pulling the limit out of the integral in the left
hand side using the dominated convergence theorem and finally using
the identity
\begin{equation}\label{eq:shift}
Q(u+\eps v)-Q(u)-\eps^2 Q(v)=Q(v+\eps u)-Q(v)-\eps^2Q(u)
\end{equation}
(with $u=\phi(g)$, $v=f$) satisfied by all quadratic forms $Q$.
\end{proof}
We are finally ready to prove the main result of this section.

\begin{theorem}[Derivative of the entropy along a geodesic]\label{thm:derentr}
Let $(X,\sfd,\mm)\in\X$ be a strong $CD(K,\infty)$ space and let
$\sigma_0,\,\sigma_1$ be probability densities. Assume that
$\sigma_1$ is bounded with bounded support as in \eqref{eq:9}, and
let $\varphi$ be a Kantorovich potential from $\sigma_0\mm$ to
$\sigma_1\mm$. Then:
\begin{itemize}
\item[(a)] if $\varphi\in L^2(X,\mm)$, $\log\sigma_0\in D(\C)$ and
\begin{equation}\label{eq:boundrhoo0}
\rho_0(x)\leq c\,\rme^{-9K^-(\sfD(x)-C)^2}\qquad\text{whenever
$\sfD(x):=\mathrm{dist}(x,\supp\rho_1)>R$ for some $c,\,C,\,R$,}
\end{equation}
then for any $\ppi\in\gopt(\sigma_0\mm,\sigma_1\mm)$ as in
Definition~\ref{def:strongcd} and $\mu_t=(\e_t)_\sharp\ppi$ it holds
\begin{equation}
\label{eq:derentropia1} \liminf_{t\downarrow
0}\frac{\ent{\mu_t}-\ent{\mu_0}}{t}\geq
\frac{\C(\varphi)-\C(\varphi+\eps\sigma_0)}{\eps}\qquad\forall\eps>0;
\end{equation}
\item[(b)] if $0<c\leq\sigma_0\leq C<\infty$ for some constants
$c,\,C$, then
\begin{equation}
\label{eq:derentropia}
\ent{\sigma_1\mm}-\ent{\sigma_0\mm}-\frac{K}{2}W_2^2(\sigma_0\mm,\sigma_1\mm)\geq
\frac{\C(\varphi)-\C(\varphi+\eps\sigma_0)}{\eps}\qquad\forall\eps>0.
\end{equation}
\end{itemize}
\end{theorem}
\begin{proof} (a) Since $\log\sigma_0\in D(\C)$ and $\sigma_0$ is
bounded we have $\C(\sigma_0)<\infty$. By
Proposition~\ref{prop:testplan} we get that $({\rm
restr}_0^{1/3})_\sharp\ppi$ is a test plan.

Now observe that the convexity of $z\mapsto z\log z$ gives
\begin{equation}
\label{eq:fine1}
\begin{split}
\frac{\ent{\mu_t}-\ent{\mu_0}}t&\geq \int\log\sigma_0
\frac{\sigma_t-\sigma_0}t\,\d\mm=\int\frac{\log(\sigma_0\circ
\e_t)-\log(\sigma_0\circ \e_0)}t \,\d\ppi.
\end{split}
\end{equation}
Here we make the fundamental use of Lemma~\ref{le:horver}: take
$f:=\log\sigma_0$, $g:=\varphi$ and notice that thanks to
Proposition~\ref{prop:brenierbip} and the second part of
Lemma~\ref{le:perhorver} applied with $\ppi$, the assumptions of
Lemma~\ref{le:horver} are satisfied by $\ppi$. Thus we have
\begin{eqnarray}
\label{eq:cena} \liminf_{t\downarrow0}\int\frac{\log(\sigma_0\circ
e_t)-\log(\sigma_0\circ e_0)}t \,\d\ppi
&\geq&\int\frac{\weakgrad{\varphi}^2(\gamma_0)-\weakgrad{(\varphi+\eps\log\sigma_0)}^2(\gamma_0)}{2\eps}\,
\d\ppi(\gamma)\nonumber\\
&=&\int\frac{\weakgrad{\varphi}^2-\weakgrad{(\varphi+\eps\log\sigma_0)}^2}{2\eps}\sigma_0\,\d\mm.
\end{eqnarray}
Since $\log\sigma_0\in D(\C)$, the integrand in the r.h.s. of
\eqref{eq:cena} is dominated, so that Lemma~\ref{le:chains} yields
\begin{eqnarray*}
\lim_{\eps\downarrow0}\int\frac{\weakgrad{\varphi}^2-\weakgrad{(\varphi+\eps\log\sigma_0)}^2}
{\eps}\sigma_0\,\d\mm
&=&\int\lim_{\eps\downarrow0}\frac{\weakgrad{\varphi}^2-\weakgrad{(\varphi+\eps\log\sigma_0))}^2}{\eps}
\sigma_0\,\d\mm\\
&=&\int\lim_{\eps\downarrow0}\frac{\weakgrad{\varphi}^2-\weakgrad{(\varphi+\eps\sigma_0)}^2}{\eps}\,\d\mm\\
&=&\lim_{\eps\downarrow0}\int\frac{\weakgrad{\varphi}^2-\weakgrad{(\varphi+\eps\sigma_0)}^2}{\eps}\,\d\mm\\
&=&2\lim_{\eps\downarrow0}\frac{\C(\varphi)-\C(\varphi+\eps\sigma_0)}{\eps}.
\end{eqnarray*}
Now the convexity of $\C$ together with \eqref{eq:fine1} and
\eqref{eq:cena} proves \eqref{eq:derentropia1}.

\noindent (b) Observe that by Proposition~\ref{prop:brenierbip} we
know that $\int\weakgrad\varphi^2\,\d\mu_0<\infty$, so that the
lower bound $\mu_0\geq c\mm$ gives $\weakgrad\varphi\in L^2(X,\mm)$
and the statement makes sense. We can also assume
$\C(\sigma_0)<\infty$, indeed if not the inequality
$\weakgrad{(\varphi+\eps\sigma_0)}\geq\eps\weakgrad{\sigma_0}-\weakgrad\varphi$
implies $\C(\varphi+\eps\sigma_0)=\infty$ and there is nothing to
prove.

Let $\sfD(x):=\sfd(x,\supp\sigma_1)$ and
$h_n:[0,\infty)\to[0,\infty)$ be given by
\[
h_n(r):=\rme^{-9K^-((r-n)^+)^2}.
\]
Define the cut-off functions $\nchi_n(x):=h_n(\sfD(x))$, and notice
that since the $h_n$'s are equi-Lipschitz, so are the $\nchi_n$'s.

Notice that $\nchi_n\uparrow 1$ in $X$ as $n\to\infty$. Define
$\ppi^n:=c_n\,(\nchi_n\circ\rme_0)\,\ppi$,
$\ppi\in\gopt(\sigma_0\mm,\sigma_1\mm)$ as in
Definition~\ref{def:strongcd} and $c_n\downarrow 1$ being the
normalizing constant, and
$\mu^n_t:=(\e_t)_\sharp\ppi^n=\sigma^n_t\mm$.

We claim that $\C(\varphi+\eps\sigma_0^n)$ converges to
$\C(\varphi+\eps\sigma_0)$ as $n\to\infty$. To prove the claim, let
$L$ be a uniform bound on the Lipschitz constants of the $\nchi^n$
and notice that the inequality \eqref{eq:weakleib} yields
$\weakgrad{(\nchi_n\sigma_0)}\leq\nchi_n\weakgrad{\sigma_0}
+\sigma_0\weakgrad{\nchi_n}\leq\weakgrad{\sigma_0}+L\|\sigma\|_\infty$,
so that the sequence $(\weakgrad{(\nchi_n\sigma_0)})$ is dominated
in $L^2(X,\mm)$. Now just observe that by the locality principle
\eqref{eq:weaklocal} we have
$\weakgrad{\sigma_0}=\weakgrad{(\nchi_n\sigma_0)}$ $\mm$-a.e.~in
$\{\nchi_n=1\}$.

Taking the previous claim into account, by Lemma~\ref{le:apprentr}
we have that $\entr{\mu^n_t}\mm\to\entr{\mu_t}\mm$ for any
$t\in[0,1]$, so that
\begin{eqnarray*}
&&\ent{\sigma_1\mm}-\ent{\sigma_0\mm}-\frac{K}{2}W_2^2(\sigma_0\mm,\sigma_1\mm)\\&=&
\lim_{n\to\infty}
\ent{\sigma_1^n\mm}-\ent{\sigma^n_0\mm}-\frac{K}{2}W_2^2(\sigma^n_0\mm,\sigma^n_1\mm)\\
&\geq&\liminf_{n\to\infty}\liminf_{t\downarrow
0}\frac{\ent{\sigma_t^n\mm}-\ent{\sigma^n_0\mm}}{t}\geq
\liminf_{n\to\infty}\frac{\C(\varphi)-\C(\varphi+\eps\sigma^n_0)}{\eps}\\
&=&\frac{\C(\varphi)-\C(\varphi+\eps\sigma_0)}{\eps}
\end{eqnarray*}
for all $\eps>0$, provided that statement (a) is applicable to
$(\sigma^n_0\mm,\sigma^n_1\mm)$.

To conclude, we prove that $\sigma^n_0$ satisfies the assumptions
made in (a). Indeed, \eqref{eq:boundrhoo0} is satisfied by
construction (with $c=2\|\sigma_0\|_\infty$, $C=R=n$), thanks to
\[
\sigma^n_0(x)=c_n\,\nchi_n(x)\,\sigma_0(x)\leq 2\|\sigma_0\|_\infty
\,\rme^{-9K^-(\sfD(x)-n)^2}\qquad\text{whenever $D(x)\geq n$},
\]
for $n$ large enough to ensure $c_n\leq 2$. In addition, the
inequality $|\nabla\log\nchi_n|(x)\leq 18 K^-\sfD(x)$ gives
\[
\weakgrad{\log\sigma^n_0}(x)\leq\weakgrad{\log\nchi_n}(x)
+\weakgrad{\log\sigma_0}(x)\leq 18K^- \sfD(x)+\frac{1}{\tilde
c}\weakgrad{\sigma_0}(x),
\]
so that $\weakgrad{\log\sigma^n_0}\in D(\C)$.
\end{proof}

\subsection{Quadratic Cheeger's energies}
\label{sec:quadratic}

Fix a metric measure space $(X,\sfd\,\mm)\in \X$; without assuming
any curvature bound, in this section we apply the tools obtained in
Lemma~\ref{le:horver} to derive useful locality and structural
properties on the Cheeger energy in the distinguished case when $\C$
is a quadratic form on $L^2(X,\mm)$. Since $\C$ is $2$-homogeneous
and convex, this property is easily see to be equivalent to the
parallelogram identity (see for instance
\cite[Proposition~11.9]{DalMaso93})
\begin{equation}\label{eq:18}
      \C(f+g)+\C(f-g)=2\C(f)+2\C(g)\quad\forevery\ f,\,g\in
      L^2(X,\mm).
\end{equation}
If this is the case we will denote by $\mathcal E$ the associated
\emph{Dirichlet} form with domain $D(\mathcal
E):=W^{1,2}(X,\sfd,\mm)$, i.e.~$\mathcal E:[D(\mathcal E)]^2\to \R$
is the unique bilinear symmetric form satisfying (see e.g.\
\cite[Prop.~11.9]{DalMaso93})
\[
\mathcal E(f,f)=2\C(f)\qquad\forall f\in W^{1,2}(X,\sfd,\mm).
\]
Recall that $W^{1,2}(X,\sfd,\mm)=D(\C)\cap L^2(X,\mm)$.

We will occasionally use this density criterion in the theory of
linear semigroups.

\begin{lemma}[Density of invariant sets]\label{lem:invariant}
Let ${\mathcal E}$ be the bilinear form associated to a nonnegative
and lower semicontinuous quadratic form $Q$ in a Hilbert space $H$,
and let ${\mathsf S}_t$ be the associated evolution semigroup. If a
subspace $V\subset D(Q)$ is dense for the norm of $H$ and ${\mathsf
S}_t$-invariant, then $V$ is also dense in $D(Q)$ for the Hilbert
norm $\sqrt{{\mathcal E}(u,u)+(u,u)^2}$.
\end{lemma}
\begin{proof} If $u\in D({\mathcal E})$ satisfies ${\mathcal E}(u,w)+(u,w)=0$ for all $w\in V$,
we can choose $w={\mathsf S}_tv$, $v\in V,$ and use the fact that
${\mathsf S}_t$ is self-adjoint to get
$$
{\mathcal E}({\mathsf S}_tu,v)+({\mathsf S}_tu,v)={\mathcal
E}(u,{\mathsf S}_tv)+(u,{\mathsf S}_tv)=0\qquad\forall v\in
V,\,\,t>0.
$$
Since ${\mathsf S}_tu$ belongs to the domain of the infinitesimal
generator of ${\mathsf S}_tu$, $v\mapsto {\mathcal E}({\mathsf
S}_tu,v)$ is continuous in $D({\mathcal E})$ for the $H$ norm, hence
${\mathcal E}({\mathsf S}_tu,v)+({\mathsf S}_tu,v)=0$ for all $v\in
D({\mathcal E})$ and $t>0$. Choosing $v={\mathsf S}_tu$ and letting
$t\downarrow 0$ gives $u=0$.
%
%To show the last statement, let us first notice that
%$V:=\bigcup_{t>0}\sfS_t(H)$ is dense in $H$ and invariant, thus
%dense in $D(Q)$. If $H_0$ is a countable dense subset of $H$, it is
%immediate to check using \eqref{eq:ultraEVI} that
%$V_0:=\bigcup_{t\in \Q\cap(0,\infty)}\sfS_t(H_0)$ is countable and
%dense in $V$ with respect to the norm of $D(Q)$.
\end{proof}

\begin{proposition}[Properties of $W^{1,2}(X,\sfd,\mm)$]\label{prop:lipdense}
If $\C$ is quadratic in $L^2(X,\mm)$ then $W^{1,2}(X,\sfd,\mm)$
endowed with the norm $\sqrt{\|f\|_2^2+\mathcal E(f,f)}$ is a
separable Hilbert space and Lipschitz functions are dense.
\end{proposition}
\begin{proof} We already know that $W^{1,2}(X,\sfd,\mm)$ is complete \cite[Remark
4.7]{Ambrosio-Gigli-Savare11} and therefore it is a Hilbert space
since $\C$ is quadratic. In particular if $f_n,f\in
W^{1,2}(X,\sfd,\mm)$ satisfy $\|f_n-f\|_2\to 0$ and
$\C(f_n)\to\C(f)$ then $f_n\to f$ strongly in $W^{1,2}(X,\sfd,\mm)$.
In fact, by the parallelogram identity and the $L^2(X,\mm)$-lower
semicontinuity of $\C$
\begin{align*}
  \limsup_{n\to\infty}\C(f-f_n)&=
  \limsup_{n\to\infty}\Big(2\C(f)+2\C(f_n)-\C(f+f_n)\Big)
  \\&=4\C(f)-\liminf_{n\to\infty} \C(f+f_n)\le 4\C(f)-\C(2f)=0.
\end{align*}
The density of Lipschitz function thus follows by
\eqref{eq:nontrivialatall}. The separability of
$W^{1,2}(X,\sfd,\mm)$ follows considering the invariant set
$V:=\bigcup_{t>0}\heatl_tL^2(X,\mm)$, which is a subspace thanks to
the semigroup property, dense in $W^{1,2}(X,\sfd,\mm)$ thanks to
Lemma~\ref{lem:invariant}. Using \eqref{eq:ultraEVI} and the
separability of $L^2(X,\mm)$ it is easy to check that $V$ is
separable with respect to the $W^{1,2}(X,\sfd,\mm)$ norm, whence the
separability of $W^{1,2}(X,\sfd,\mm)$ follows.
\end{proof}

The terminology Dirichlet form, borrowed from \cite{Fukushima80}, is
justified by the fact that $\mathcal E$ is closed (because $\C$ is
$L^2(X,\mm)$-lower semicontinuous) and Markovian (by the chain rule
\eqref{eq:chainrule}). Good references on the theory of Dirichlet
forms are\cite{Fukushima80,Fukushima-Oshima-Takeda11} for locally
compact spaces and \cite{Rockner92}. The second reference (but see
also \cite[A.4]{Fukushima-Oshima-Takeda11}), where the theory is
extended to infinite-dimensional spaces and even to some classes of
non-symmetric forms is more appropriate for us, since we are not
assuming local compactness of our spaces.

In this section we analyze the basic properties of this form and
relate the energy measure $[f]$ appearing in Fukushima's theory, a
kind of localization of ${\mathcal E}$, to $\weakgrad{f}$. Recall
that for any $f\in D(\mathcal E)\cap L^\infty(X,\mm)$ the energy
measure $[f]$ (notice the factor $1/2$ with respect to
  the definition of \cite[(3.2.14)]{Fukushima-Oshima-Takeda11})
is defined by
\begin{equation}\label{eq:fuku1}
\mathcal[f](\varphi):={\color{red}-}
{\mathcal E}(f,f\varphi)- {\mathcal E}(\frac
{f^2} 2,\varphi)\quad \text{for any }\varphi\in D(\mathcal E)\cap
L^\infty(X,\mm).
\end{equation}
We shall prove in Theorem~\ref{thm:tangente} that
$[f]=\weakgrad{f}^2\mm$.
The first step concerns locality, which is not difficult to prove in
our setting:
\begin{proposition}\label{prop:stronglocal}
$\mathcal E$ is strongly local:
\begin{equation}
\label{eq:strongloc} \text{$f,\,g\in D(\mathcal E)$, $g$ constant on
$\{f\neq 0\}$}\qquad\Rightarrow\qquad\mathcal E(f,g)=0.
\end{equation}
\end{proposition}
\begin{proof}
By definition we have
\begin{equation}
\label{eq:polar}
2\mathcal E(f,g)=\int\weakgrad{(f+g)}^2-\weakgrad f^2-\weakgrad g^2\,\d\mm.
\end{equation}
By the assumption on $g$, locality \eqref{eq:weaklocal} and chain
rule \eqref{eq:chainrule} we get that $\mm$-a.e.~on $\{f\neq 0\}$.
$\weakgrad{(f+g)}=\weakgrad{f}$ and $\weakgrad{g}$ vanishes. On the
other hand, $\mm$-a.e.~on $X\setminus \{f\neq 0\}$ we have
$\weakgrad{(f+g)}=\weakgrad{g}$ and $\weakgrad{f}$ vanishes.
\end{proof}
The identification of $[f]$ with $\weakgrad{f}^2\mm$ requires a
deeper understanding of the Leibnitz formula in our context. Our
goal is to prove the existence of a bilinear symmetric map from
$[D(\C)]^2$ to $L^1(X,\mm)$, that we will denote by $\nabla
f\cdot\nabla g$, which gives a pointwise representation of the
Dirichlet form, in the sense that
\[
\mathcal E(f,g)=\int\nabla f\cdot\nabla g\,\d\mm\qquad\forall
f,\,g\in W^{1,2}(X,\sfd,\mm).
\]
In the spirit of Cheeger's theory \cite{Cheeger00}, we may think to
differentials of Lipschitz functions as $L^\infty$ sections of the
cotangent bundle to $X$. Then, if we keep $f$ fixed, the map
$g\mapsto\nabla f(x)\cdot\nabla g(x)$ may be interpreted as the
action of $df$ on the section induced by $g$, and we may use this
map to define the gradient of $f$.

We remark that $\nabla f\cdot \nabla g$ could be considered as the
``carr\'e du champ'' operator $\Gamma(f,g)$ of $\Gamma$-calculus in
this context, widely used in the study of diffusion semigroups. We
adopted the appealing notation $\nabla f\cdot\nabla g$ since in our
approach this quantity is directly obtained by a pointwise
``vertical" differential of the squared weak upper gradient (like
when one tries to recover a scalar product from the squared norm).
As a byproduct, we will show that the Laplacian satisfies a suitable
formulation of the diffusion condition \cite[1.3]{Bakry06}, so that
useful estimates can be derived by the so called $\Gamma$-calculus.
An example of application will be given in Section \ref{sub8}.

Let $f,\,g\in D(\C)$ and notice that the inequality
\[
\weakgrad{((1-\lambda)f+\lambda g)}\leq(1-\lambda)\weakgrad
f+\lambda\weakgrad g\qquad\text{$\mm$-a.e.~in $X$,}
\]
valid for any $\lambda\in[0,1]$, immediately yields that
$\eps\mapsto \weakgrad{(f+\eps g)}^2$ satisfies the usual convexity
inequality $\mm$-a.e.\ and $\eps\mapsto\eps^{-1}[\weakgrad{(f+\eps
g)}^2-\weakgrad f^2]$ is nondecreasing $\mm$-a.e.~in  $\R\setminus
\{0\}$, in the sense that
\[
\frac{\weakgrad{(f+\eps' g)}^2-\weakgrad f^2}{\eps'}\leq
\frac{\weakgrad{(f+\eps g)}^2-\weakgrad
f^2}{\eps}\qquad\text{$\mm$-a.e.~in $X$ for $\eps',\eps\in
\R\setminus \{0\},\ \eps'\leq\eps$.}
\]
\begin{definition}[The function $\nabla f\cdot\nabla g$]
For $f,\,g\in D(\C)$ we define $\nabla f\cdot\nabla g$ as
\begin{equation}
\nabla f\cdot\nabla g:=\lim_{\eps\downarrow
0}\frac{\weakgrad{(f+\eps g)}^2-\weakgrad f^2}{2\eps} \label{eq:26}
\end{equation}
where the limit is understood in $L^1(X,\mm)$.
\end{definition}
Notice that by monotone convergence and the lower bound obtained by
taking a negative $\eps'$ in the previous monotonicity formula, the
limit in \eqref{eq:26} exists in $L^1(X,\mm)$ along any
monotonically decreasing sequence $(\epsilon_i)\subset (0,\infty)$.
This obviously implies existence of the full limit as
$\eps\downarrow 0$; we also have
\begin{equation}
\label{eq:doteweakgrad} \nabla f\cdot\nabla f=\weakgrad
f^2\qquad\text{$\mm$-a.e.~in $X$.}
\end{equation}
Notice also that we don't know, yet, whether $(f,g)\mapsto\nabla
f\cdot\nabla g$ is symmetric, or bilinear, the only trivial
consequence of the definition being the positive homogeneity w.r.t.
$g$. Now we examine the continuity properties of $\nabla
f\cdot\nabla g$ with respect to $g$.

\begin{proposition}\label{eq:daunaparte}
For $f,\,g,\,\tilde{g}\in D(\C)$ it holds
\[
\big|\nabla f\cdot\nabla g-\nabla f\cdot\nabla \tilde g\big|\leq
\weakgrad f\weakgrad{(g-\tilde g)}\qquad\text{$\mm$-a.e.~in $X$}
\]
and, in particular, $\nabla f\cdot\nabla g\in L^1(X,\mm)$ and
\begin{equation} \label{eq:boundpuntuale}
|\nabla f\cdot\nabla
g|\leq\weakgrad f\weakgrad g\qquad\text{$\mm$-a.e.~in $X$.}
\end{equation}
\end{proposition}
\begin{proof} It follows from
\begin{eqnarray*}
&&\Big|\big(\weakgrad{(f+\eps g)}^2-\weakgrad
f^2\big)-\big(\weakgrad{(f+\eps \tilde g)}^2-\weakgrad
f^2\big)\Big|
=\Bigl|\weakgrad{(f+\eps g)}^2-\weakgrad{(f+\eps \tilde g)}^2\Bigr|\\
&=&\Bigl|\big(\weakgrad{(f+\eps g)}-\weakgrad{(f+\eps \tilde g)}\big)\,
\big(\weakgrad{(f+\eps g)}+\weakgrad{(f+\eps \tilde g)}\big)\Bigr|\\
&\leq&\eps\weakgrad {(g-\tilde g)}\,\big(\weakgrad{(f+\eps
g)}+\weakgrad{(f+\eps \tilde g)}\big),
\end{eqnarray*}
dividing by $\eps$, letting $\eps\downarrow 0$ and using the strong
convergence of $\weakgrad{(f+\eps g)}$ and $\weakgrad{(f+\eps\tilde
g)}$ to $\weakgrad{f}$.
\end{proof}

Observe that the second chain rule given in Lemma~\ref{le:chains}
grants that
\begin{equation}
\label{eq:chainfacile} \int\nabla (\phi\circ f)\cdot\nabla g\,\d\mm
=\int(\phi'\circ f)\,\nabla f\cdot\nabla g\,\d\mm
\end{equation}
for $\phi$ nondecreasing and $C^1$ on an interval containing the
image of $f$.

\begin{proposition}
For any $f,\,g\in D(\C)$ it holds
\begin{equation}
\label{eq:dirlocale}
\mathcal E(f,g)=\int \nabla f\cdot\nabla g\,\d\mm.
\end{equation}
Also, we have
\begin{equation}
\label{eq:giochettodisegno} \nabla f\cdot\nabla
g=-\nabla(-f)\cdot\nabla g=-\nabla f\cdot\nabla(- g)=\nabla
(-f)\cdot\nabla(- g)\qquad\text{$\mm$-a.e.~in $X$.}
\end{equation}
\end{proposition}
\begin{proof}
The equality \eqref{eq:dirlocale} follows replacing $g$ by $\eps g$
into \eqref{eq:polar}, dividing by $\eps$ and letting
$\eps\downarrow 0$. To get \eqref{eq:giochettodisegno} notice that
the $\mm$-a.e.~convexity of $\eps\mapsto \weakgrad{(f+\eps g)}^2(x)$
in $\R$ yields
\begin{equation}\label{eq:opposite}
\nabla f\cdot\nabla g+\nabla f\cdot\nabla(- g)\geq
0\qquad\text{$\mm$-a.e.~in $X$.}
\end{equation}
Since $\mathcal E(f,g)=-\mathcal E(f,-g)$ we can use
\eqref{eq:dirlocale} to obtain that the sum in \eqref{eq:opposite}
is null $\mm$-a.e.~in $X$. We conclude using the identity $\nabla
f\cdot\nabla g=\nabla (-f)\cdot\nabla(-g)$, a trivial consequence of
the definition.
\end{proof}

\begin{lemma}\label{le:planperu}
Let $u\in D(\C)$ be a given bounded function and let $E_t(\gamma)$
be defined as in \eqref{eq:moregenerald}. Then there exists a test
plan $\ppi$ satisfying $(\e_0)_\sharp\ppi=\mm$ and
\begin{equation}
\label{eq:plandiu} \lim_{t\downarrow
0}\frac{E_t}{t}=\lim_{t\downarrow 0}\frac{u\circ\e_0-u\circ
\e_t}{E_t}=\weakgrad u\circ\e_0\qquad\text{in
$L^2\bigl(\AC2{[0,1]}X ,\ppi\bigr)$.}
\end{equation}
\end{lemma}
\begin{proof}
Let $\rho_0:=c\,\rme^u$, where $c$ is the normalization constant, put
$\mu_0:=\rho_0\mm$ and $\rho_t:=\heatl_t(\rho_0)$. Notice that
$\rho_0$ is uniformly bounded away from 0 and $\infty$ and that
$\C(\rho_t)\to\C(\rho_0)$ as $t\downarrow 0$ implies, by the same
Hilbertian argument of Proposition~\ref{prop:lipdense}, strong convergence of
$\weakgrad{\rho_t}$ to $\weakgrad{\rho_0}$ in $L^2(X,\mm)$. Define
the functions $A_t,\,B_t,\,C_t,\,D_t:\AC2{[0,1]}X \to\R$ by (with
the usual convention if $E_t(\gamma)=0$)
\[
\begin{split}
A_t(\gamma)&:=\frac{\log\rho_0(\gamma_0)-\log\rho_0(\gamma_t)}{t}=\frac{u(\gamma_0)-u(\gamma_t)}t,\\
B_t(\gamma)&:=\frac{\log\rho_0(\gamma_0)-\log\rho_0(\gamma_t)}{E_t(\gamma)}=
\frac{u(\gamma_0)-u(\gamma_t)}{E_t(\gamma)},\\
C_t(\gamma)&:=\frac{E_t(\gamma)}{t},\\
D_t(\gamma)&:=\sqrt{\frac1t\int_0^t\frac{\weakgrad{\rho_0}^2(\gamma_s)}{\rho_0^2(\gamma_s)}\,\d
s}=\sqrt{\frac1t\int_0^t\weakgrad{u}^2(\gamma_s)\,\d s.}
\end{split}
\]
Now use \cite{Lisini07} to get the existence of a plan $\ppi\in
\prob{\AC2{[0,1]}X }$ such that $(\e_t)_\sharp\ppi=\mu_t:=\rho_t\mm$
for all $t\in [0,1]$ and
\[
\int|\dot\gamma_t|^2\,\d\ppi(\gamma)=|\dot\mu_t|^2\qquad\text{for
a.e.~$t\in [0,1]$.}
\]
The maximum principle ensures that $\ppi$ is a test plan.

By Lemma~\ref{le:convl2} below we get that $D_t\to\weakgrad
u\circ\e_0$ in $L^2\bigl(\AC2{[0,1]}X ,\ppi\bigr)$. From the second
equality in \eqref{eq:edissrateflow} we have
\begin{equation}
\label{eq:boundc}
\begin{split}
\lim_{t\downarrow 0}\left\|C_t\right\|^2_2&= \lim_{t\downarrow
0}\frac1t\int\int_0^t|\dot\gamma_s|^2\,\d s\,\d\ppi
=\lim_{t\downarrow 0}\frac1t\int_0^t|\dot\mu_s|^2\,\d s\\
&=\lim_{t\downarrow 0}\frac1t
\int_0^t\int\frac{\weakgrad{\rho_s}^2}{\rho_s}\,\d\mm\,\d s
=\int\frac{\weakgrad{\rho_0}^2}{\rho_0}\,\d\mm=\|\weakgrad
u\circ\e_0\|^2_2,
\end{split}
\end{equation}
and from Lemma~\ref{le:perhorver} and \eqref{eq:BtDt1} we know that
\begin{equation}\label{eq:BtDt}
|B_t|\leq D_t\qquad\text{and}\qquad\limsup_{t\downarrow
0}\|B_t\|_2^2\leq\|\weakgrad u\circ\e_0\|^2_2.
\end{equation}
Estimates \eqref{eq:boundc} and \eqref{eq:BtDt} imply
\begin{equation}
\label{eq:spaccata}
\begin{split}
\limsup_{t\downarrow 0}\int A_t\,\d\ppi&=\limsup_{t\downarrow 0}
\int B_tC_t\,\d\ppi\leq\limsup_{t\downarrow 0}\int |B_t|C_t\,\d\ppi\\
&\leq\limsup_{t\downarrow 0}\|B_t\|_2\|C_t\|_2\leq\|\weakgrad
u\circ\e_0\|^2_2.
\end{split}
\end{equation}
Notice that from the convexity of $z\mapsto z\log z$ we have
\[
\int A_t\,\d\ppi=\int\frac{\log\rho_0(\rho_0-\rho_t)}t\,\d\ppi\geq
\frac{\entr{\mu_0}{\mm}-\entr{\mu_t}\mm}t,
\]
and from the first equality in \eqref{eq:edissrateflow} we deduce
\[
\lim_{t\downarrow 0} \frac{\entr{\mu_0}{\mm}-\entr{\mu_t}\mm}t=
\lim_{t\downarrow 0}\frac1t
\int_0^t\int\frac{\weakgrad{\rho_s}^2}{\rho_s}\,\d s\d\mm=
\int\frac{\weakgrad{\rho_0}^2}{\rho_0}\,\d\mm=\|\weakgrad
u\circ\e_0\|^2_2.
\]
Thus from \eqref{eq:spaccata} we deduce that $\int A_t\,\d\ppi$
converges to $\|\weakgrad u\circ\e_0\|^2_2$ as $t\downarrow 0$.
Repeating now \eqref{eq:spaccata} with $\liminf$ we deduce that also
$\|B_t\|^2_2$ converges as $t\downarrow 0$ to $\|\weakgrad
u\circ\e_0\|^2_2$. Now, this convergence, the first inequality in
\eqref{eq:BtDt} and the $L^2$-convergence of $D_t$ to $\weakgrad
u\circ \e_0$ yield the $L^2$-convergence of $|B_t|$ to the same
limit.

Finally, from the fact that the first inequality in
\eqref{eq:spaccata} is an equality, we get that also $B_t$ converges
to $\weakgrad u\circ\e_0$ in $L^2(\ppi)$. Also, since the second
inequality in \eqref{eq:spaccata} is an equality and
\eqref{eq:boundc} holds one can conclude that $C_t$ converges to
$\weakgrad u\circ\e_0$ in $L^2(\ppi)$ as well.

Thus $\ppi$ has all the required properties, except the fact that
$(\e_0)_\sharp\ppi$ is not $\mm$. To conclude, just replace $\ppi$
with $\tilde c\rho_0^{-1}\circ\e_0\ppi$, $\tilde c$ being the
renormalization constant.
\end{proof}

\begin{lemma}\label{le:convl2}
Let $f\in L^1(X,\mm)$ be nonnegative and define
$F_t:\AC2{[0,1]}X \to [0,\infty]$, $t\in [0,1]$, by
\[
F_t(\gamma):=\sqrt{\frac1t\int_0^tf(\gamma_s)\,\d s}\quad t\in
(0,1],\qquad F_0:= \sqrt{f(\gamma_0)}.
\]
Then $F_t\to F_0$ in $L^2\bigl(\AC2{[0,1]}X ,\ppi\bigr)$ as
$t\downarrow0$ for any test plan $\ppi$ whose 2-action
$\int\int_0^1|\dot\gamma_t|^2\,\d t\,\d\ppi(\gamma)$ is finite.
\end{lemma}
\begin{proof}
To prove the thesis it is sufficient to show that $F_t^2\to F^2_0$
in $L^1(\AC2{[0,1]}X ,\ppi)$.

Now assume first that $f$ is Lipschitz. In this case the conclusion
easily follows from the inequality
$|F_t^2(\gamma)-F_0^2(\gamma)|\leq \Lip(f) \tfrac
1t\int_0^t\sfd(\gamma_0,\gamma_s)\,\d s\leq
\Lip(f)\int_0^t|\dot\gamma_s|\,\d s$ and the fact that
$\int\int_0^1|\dot\gamma_t|^2\,\d t\,d\ppi(\gamma)<\infty$. To pass
to the general case, notice that Lipschitz functions are dense in
$L^1(X,\mm)$ and conclude by the continuity estimates
\[
\int\frac1t\left|\int_0^t h(\gamma_s)\,\d s\right|\,\d\ppi
\leq\frac1t\int_0^t\int|h(\gamma_s)|\,\d\ppi\,\d s\leq
C\|h\|_1,\qquad \int|h(\gamma_0)|\,\d\ppi\leq C\|h\|_1,
\]
where $C>0$ satisfies $(\e_t)_\sharp\ppi\leq C\mm$ for any
$t\in[0,1]$.
\end{proof}

\begin{proposition}[Leibnitz formula for nonnegative functions]
Let $f,\,g,\,h\in D(\C)\cap L^\infty(X,\mm)$ with $g,\,h$
nonnegative. Then
\begin{equation}
\label{eq:leibniz}
\mathcal E(f,gh)=\int\nabla
f\cdot\nabla(gh)\,\d\mm=\int h\nabla f\cdot\nabla g+g\nabla
f\cdot\nabla h\,\d\mm.
\end{equation}
\end{proposition}
\begin{proof} Notice first that if $(g_n)$, $(h_n)$ are equibounded and converge
in $W^{1,2}(X,\sfd,\mm)$ to $g$, $h$ respectively then
\eqref{eq:weakleib} ensures that $g_nh_n$ converge to $gh$ strongly
in $W^{1,2}(X,\sfd,\mm)$. Hence, taking
Proposition~\ref{prop:lipdense} and Proposition~\ref{eq:daunaparte}
into account, we can assume with no loss of generality $g,\,h$ to be
bounded, nonnegative and Lipschitz.

Now we apply Lemma~\ref{le:planperu} with $u=f$. The definition of
$\nabla f\cdot\nabla(gh)$ and inequality \eqref{eq:derhorvert} gives
\begin{eqnarray*}
\int\nabla f\cdot\nabla(gh)\,\d\mm&=&
\liminf_{\eps\downarrow0}\int\frac{\weakgrad{(f+\eps
gh)}^2-\weakgrad f^2}{\eps}\,\d\mm\\
&\geq&\limsup_{t\downarrow0}\int\frac{g(\gamma_0)h(\gamma_0)-g(\gamma_t)h(\gamma_t)}t\,\d\ppi(\gamma).
\end{eqnarray*}
Now observe that the convergence
\[
\left|\frac{\big(g(\gamma_t)-g(\gamma_0)\big)\big(h(\gamma_t)-h(\gamma_0))}t\right|
\leq {\Lip(g)\Lip(h)}\frac{\sfd^2(\gamma_0,\gamma_t)}{t}\to0\qquad\textrm{in }L^1(\ppi),
\]
ensures that
\begin{eqnarray*}
&&\limsup_{t\downarrow0}\int\frac{g(\gamma_0)h(\gamma_0)-g(\gamma_t)h(\gamma_t)}t \,\d\ppi(\gamma)
\geq\liminf_{t\downarrow0}\int\frac{g(\gamma_0)h(\gamma_0)-g(\gamma_t)h(\gamma_t)}t\,\d\ppi(\gamma)\\
&&\geq\liminf_{t\downarrow0}\int
g(\gamma_0)\frac{h(\gamma_0)-h(\gamma_t)}t
\,\d\ppi(\gamma)+\liminf_{t\downarrow0}\int
h(\gamma_0)\frac{g(\gamma_0)-g(\gamma_t)}t \,\d\ppi(\gamma).
\end{eqnarray*}
Now applying inequality \eqref{eq:derhorvert} to the plans $(g\circ
\rme_0)\,\ppi$ and $(h\circ\rme_0)\,\ppi$ we get
\[
\begin{split}
\liminf_{t\downarrow0}\int g(\gamma_0)\frac{h(\gamma_0)-h(\gamma_t)}t \,\d\ppi(\gamma)&\geq
\limsup_{\eps\downarrow0}\int g\frac{\weakgrad f^2-\weakgrad{(f-\eps h)}^2}\eps \,\d\mm,\\
\liminf_{t\downarrow0}\int h(\gamma_0)\frac{g(\gamma_0)-g(\gamma_t)}t \,\d\ppi(\gamma)&\geq
\limsup_{\eps\downarrow0}\int h\frac{\weakgrad f^2-\weakgrad{(f-\eps g)}^2}\eps \,\d\mm.
\end{split}
\]
Recalling the convergence in $L^1(X,\mm)$ of the difference
quotients in \eqref{eq:26} we get
\[
\begin{split}
\limsup_{\eps\downarrow0}\int g  \frac{\weakgrad f^2-\weakgrad{(f-\eps h)}^2}\eps \,\d\mm&=
-\int g\nabla f\cdot\nabla(-h)\d\mm=\int g\nabla f\cdot\nabla h\,\d\mm,\\
\limsup_{\eps\downarrow0}\int h \frac{\weakgrad
f^2-\weakgrad{(f-\eps g)}^2}\eps \,\d\mm&= -\int h\nabla
f\cdot\nabla(-g)\d\mm=\int h\nabla f\cdot\nabla g\,\d\mm.
\end{split}
\]
Thus we proved that the inequality $\geq$ always holds in
\eqref{eq:leibniz}. Replacing $f$ with $-f$ and using
\eqref{eq:giochettodisegno} once more we get the opposite one and
the conclusion.
\end{proof}
\begin{theorem}[Leibnitz formula and identification of
  {$[f]$}]\label{thm:tangente}
Let $(X,\sfd,\mm)\in \X$ and let us assume that Cheeger's energy
$\C$ is quadratic in $L^2(X,\mm)$ as in \eqref{eq:18}. Then
\begin{enumerate}[(i)]
\item The map $(f,g)\mapsto\nabla f\cdot\nabla g$ from $[D(\C)]^2$ to
$L^1(X,\mm)$ is bilinear, symmetric and satisfies
\eqref{eq:boundpuntuale}. In particular it is continuous from
$[W^{1,2}(X,\sfd,\mm)]^2$ to $L^1(X,\mm)$.
\item For all $f,\,g\in D(\C)$ it holds
\begin{equation}
  \label{eq:cosimeglio}
  \weakgrad{(f+g)}^2+\weakgrad{(f-g)}^2=2\weakgrad{f}^2+2\weakgrad{g}^2\quad
  \text{$\mm$-a.e.\ in }X.
\end{equation}
In particular $\C$ is a quadratic form in $L^1(X,\mm)$.
\item  The Leibnitz formula \eqref{eq:leibniz} holds with
  equality and with no sign restriction on $g$, $h$.
\item The energy measure
 $[f]$ in \eqref{eq:fuku1} coincides with $\weakgrad{f}^2\mm$.
\end{enumerate}
\end{theorem}
\begin{proof}
The continuity bound follows at once from \eqref{eq:boundpuntuale}.
In order to show symmetry and bilinearity it is sufficient to prove
\eqref{eq:cosimeglio} when $f,\,g\in D(\C)$. In turn, this property
follows if we are able to prove that
\begin{equation}\label{eq:31}
f\mapsto\int h\weakgrad f^2\,\d\mm\quad\text{is quadratic in
$D(\C)$}
\end{equation}
for all $h\in L^\infty(X,\mm)$ nonnegative. Since $D(\C)\cap
L^\infty(X,\mm)$ is weakly$^*$ dense in $L^\infty(X,\mm)$, it is
sufficient to prove this property for nonnegative $h\in D(\C)\cap
L^\infty(X,\mm)$. Pick $f=g\in D(\C)\cap L^\infty(X,\mm)$
nonnegative in \eqref{eq:leibniz} to get
\[
\begin{split}
\int h\weakgrad f^2\,\d\mm& =
-\int f\nabla f\cdot\nabla h\,\d\mm+\int\nabla f\cdot\nabla(fh)\,\d\mm\\
&=-\int \nabla\big(\frac{f^2}2\big)\cdot\nabla h\,\d\mm+\int\nabla
f\cdot\nabla(fh)\,\d\mm\\&= -\mathcal E(\frac{f^2}2,h)+\mathcal
E(f,fh),
\end{split}
\]
having used equation \eqref{eq:chainfacile} in the second equality.
Now, splitting $f$ in positive and negative parts, we can extend the
formula to $D(\C)\cap L^\infty(X,\mm)$, since $\mathcal E$ is
bilinear and the strong locality ensures $\mathcal E(f^+,f^-h)=0$,
$\mathcal E(f^-,f^+h)=0$. Both maps $f\mapsto \mathcal E(f^2/2,h)$
and $f\mapsto \mathcal E(f,fh)$ are immediately seen to be
quadratic, so the same is true for $\int \weakgrad f^2h\,d\mm$. Thus
we proved that the map $\int h\weakgrad f^2\,\d\mm$ is quadratic on
$D(\C)\cap L^\infty(X,\mm)$. The statement for the full domain
$D(\C)$ follows from a simple truncation argument: if
$f^N:=\max\{\min\{f,N\},-N\},\,g^N:=\max\{\min\{g,N\},-N\}\in
L^2(X,\mm)$ are the truncated functions, the chain rule
\eqref{eq:chainrule} gives
$$
\int h\weakgrad{(f_N+g_N)}^2\,\d\mm+\int
h\weakgrad{(f_N-g_N)}^2\,\d\mm\leq 2\int
h\weakgrad{f}^2\,\d\mm+2\int h\weakgrad{g}^2.
$$
which yields in the limit one inequality. A similar argument applied
to $f+g$ and $f-g$ provides the converse inequality and proves
\eqref{eq:31}. Finally, we can use the fact that $h$ is arbitrary to
prove the pointwise formulation \eqref{eq:cosimeglio}.
\end{proof}

The property \eqref{eq:31} shows that if $(X,\sfd,\mm)$ gives raise
to a quadratic Cheeger's energy, also $(X,\sfd,h\mm)$ enjoys the
same property, provided a uniform bound $0<c\le h\le c^{-1}$ is
satisfied (indeed, we can use \cite[Lemma
4.11]{Ambrosio-Gigli-Savare11} to prove that $\weakgrad{f}$ is
independent of $h$). The next result consider the case when $h$ is
the characteristic function of a closed subset of $X$.

\begin{theorem}\label{thm:gradristretti}
Let $(X,\sfd,\mm)\in \X$ and let $Y\subset X$ be a closed set of
positive measure. For $f:Y\to\R$, denote by $\weakgrado f{Y}$ the
minimal weak upper gradient of $f$ calculated in the metric measure
space $(Y,\sfd,\mm(Y)^{-1}\mm\restr{Y})$. Then:
\begin{itemize}
\item[(i)] Let $f: X\to\R$ Borel and Sobolev along a.e.~curve of $X$, and define $g:Y\to\R$ by
$g:=f\restr{Y}$. Then $g$ is Sobolev along a.e.~curve of $Y$ and
$\weakgrad f=\weakgrado{g}{Y}\ $ $\mm$-a.e.~in $Y$.
\item[(ii)] Let $g:Y\to\R$ be Borel, Sobolev along a.e.~curve in $Y$ and such that
${\rm dist}(\supp g,\partial Y)>0$. Define $f:X\to\R$ by $f\restr
Y:=g$ and $f\restr{X\setminus Y}:=0$. Then $f$ is Sobolev along a.e.
curve of $X$ and $\weakgrad f=\nchi_Y\weakgrado {g}Y$ $\mm$-a.e.~in
$X$.
\item[(iii)] If moreover $\C$ is a quadratic form in $L^2(X,\mm)$ according to
\eqref{eq:18} and $\mm(\partial Y)=0$, then
\[
\C_Y(f):=\left\{\begin{array}{ll}
\displaystyle{\int_Y\weakgrado f{Y}^2\,\d\mm_Y}&\textrm{ if }f\textrm{ is Sobolev along a.e.~curve in }Y,\\
+\infty&\textrm{ otherwise},
\end{array}
\right.
\]
is a quadratic form in $L^2(Y,\mm_Y)$.
\end{itemize}
\end{theorem}
\begin{proof} {(i)} The fact that $g$ is Sobolev along a.e.~curve in $Y$ is
obvious, since this class of curves is smaller. It is also obvious
that $\weakgrado gY\leq\weakgrad f\ $ $\mm$-a.e.~in $Y$, so that to
conclude it is sufficient to prove the opposite inequality. Let
$G:X\to[0,\infty]$ be defined by
\begin{equation}
\label{eq:sembrafunzioni}
G(x):=\left\{
\begin{array}{ll}
\weakgrado gY (x)&\qquad\textrm{if }x\in Y,\\
+\infty&\qquad\textrm{ otherwise}.
\end{array}
\right.
\end{equation}
Then it is trivial from the definition that $G$ is a weak upper
gradient for $f$ in $X$. Thus, the fact that $\weakgrad f$ is the
minimal weak upper gradient gives that $\weakgrad f\leq G$
$\mm$-a.e.~in $X$, which is the thesis.\\* {(ii)} From the
hypothesis that $g$ is Sobolev along a.e. curve in $Y$ and supported
in a set having positive distance from $\partial Y$ it follows that
$f$ is Sobolev along a.e.~curve in $X$. To prove this, if $C$
denotes the support of $g$, notice first that for any absolutely
continuous curve $\gamma$ the set $L_r:=\{t\in [0,1]:\ {\rm
dist}(\gamma_t,C)=r\}$ is finite for a.e. $r$ (if $\gamma_t$ is
Lipschitz we can apply the coarea inequality, see for instance
\cite[Corollary~2.10.11]{Federer69}, in the general case we can
reparameterize $\gamma$). Now, setting $R:={\rm dist}(C,\partial
Y)>0$ and choosing $r\in (0,R)$ such that $L_r$ is finite, we can
use this set of times to split $\gamma$ in finitely many curves
contained in $Y$ and finitely many ones not intersecting $C$. The
equality $\weakgrad f=\nchi_Y\weakgrado {g}Y$ $\mm$-a.e.~in $X$ then
follows by locality (in $X\setminus Y$) and by $(i)$ (in $Y$, since
$g=f\restr{Y}$).
\\*{(iii)}
 Fix $r>0$, define
\[
Y_r:=\bigl\{x\in Y:\ \ \sfd(x,\partial Y)>r\bigr\},
\]
so that $Y_r\uparrow Y\setminus\partial Y$ as $r\downarrow 0$, and
let $\nchi_r:Y\to[0,1]$ be a Lipschitz cut-off function with support
contained in $Y\setminus Y_{r/2}$ and identically equal to 1 on
$Y_r$. Notice that, since $\mm(\partial Y)=0$, to prove the
quadratic property of $\C_Y$ it is sufficient to prove, for all
$r>0$, that $f\mapsto\int_{Y_r}\weakgrado f{Y}^2\,\d\mm_Y$ is
quadratic in the class of functions which are Sobolev along a.e.
curve in $Y$. By the previous points and the locality principle
\eqref{eq:weaklocal} we know that
\[
\weakgrado f{Y}=\weakgrad{(f\nchi_r)}=\weakgrad
f\qquad\text{$\mm$-a.e.~in $Y_r$,}
\]
so that the conclusion follows from \eqref{eq:cosimeglio} of
Theorem~\ref{thm:tangente}.
\end{proof}

\section{Riemannian Ricci bounds: definition}\label{sec:rcd}

Let $(X,\sfd,\mm)\in\X$. We say that $(X,\sfd,\mm)$ has
\emph{Riemannian Ricci curvature} bounded below by $K$ (in short, a
$\rcd K\infty$ space) if any of the 3 equivalent conditions of
Theorem~\ref{thm:mainriem} below is fulfilled. Basically, one adds
to the strong $CD(K,\infty)$ condition a linearity assumption on the
heat flow, stated either at the level of $\heatw_t$ or at the level
of $\heatl_t$. A remarkable fact is that all these conditions are
encoded in the $\EVI_K$ property of the gradient flow.

Before stating the theorem we observe that linearity at the level of
$\heatw_t$ will be understood as additivity, namely
$$
\heatw_t((1-\lambda)\mu+\lambda\nu)=(1-\lambda)\heatw_t(\mu)+\lambda\heatw_t(\nu)
\qquad\forall \mu,\,\nu\in\probt {X,\mm},\,\,\lambda\in [0,1].
$$

\begin{theorem}[3 general equivalences]\label{thm:mainriem}
Let $(X,\sfd,\mm)\in\X$. Then the following three properties are
equivalent.
\begin{itemize}
\item[(i)] $(X,\sfd,\mm)$ is a strong $CD(K,\infty)$ space and the semigroup $\heatw_t$ on $\probt {X,\mm}$
is additive.
\item[(ii)] $(X,\sfd,\mm)$  is a strong $CD(K,\infty)$ space and $\C$
  is a quadratic form on $L^2(X,\mm)$ according to \eqref{eq:18}.
\item[(iii)] $(X,\sfd,\mm)$ is a length space and any $\mu\in\probt {X,\mm}$ is the
starting point of an $\EVI_K$ gradient flow of $\entv$.
\end{itemize}
If any of these condition holds, the semigroups $\heatl_t$ and
$\heatw_t$ are also related for all $t\geq 0$ by
\begin{equation}\label{eq:heatlheatw}
(\heatl_t f)\mm=\int f(x)\,\heatw_t(\delta_x)\,\d\mm(x)
\qquad\forall f\in L^2(X,\mm),
\end{equation}
meaning that the signed measure $(\heatl_tf)\mm$ is the weighted
superposition, with weight $f(x)$, of the probability measures
$\heatw_t(\delta_x)$.
\end{theorem}
\begin{proof} $(i)\Rightarrow(ii)$. The additivity assumption on the heat semigroup,
the identification
Theorem~\ref{thm:heatgf} and the 1-homogeneity of the heat semigroup
in $L^2(X,\mm)$ (Proposition~\ref{prop:proprheat}), easily yield
that the heat semigroup $\heatl_t$ is linear in $L^2(X,\mm)$ as well
(see the proof of \eqref{eq:heatlheatw} below). This further implies
that its infinitesimal generator $-\Deltam$ is a linear operator, so
that $D(\Deltam)$ is a linear subspace of $L^2(X,\mm)$ and
$\Deltam:D(\Deltam)\to L^2$ is linear. Now, given $f\in D(\C)$, recall
that $t\mapsto \C(\heatl_t(f))$ is continuous on $[0,\infty)$ and
locally Lipschitz on $(0,\infty)$, goes to 0 as $t\to\infty$ and
$\tfrac{\d}{\d t}\C(\heatl_t(f))=-\|\Deltam \heatl_t(f)\|^2_2$ for
a.e.~$t>0$ (Theorem~\ref{thm:gfc}), thus
\[
\C(f)=\int_0^\infty\|\Deltam \heatl_t(f)\|^2_2\,\d t\qquad\forall
f\in D(\C).
\]
Now, recall that quadratic forms can be characterized in terms of
the parallelogram identity; thus $\C$, being on its domain an
integral of the quadratic forms $f\mapsto \|\Deltam
\heatl_t(f)\|^2_2$, is a quadratic form.

\noindent$(ii)\Rightarrow(iii)$. Using $(iii)$ of
Proposition~\ref{prop:evipropr}, to conclude it is sufficient to
show that $\heatw_t(\mu)$ is an $\EVI_K$ gradient flow for $\entv$
for any $\mu\ll\mm$ with density uniformly bounded away from 0 and
infinity. Thus, choose $\mu=\rho\mm\in\probt X$ such that
$0<c\leq\rho\leq C<\infty$ and define
$\mu_t:=\heatw_t(\mu)=\rho_t\mm$. By
Proposition~\ref{prop:eviequiv}, in order to check that $(\mu_t)$ is
an $\EVI_K$ gradient flow it is sufficient to pick reference
measures $\nu$ in \eqref{eq:defevi} of the form
$\nu=\sigma\mm\in\probt X$, with $\eta$ bounded and with bounded
support. For any $t>0$, choose $\ppi_t\in\gopt(\mu_t,\nu)$ given by
the strong $CD(K,\infty)$ condition and let $\varphi_t$ be a
Kantorovich potential for $(\mu_t,\nu)$. Now, recall that
Theorem~\ref{thm:derentr}(b) provides the lower bound
\[
\ent{\nu}-\ent{\mu_t}-\frac K2W_2^2(\mu_t,\nu)\geq
\lim_{\eps\downarrow
0}\frac{\C(\varphi_t)-\C(\varphi_t+\eps\rho_t)}{\eps},
\]
thus to conclude it is sufficient to show that for a.e.~$t>0$ it
holds
\begin{equation}
\label{eq:evilocal} \frac{\d}{\d t}\frac12W_2^2(\mu_t,\nu)\leq
\lim_{\eps\downarrow
0}\frac{\C(\varphi_t)-\C(\varphi_t+\eps\rho_t)}{\eps}.
\end{equation}
 By Theorem~\ref{thm:derw2} we know that for a.e.
$t>0$ it holds
\begin{equation}\label{eq:evilocalbis}
\frac{\d}{\d t}\frac12W_2^2(\mu_t,\nu)\leq\lim_{\eps\downarrow
0}\frac{\C(\rho_t-\eps\varphi_t)-\C(\rho_t)}{\eps}.
\end{equation}
In order to obtain \eqref{eq:evilocal} from \eqref{eq:evilocalbis}
we crucially use the hypothesis on Cheeger's energy: recalling
\eqref{eq:shift}, the fact that $\C$ is a quadratic form ensures the
identity
\[
 \frac{\C(\rho_t-\eps\varphi_t)-\C(\rho_t)}{\eps}
 =\frac{\C(\varphi_t)-\C(\varphi_t+\eps\rho_t)}{\eps}+O(\eps),
\]
(see \eqref{eq:shift}) and therefore \eqref{eq:evilocal} is proved.

\noindent$(iii)\Rightarrow(i)$. By Lemma~\ref{le:EVISCD} below we
deduce that $(X,\sfd,\mm)$ is a strong $CD(K,\infty)$ space.
 We turn to the additivity. Let $(\mu_t^0)$, $(\mu^1_t)$ be two
$\EVI_K$ gradient flows of the relative entropy and define
$\mu_t:=\lambda\mu^0_t+(1-\lambda)\mu^1_t$, $\lambda\in(0,1)$. To
conclude it is sufficient to show that $(\mu_t)$ is an $\EVI_K$
gradient flow of the relative entropy as well. By \eqref{eq:convw2}
we know that $(\mu_t)\subset\probt X$ is an absolutely continuous
curve, so we need only to show that
\begin{equation}
\label{eq:puntuale} \limsup_{h\downarrow0}\frac{\rme^{Kh}
W_2^2(\mu_{t+h},\nu)-W_2^2(\mu,\nu)}{2h}
+\entr{\mu_t}\mm\leq\entr\nu\mm,\qquad\forall t>0.
\end{equation}
Thus, given a reference measure $\nu$, for any $t>0$ let
$\ggamma_t\in\opt{\mu_t}{\nu}$, define
$\nu^0_t:=(\ggamma_t)_\sharp\mu^0_t$,
$\nu^1_t:=(\ggamma_t)_\sharp\mu^1_t$ (recall
Definition~\ref{def:pushplan} and notice that
$\mu^0_t,\mu^1_t\ll\mu_t=\pi^1_\sharp\ggamma_t$). By equation
\eqref{eq:optpush} we have
\[
\begin{split}
W_2^2(\mu_t,\nu)&=\int \sfd^2(x,y)\,\d\ggamma_t(x,y)=
\int \sfd^2(x,y)\Big((1-\lambda)\frac{\d\mu^0_t}{\d\mu_t}(x)+\lambda\frac{\d\mu^1_t}{\d\mu_t}(x)\Big)\d\ggamma_t(x,y)\\
&=(1-\lambda)\int \sfd^2(x,y)\frac{\d\mu^0_t}{\d\mu_t}(x)\,\d\ggamma_t(x,y)
+\lambda\int \sfd^2(x,y)\frac{\d\mu^1_t}{\d\mu_t}(x)\,\d\ggamma_t(x,y)\\
&=(1-\lambda)W_2^2(\mu^0_t,\nu^0_t)+\lambda W_2^2(\mu^1_t,\nu^1_t),
\end{split}
\]
while the convexity of $W_2^2$ yields
\[
W_2^2(\mu_{t+h},\nu)\leq(1-\lambda)W_2^2(\mu^0_{t+h},\nu^0_t)+\lambda W_2^2(\mu^1_{t+h},\nu^0_t),\qquad\forall h>0.
\]
Hence for any $t\geq 0$ we have
\begin{gather}
\notag \limsup_{h\downarrow0}
\frac{\rme^{Kh}\,W_2^2(\mu_{t+h},\nu)-W_2^2(\mu_t,\nu)}{2h} \leq
(1-\lambda) \limsup_{h\downarrow0}\frac{\rme^{Kh}\,
W_2^2(\mu^0_{t+h},\nu^0_t)-W_2^2(\mu^0_t,\nu^0_t)}{2h}
\\
+\lambda\limsup_{h\downarrow0}\frac{\rme^{Kh}\,
W_2^2(\mu_{t+h}^1,\nu^1_t)-W_2^2(\mu^1_t,\nu^1_t)}{2h}.
\label{eq:euno}
%\end{split}
\end{gather}
Now we use the assumption that $(\mu^0_t)$ and $(\mu^1_t)$ are
gradient flows in the $\EVI_K$ sense: fix $t$ and choose
respectively $\nu^0_t$ and $\nu^1_t$ as reference measures in
\eqref{eq:evipoint} to get
\begin{equation}
\label{eq:eddue}
\begin{split}
\limsup_{h\downarrow0}\frac{\rme^{Kh}\,
W_2^2(\mu^0_{t+h},\nu^0_t)-W_2^2(\mu^0_t,\nu^0_t)}{2h}
&\leq \entr{\nu^0_t}\mm-\entr{\mu^0_t}\mm,\\
\limsup_{h\downarrow0}\frac{\rme^{Kh}\,
W_2^2(\mu^1_{t+h},\nu^1_t)-W_2^2(\mu^1_t,\nu^1_t)}{2h} &\leq
\entr{\nu^1_t}\mm-\entr{\mu^1_t}\mm.
\end{split}
\end{equation}
Finally, from Proposition~\ref{prop:basegammapf} we have
\begin{equation}
\label{eq:ettre}
\entr{\mu_t}\mm-\entr\nu\mm\leq(1-\lambda)\Big(\entr{\mu^0_t}\mm-\entr{\nu^0_t}\mm\Big)
+\lambda\Big(\entr{\mu^1_t}\mm-\entr{\nu^1_t}\mm\Big).
\end{equation}
Inequalities \eqref{eq:euno}, \eqref{eq:eddue} and \eqref{eq:ettre} yield \eqref{eq:puntuale}.

Finally, we prove \eqref{eq:heatlheatw}. By linearity we can assume
that $f$ is a probability density. Notice that the additivity of the
semigroup $\heatw_t$ gives
$\heatw_t(\sum_ia_i\delta_{x_i})=\sum_ia_i\heatw_t(\delta_{x_i})$
whenever $a_i\geq 0$, $\sum_i a_i=1$ and $x_i\in\supp\mm$. Hence, if
$f$ is a continuous probability density in $\supp\mm$ with bounded
support, by a Riemann sum approximation we can use the continuity of
$\heatw_t$ to get
$$
\heatw_t(f\mm)=\int f(x)\heatw_t(\delta_x)\,\d\mm(x)
$$
and the identification of gradient flows provides
\eqref{eq:heatlheatw}. By a monotone class argument we extend the
validity of the formula from continuous to Borel functions $f\in
L^2(X,\mm)$.\end{proof}

\begin{lemma}
  \label{le:EVISCD}
  Any $(X,\sfd,\mm)\in\X$ satisfying
  condition $(iii)$ of Theorem~\ref{thm:mainriem} is a strong $CD(K,\infty)$ space.
\end{lemma}
\begin{proof}
  Let us first prove that $(X,\sfd,\mm)$ is a $CD(K,\infty)$ space.
  We notice that, since by assumption $(\supp\mm,\sfd)$ is a length space, then
  $\probt{X,\mm}$ is a length metric space.
  Therefore, up to a suitable reparameterization,
  for every $\mu_0,\,\mu_1\in D(\entv)\subset\probt{X,\mm}$ and $\eps>0$
  there exists a $L_\eps$-Lipschitz curve
  $(\mu^\eps)\in \mathrm{Lip}([0,1];\probt{X,\mm})$ connecting $\mu_0$ to
  $\mu_1$ with $L_\eps^2\le W^2_2(\mu_0,\mu_1)+\eps^2.$

  We can thus set $\tilde\mu^\eps_s:=\heatw_\eps(\mu_\eps)$,
  where $\heatw_t(\mu)$ denotes the $K$-gradient flow starting from $\mu$,
  so that by \eqref{eq:6} of Proposition~\ref{prop:danerisavare} we
  get
  \begin{displaymath}
    \ent{\tilde\mu^\eps_s}\le (1-s)\ent{\mu_0}+s\ent{\mu_1}-
    \frac K2s(1-s)W_2^2(\mu_0,\mu_1)+\frac {\eps^2}{\mathrm I_K(\eps)}
  \end{displaymath}
  Since $ {\eps^2}/{\mathrm I_K(\eps)}\to 0$ as $\eps\downarrow 0$,
  the family $\{\tilde\mu^\eps_s\}$ has uniformly bounded entropy and
  therefore
  it is tight. By $(ii)$ of Proposition~\ref{prop:evipropr} we know that
  \begin{displaymath}
    W_2(\tilde\mu^\eps_r,\tilde\mu^\eps_s)\le
    \rme^{-K\eps}L_\eps|r-s|\quad\forevery r,s\in [0,1],\ \eps>0.
  \end{displaymath}
  Since $\limsup_{\eps\down0}L_\eps\le W_2(\mu_0,\mu_1)$,
  we can applying the refined Ascoli-Arzel\`a compactness theorem of
  \cite[Prop.~3.3.1]{Ambrosio-Gigli-Savare08} to find a vanishing
  sequence $\eps_n\downarrow0$ and a limit curve $(\mu_s)\subset
  \probt{X,\mm}$ connecting
  $\mu_0$ to $\mu_1$ such that
  \begin{displaymath}
    \mu^{\eps_n}_s\to\mu_s\quad\text{in }\prob{X},\quad
    W_2(\mu_r,\mu_s)\le W_2(\mu_0,\mu_1)|r-s|,\quad
    \ent{\mu_s}<\infty\quad\forevery r,s\in [0,1].
  \end{displaymath}
  It turns out that $(\mu_s)$ is a geodesic in $D(\entv)$ connecting
  $\mu_0$ to $\mu_1$ and therefore Proposition~\ref{prop:danerisavare}
  shows that $\entv$ is $K$-convex along $\mu_s$.

  The same Proposition shows that $\entv$ is $K$-convex along any
  geodesic contained in $\probt{X,\mm}$: in particular,
  taking any optimal geodesic plan $\ppi$ as the one induced by
  the geodesic obtained by the previous argument,
  $\entv$ satisfies the $K$-convexity inequality associated to
  $\ppi_F$ as in Definition~\ref{def:strongcd}, since all the measures
  $\mu_{F,t}$ belong to $D(\entv)$.
\end{proof}

\section{Riemannian Ricci bounds: properties}\label{sec:rcdprop}

\subsection{Heat Flow}\label{sub7}

In this section we study more in detail the properties of the
$L^2$-semigroup $\heatl_t$ in a $\rcd K\infty$ space
$(X,\sfd,\mm)\in\X$ and the additional informations that one can
obtain from the relation \eqref{eq:heatlheatw} with the
$W_2$-semigroup $\heatw_t$. First of all, let us remark that since
$\C$ is quadratic the operator $\Deltam$ is in fact the
infinitesimal generator of $\heatl_t$, and therefore is linear.
Furthermore, denoting by $\mathcal E(u,v):[D(\C)]^2\to\R$ the
Dirichlet form induced by $\C$, the relation ${\cal E}(u,v)=-\int
v\Deltam u\,\d\mm$ for $u\in D(\Deltam)$, $v\in D(\C)$ implies that
$\Deltam$ is self-adjoint in $L^2(X,\mm)$ and the same is true for
$\heatl_t$.

Also, again by Proposition~\ref{prop:evipropr}, and the definition
of $\rcd K\infty$ spaces, we know that for any $x\in\supp\mm$ there
exists a unique $\EVI_K$ gradient flow $\heatw_t(\delta_x)$ of
$\entv$ starting from $\delta_x$, related to $\heatl_t$ by
\eqref{eq:heatlheatw}.

Since $\entr{\ke xt}\mm<\infty$ for any $t>0$, it holds $\ke
xt\ll\mm$, so that $\ke xt$ has a density, that we shall denote by
$\ked xt$. The functions $\ked xt (y)$ are the so-called transition
probabilities of the semigroup. By standard measurable selection
arguments we can choose versions of these densities in such a way
that the map $(x,y)\mapsto\ked xt(y)$ is $\mm\times\mm$-measurable
for all $t>0$.

In the next theorem we prove additional properties of the flows. The
information on both benefits of the identification theorem: for
instance the symmetry property of transition probabilities is not at
all obvious when looking at $\heatw_t$ only from the optimal
transport point of view, and heavily relies on
\eqref{eq:heatlheatw}, whose proof in turn relies on the
identification Theorem~\ref{thm:heatgf} proved in
\cite{Ambrosio-Gigli-Savare11}. On the other hand, the regularizing
properties of $\heatl_t$ are deduced by duality by those of
$\heatw_t$, using in particular the contractivity estimate (see \eqref{eq:21})
\begin{equation}
\label{eq:contrw2} W_2(\heatw_t(\mu),\heatw_t(\nu))\leq
e^{-Kt}W_2(\mu,\nu)\qquad t\geq0,\ \mu,\,\nu\in\probt {X,\mm}.
\end{equation}
and the regularization estimates for the Entropy and its slope
(apply \eqref{eq:ultraEVI} with $z:=\mm$)
\begin{equation}
  \label{eq:24}
  \mathrm I_K(t)\ent{\heatw_t(\mu)}+\frac {(\mathrm I_K(t))^2}2|\nabla^-\entv|^2(\heatw_t(\mu))
  \le \frac 12 W_2^2(\mu,\mm).
\end{equation}
Notice also that \eqref{eq:contrw2} yields $
W_1(\heatw_t(\delta_x),\heatw_t(\delta_y))\le\rme^{-Kt}\sfd(x,y)$
for all $x,\,y\in\supp\mm$ and $t\geq 0$. This implies that $\rcd
K\infty$ spaces have Ricci curvature bounded from below by $K$
according to \cite{Ollivier09}, \cite{Joulin07}. Notice also that
using \eqref{eq:heatlheatw}, the identification of gradient flows
and a simple convexity argument, we can recover the inequality
$$
W_1(\heatw_t(\mu),\heatw_t(\nu))\leq\rme^{-Kt}W_1(\mu,\nu)
$$
first with when $\mu,\,\nu\in\probt{X,\mm}$ have densities in
$L^2(X,\mm)$ and then, by approximation, in the general case when
$\mu,\,\nu\in {\mathscr P}_1(X,\mm)$.

Notice also that, as a consequence of \cite[Theorem~1.3]{Rajala11}
and the density of Lipschitz functions one has the weak local
$(1,1)$-Poincar\'e inequality
$$
\int_{B_r(x)} |u-\bar u|\,\d\mm\leq 4r\int_{B_{2r}(x)}\weakgrad
u\,d\mm\quad\text{with}\quad \bar
u:=\frac{1}{\mm(B_r(x))}\int_{B_r(x)}u\,\d\mm
$$
for all $x\in\supp\mm$, $r>0$, $u\in W^{1,2}(X,\sfd,\mm)$, which
implies the standard weak local $(1,1)$-Poincar\'e inequality under
doubling assumptions on $\mm$.

Further relevant properties will be obtained in the next section.
\begin{theorem}[Regularizing properties of the heat flow]\label{thm:facili}
Let $(X,\sfd,\mm)\in\X$ be a $\rcd K\infty$ space. Then:
\begin{itemize}
\item[(i)] The transition probability densities are symmetric
\begin{equation}\label{eq:symmetrictp}
\ked xt (y)=\ked yt (x)\qquad\text{$\mm\times\mm$-a.e.~in $X\times
X$, for all $t>0,$}
\end{equation}
and satisfy for all $x\in X$ the Chapman-Kolmogorov formula:
\begin{equation}
\label{eq:chapman} \ked x{t+s}(y)=\int\ked xt (z)\ked zs
(y)\,\d\mm(z)\qquad\text{for $\mm$-a.e.~$y\in X$, for all $t,\,s\geq
0$.}
\end{equation}
\item[(ii)] The formula
\begin{equation}
\label{eq:l1} \tilde\heatl_tf(x):=\int f(y)\,\d\ke xt(y)\qquad
x\in\supp\mm
\end{equation}
provides a version of $\heatl_t f$ for every $f\in L^2(X,\mm)$, an
extension of $\heatl_t$ to a continuous contraction semigroup in
$L^1(X,\mm)$ which is pointwise everywhere defined if $f\in
L^\infty(X,\mm)$.
\item[(iii)] The semigroup $\tilde\heatl_t$ maps contractively $L^\infty(X,\mm)$ in
$\Cb(\supp\mm)$ and, in addition, $\tilde\heatl_tf(x)$ belongs to $
\Cb\bigl((0,\infty)\times\supp\mm\bigr)$.
\item[(iv)] If $f:\supp\mm\to\R$ is Lipschitz, then $\tilde\heatl_tf$ is Lipschitz on $\supp\mm$
as well and $\Lip(\tilde\heatl_tf)\leq e^{-Kt}\Lip(f)$.
\end{itemize}
\end{theorem}
\begin{proof} $(i)$. Fix $f,\,g\in \Cb(X)$ and notice that \eqref{eq:heatlheatw}
gives
$$
\int g\,\heatl_t f\,\d\mm=\int f(x)\int g\,\d\ke xt\,\d\mm(x)=
\int\int f(x)g(y)\ked xt (y)\,\d\mm(y)\,\d\mm(x).
$$
Reversing the roles of $f$ and $g$ and using the fact that
$\heatl_t$ is self-adjoint it follows that $\int\int (\ked xt
(y)-\ked yt (x))f(x)g(y)\,\d\mm\,\d\mm$ vanishes, and since $f$ and
$g$ are arbitrary we obtain \eqref{eq:symmetrictp}. Formula
\eqref{eq:chapman} is a direct consequence of the semigroup property
$\heatw_{t+s}(\delta_x)=\heatw_t(\heatw_s(\delta_x))$.

\noindent $(ii)$ Using the symmetry of transition probabilities, for
$f\in L^1(X,\mm)$ nonnegative we get $\|\tilde\heatl_tf\|_1=\int\int
f(y)\ked yt (x)\,\d\mm(x)\,\d\mm(y)=\|f\|_1$. By linearity this
shows that $\tilde\heatl_t$ is well defined $\mm$-a.e. and defines a
contraction semigroup in $L^1(X,\mm)$. The fact that the right hand
side in \eqref{eq:l1} provides a version of $\heatl_tf$ follows once
more from \eqref{eq:heatlheatw} and the symmetry of transition
probabilities.

\noindent (iii) Contractivity of $\tilde\heatl_t$ in
$L^\infty(X,\mm)$ is straightforward. By \eqref{eq:contrw2} we get
that $\ke ys\to\ke xt$ in duality with $\Cb(X)$ when $y\to x$ in $X$
and $s\to t$. Also, the a priori estimate \eqref{eq:ultraEVI} shows
that $(t,y)\mapsto\entr{\ke yt}\mm$ is bounded on sets of the form
$(\epsilon,\infty)\times B$, with $B$ bounded and $\epsilon>0$. Thus
the family $\{\ked yt\}_{y\in B,t\geq\epsilon}$ is equi-integrable.
This shows that $\ked ys\to\ked xt$ weakly in $L^1(X,\mm)$ when
$(y,s)\to (x,t)\in X\times (0,\infty)$ and proves the continuity of
$\tilde\heatl_tf(x)$.

\noindent $(iv)$ By \eqref{eq:l1} we get $|\tilde\heatl_t
f(x)-\tilde\heatl_t f(y)|\leq {\rm Lip}(f)W_1(\ke xt,\ke yt)\leq
{\rm Lip}(f)W_2(\ke xt,\ke yt)$. We can now use \eqref{eq:contrw2}
to conclude (see \cite{Kuwada10} for a generalization of this
duality argument).
\end{proof}

Using Lemma~\ref{le:sug} below we can refine $(iv)$ of
Theorem~\ref{thm:facili} and prove by a kind of duality argument
\cite{Kuwada10} a Bakry-Emery estimate in $\rcd K\infty$ spaces.

\begin{theorem}[Bakry-Emery in $\rcd K\infty$ spaces]\label{thm:bakryemery}
For any $f\in D(\C)$ and $t>0$ we have
\begin{equation}\label{eq:bakryemery}
\weakgrad{(\heatl_tf)}^2\leq
\rme^{-2Kt}\heatl_t(\weakgrad{f}^2)\qquad\text{$\mm$-a.e.~in $X$.}
\end{equation}
In addition, if
 $\weakgrad{f}\in L^\infty(X,\mm)$ and $t>0$, then
$\rme^{-Kt}\bigl(\tilde\heatl_t\weakgrad{f}^2\bigr)^{1/2}$ is an
upper gradient of $\tilde\heatl_tf$ on $\supp\mm$, so that
\begin{equation}
\text{$|\nabla\tilde\heatl_tf|\leq\rme^{-Kt}
\bigl(\tilde\heatl_t\weakgrad{f}^2\bigr)^{1/2}$\quad pointwise in
$\supp\mm$,} \label{eq:17}
\end{equation}
and $f$ has a Lipschitz version $\tilde{f}:X\to\R$,
with ${\rm Lip}(\tilde{f})\leq\|\weakgrad{f}\|_\infty$.
\end{theorem}
\begin{proof} With no loss of generality we can assume, by a truncation argument,
that $f\in L^\infty(X,\mm)$.

Given $f:X\to\R$ Lipschitz and $x,\,y\in \supp\mm$, let $\gamma_s\in
\AC2{[0,1]}{\supp\mm}$ be connecting $x$ to $y$. Given $t>0$, we can
then apply Lemma~\ref{le:sug} with $\mu_s:=\ke {\gamma_s}t$ to get
\begin{eqnarray}\label{eq:intermediate}
|\tilde\heatl_t f(x)-\tilde\heatl_t f(y)|&=&\biggl|\int f\,\d\ke
xt-\int f\,\d\ke yt\biggr|\leq\int_0^1\bigl(\int |\nabla
f|^2\,\d\mu_s\bigr)^{1/2} |\dot\mu_s|\,\d s\\&\leq&
\rme^{-Kt}\int_0^1\bigl(\int |\nabla
f|^2\,\d\mu_s\bigr)^{1/2}|\dot\gamma_s|\,\d s
=\rme^{-Kt}\int_0^1\bigl(\tilde\heatl_t( |\nabla
f|^2)(\gamma_s)\bigr)^{1/2}|\dot\gamma_s|\,\d s\nonumber
\end{eqnarray}
(in the last inequality we used the contractivity property, which
provides the upper bound on $|\dot\mu_s|$). Notice that we can use
the length property of $\supp\mm$ to get, by a limiting argument,
\begin{equation}\label{eq:Alica2}
|\tilde\heatl_tf(x)-\tilde\heatl_t f(y)|\leq \sfd(x,y)\sup\left\{
\rme^{-Kt}\big(\tilde\heatl_t(|\nabla f|^2)\big)^{1/2}(z):\
\sfd(z,x)\leq 2\sfd(x,y)\right\}
\end{equation}
for all $x,\,y\in\supp\mm$. Taking the continuity of
$\tilde\heatl_t|\nabla f|^2$ into account, this implies the
Lipschitz Bakry-Emery estimate
$$
|\nabla\tilde\heatl_t f|^2\leq\rme^{-2Kt}\tilde\heatl_t|\nabla f|^2
\qquad\text{in $\supp\mm$.}
$$
To prove \eqref{eq:bakryemery} for functions $f\in D(\C)$ we
approximate $f$ in the strong $W^{1,2}$ topology by Lipschitz
functions $f_n$ in such a way that $|\nabla f_n|\to\weakgrad{f}$ in
$L^2(X,\mm)$ and use the stability properties of weak upper
gradients.

Now, assume that $L:=\|\weakgrad{f}\|_\infty$ is finite. From
\eqref{eq:bakryemery} we obtain that for all $t>0$ the continuous
function $f_t:=\tilde\heatl_t f$ satisfy
$\|\weakgrad{f_t}\|_\infty\leq L\rme^{-Kt}$. Given
$x,\,y\in\supp\mm$, fix $r>0$ and apply \eqref{eq:8} to find a
geodesic test plan $\ppi$ connecting $\mm(B_r(x))^{-1} \mm\res
B_r(x)$ to $\mm(B_r(y))^{-1} \mm\res B_r(y)$; the weak upper
gradient property then gives
$$
\biggl|\media_{B_r(x)}f_t \,\d\mm-
\media_{B_r(y)}f_t\,\d\mm\biggr|\leq \int
\int_0^1\weakgrad{f_t}(\gamma_s)|\dot\gamma_s|\,\d
s\,\d\ppi(\gamma).
$$
Now, since $(\e_t)_\sharp\ppi\ll\mm$ we can use
\eqref{eq:bakryemery} to estimate the right hand side as follows:
$$
\biggl| \media_{B_r(x)}f_t\,\d\mm-
\media_{B_r(y)}f_t\,\d\mm\biggr|\leq
(2r+\sfd(x,y))\rme^{-Kt}\sup\left\{\bigl(\tilde\heatl_t\weakgrad{f}^2\bigr)^{1/2}(z):\
\sfd(z,x)\leq 2r+\sfd(x,y) \right\}.
$$
We can now let $r\to 0$ to get
$$
\bigl|f_t(x)-f_t(y)|\leq\sfd(x,y)
\rme^{-Kt}\sup\left\{\bigl(\tilde\heatl_t\weakgrad{f}\bigr)^{1/2}(z):\
\sfd(z,x)\leq 2\sfd(x,y)\right\}.
$$
This provides the Lipschitz estimate on $f_t$, the upper gradient
property and \eqref{eq:17}. Finally, choosing a sequence
$(t_i)\downarrow 0$ such that $f_{t_i}\to f$ $\mm$-a.e. we obtain a
set $Y\subset\supp\mm$ of full $\mm$-measure such that $f\vert_{Y}$
is $L$-Lipschitz; $\tilde{f}$ is any $L$-Lipschitz extension of
$f\vert_{Y}$ to $X$.
\end{proof}

\begin{remark}[Weak Bochner inequality] {\rm
Following \emph{verbatim} the proof in
\cite[Theorem~4.6]{GigliKuwadaOhta10}, relative to the Alexandrov
case, one can use the Leibnitz rule of Theorem~\ref{thm:tangente}
and the Bakry-Emery estimate to prove Bochner's inequality (in the
case $N=\infty$)
$$
\frac12\Delta (\nabla f\cdot\nabla f)-\nabla(\Deltam f)\cdot\nabla
f\geq K\nabla f\cdot\nabla f.
$$
in a weak form. Precisely, for all $f\in D(\Deltam)$ with $\Delta
f\in W^{1,2}(X,\sfd,\mm)$ and all $g\in D(\Deltam)$ bounded and
nonnegative, with $\Deltam g\in L^\infty(X,\mm)$, it holds:
$$
\frac12\int_X\Delta g\weakgrad f^2\,\d\mm-\int_X g\nabla(\Deltam
f)\cdot\nabla f\,\d\mm\geq K\int_X g\weakgrad f^2\,\d\mm.
$$}\fr
\end{remark}

\begin{remark}[Lipschitz continuity of $\heatl_tf$ and $\heatw_t(\mu)$]
{\rm If we assume the stronger $L^1\mapsto L^p$ regularization
property
\begin{equation}\label{eq:stronger}
\|\heatl_t f\|_p \leq C(t)\|f\|_1\qquad\forevery f\in
L^2(X,\mm),\,\,t>0
\end{equation}
for some $p>1$, then we can improve \eqref{eq:bakryemery} to a
pointwise inequality as follows:
$$
|\nabla (\tilde\heatl_t f)|^2\leq
\rme^{-2Kt}\tilde\heatl_t(\weakgrad{f}^2)\qquad\text{in $\supp\mm$,
for all $f\in L^2(X,\mm)$.}
$$
Indeed, we can first use \eqref{eq:stronger} to get, by
approximation, $\|\ked xt\|_p\leq C(t)$ for all $x\in\supp\mm$.
Using the Young inequality for linear semigroups, this gives the
implication
$$
\Vert \heatl_t f\Vert_{q^*}\leq C(t)\|f\|_q\qquad\text{whenever
$\frac{1}{q^*}+1\leq\frac{1}{q}+\frac{1}{p}$.}
$$
Then, choosing $N\geq 1$ so large that $p\geq N/(N-1)$, by iterating
this estimate $N$ times the semigroup property yields the
$L^1\mapsto L^\infty$ regularization
\begin{equation}\label{eq:stronger1}
\sup_{\supp\mm}|\tilde\heatl_t f| \leq
(C(t/N))^N\|f\|_1\qquad\forall f\in L^1(X,\mm).
\end{equation}
Now we can apply the first part of Theorem~\ref{thm:bakryemery} to
$u_s:=\heatl_{t-s}f$, whose minimal weak upper gradient is in
$L^\infty(X,\mm)$, to obtain that
$\rme^{-Ks}(\tilde\heatl_s(\weakgrad{{u_s}}^2))^{1/2}$ is an upper
gradient of $\tilde\heatl_tf$ and then pass to the limit as
$s\downarrow 0$ to obtain that
$(\tilde\heatl_t(\weakgrad{f}^2))^{1/2}$ is an upper gradient of
$\tilde\heatl_tf$ on $\supp\mm$. Using the length property as in the
proof of Theorem~\ref{thm:bakryemery}, from this estimate the bound
on the slope of $\tilde\heatl_t f$ follows.

In particular we obtain that $\tilde\heatl_t f$ is Lipschitz on
$\supp\mm$ for all $t>0$. Using again the inequality $\|\ked
xt\|_\infty\leq C(t)$ for all $x\in\supp\mm$ and the semigroup
property we obtain that also $\heatw_t(\mu)$ has a Lipschitz density
for all $\mu\in\probt{X,\mm}$.

The stronger regularizing property \eqref{eq:stronger} is known to
be true, for instance, if doubling and Poincar\'e hold in
$(X,\sfd,\mm)$, see \cite[Corollary~4.2]{Sturm96}. Notice also that
Theorem~\ref{thm:lipreg} below ensures in any case the Lipschitz
regularization, starting from bounded functions.}\fr
\end{remark}

\subsection{Dirichlet form and Brownian
motion}\label{sub8}

In this section we fix a $\rcd K\infty$ space $(X,\sfd,\mm)$.
Recalling that the associated Cheeger's energy is a quadratic form,
we will denote by $\mathcal E$ the associated \emph{Dirichlet} form
as in Section~\ref{sec:quadratic}. In particular $\C$ satisfies all
the properties stated in Theorem~\ref{thm:tangente}.

Notice also that (see for instance
\cite[Theorem~5.2.3]{Fukushima80}) it is not difficult to compute
$[f]$ in terms of $\heatl_t$ or in terms of $\ke xt$ by
\begin{equation}\label{eq:fuku3}
[f]=\lim_{t\downarrow
0}\frac{1}{2t}(f^2+\heatl_tf^2-2f\heatl_tf),\qquad
[f](\varphi)=\lim_{t\downarrow 0}\frac{1}{2t}\int\int
(f(x)-f(y))^2\varphi(y)\,\d\ke xt (y)\,\d\mm(y).
\end{equation}

A direct application of the theory of Dirichlet forms yields the
existence of a Brownian motion in $(X,\sfd,\mm)$ with continuous
sample paths. Continuity of sample paths depends on a locality
property, which in our context holds in a particularly strong form,
see \eqref{eq:strongloc}.

\begin{theorem}[Brownian motion]\label{thm:brownian}
Let $(X,\sfd,\mm)$ be a $\rcd K\infty$ space. There exists a unique
(in law) Markov process $\{{\mathbf X}_t\}_{\{t\geq 0\}}$ in
$(\supp\mm,\sfd)$ with continuous sample paths in $[0,\infty)$ and
transition probabilities $\heatw_t(\delta_x)$, i.e.
\begin{equation}\label{eq:transitionmp}
{\mathbf P}\bigl({\mathbf X}_{s+t}\in A\bigl|{\mathbf
X}_s=x\bigr)=\ke xt(A) \qquad\forall s,\,t\geq 0,\,\,\text{$A$
Borel}
\end{equation}
for $\mm$-a.e. $x\in\supp\mm$.
\end{theorem}
\begin{proof} Uniqueness in law is obvious, since all
finite-dimensional distributions are uniquely determined by
\eqref{eq:transitionmp}, \eqref{eq:chapman} and the Markov property.

First, in the case when $(X,\sfd)$ is not locally compact, we prove
a tightness property arguing exactly as in
\cite[Theorem~1.2]{Ambrosio-Savare-Zambotti09},
\cite[Proposition~IV.4.2]{Rockner92} (the construction therein uses
only distance functions and the inequality $[\sfd(\cdot,x)]\leq\mm$)
to prove a tightness property, namely the existence of a
nondecreasing sequence of compact sets $F_n\subset\supp\mm$
satisfying ${\rm cap}_{{\mathcal E}}(\supp\mm\setminus F_n)\to 0$
(here ${\rm cap}_{{\mathcal E}}$ is the capacity associated to
${\mathcal E}$).

Since ${\mathcal E}$ is a strongly local Dirichlet form, and
Lipschitz functions are dense in $D({\cal E})$ for the $W^{1,2}$
norm (Proposition~\ref{prop:lipdense}) we may apply
\cite[Theorem~4.5.3]{Fukushima80} in the locally compact case or
\cite[Theorem~IV.3.5, Theorem~V.1.5]{Rockner92} in the general case
to obtain a Markov family $\{{\mathbf P}_x\}_{x\in\supp\mm}$ of
probability measures in $\mathrm C\bigl([0,\infty);X\bigr)$
satisfying
$$
\tilde\heatl_t f(x)=\int f(\gamma_t)\,\d{\mathbf
P}_x(\gamma)\qquad\text{for all $t\geq 0$, $f\in\Cb(X)$, $x\in
X\setminus N$}
$$
with $\mm(N)=0$. Then we can take the law ${\mathbb P}:=\int
{\mathbb P}_x\,\d\mm(x)$ in $\mathrm C\bigl([0,\infty);X\bigr)$ and
consider the canonical process ${\mathbf X}_t(\gamma)=\gamma(t)$ to
obtain the result.
\end{proof}

As a further step we consider the distance induced by the bilinear
form $\mathcal E$
\begin{equation}\label{eq:fuku2}
\sfd_{{\mathcal E}}(x,y):=\sup\left\{|\tilde g(x)-\tilde g(y)|:\
g\in D({\mathcal E}),\,\,[g]\leq\mm\right\}\qquad\forall (x,y)\in
\supp\mm\times\supp\mm,
\end{equation}
which we identify in Theorem~\ref{thm:idistances} with $\sfd$ (the
function $\tilde{g}$ is the continuous representative in the
Lebesgue class of $g$, see Theorem~\ref{thm:bakryemery}).

\begin{remark}\label{rem:KoskelaZhou}{\rm
In \cite{Koskela-Zhou11} the techniques of
\cite{GigliKuwadaOhta10,Ambrosio-Gigli-Savare11} are applied to a
case slightly different than the one considered here. The starting
point of \cite{Koskela-Zhou11} is a Dirichlet form $\mathcal E$ on a
measure space $(X,\mm)$ and $X$ is endowed with the distance
$\sfd_{\mathcal E}$. Assuming compactness of $(X,\sfd_{\mathcal
E})$, $K$-geodesic convexity of $\entv$ in $\probt{X}$ with cost
function $c=d_{\mathcal E}^2$, doubling, weak $(1,2)$-Poincar\'e
inequality and the validity of the so-called Newtonian property, the
authors prove that the $L^2(X,\mm)$ heat flow induced by $\mathcal
E$ coincides with $\heatw_t$. The authors also analyze some
consequences of this identification, as Bakry-Emery estimates and
the short time asymptotic of the heat kernel (a theme discussed
neither here nor in \cite{GigliKuwadaOhta10}). As a consequence of
\cite[Theorem~5.1]{Koskela-Zhou11} and
\cite[Theorem~9.3]{Ambrosio-Gigli-Savare11} the Dirichlet form
coincides with the Cheeger energy of $(X,\sfd_{\mathcal E},\mm)$
(because their flows coincide). This is a non trivial property,
because as shown in \cite{Sturm97}, a Dirichlet form is not uniquely
determined by its intrinsic distance, see also the next
result.\fr}\end{remark}

\begin{theorem}[Identification of $\sfd_{{\mathcal E}}$ and $\sfd$]\label{thm:idistances}
The function $\sfd_{{\mathcal E}}$ in \eqref{eq:fuku2} coincides
with $\sfd$ on $\supp\mm\times\supp\mm$.
\end{theorem}
\begin{proof} Choosing $g(z)=\sfd(z,x)$, since $[g]=\weakgrad{g}^2\mm\leq\mm$
we obtain immediately that $\sfd_{{\mathcal E}}(x,y)\geq\sfd(x,y)$
on $\supp\mm\times\supp\mm$. In order to prove the converse
inequality we notice that $[g]\leq\mm$ implies, by
Theorem~\ref{thm:bakryemery}, that the continuous representative
$\tilde g$ has Lipschitz constant less than $1$ in $X$, hence
$|\tilde{g}(x)-\tilde{g}(y)|\leq\sfd(x,y)$.
\end{proof}

We conclude this section with an example of application, following
the ideas of \cite{Bakry06}, of the calculus tools developed in
Section~\ref{sec:quadratic} combined with lower Ricci curvature
bounds, in particular with the Bakry-Emery estimate
\eqref{thm:bakryemery}.

\begin{theorem}[Lipschitz regularization]\label{thm:lipreg}
If $f\in L^2(X,\mm)$ then $\heatl_t f\in D(\cE)$ for
    every $t>0$ and
    \begin{equation}
      \label{eq:19}
      2\, \mathrm I_{2K}(t)\weakgrad {\heatl_t f }^2\le \heatl_t
      f^2\quad\mm\text{-a.e.\ in $X$};
    \end{equation}
    in particular,
  if $f\in L^\infty(X,\mm)$ then $\tilde\heatl_t f\in \Lip(\supp\mm)$ for every
  $t>0$ with
  \begin{equation}
    \label{eq:28}
    \sqrt{2\,\mathrm I_{2K}(t)}\,{\rm Lip}(\tilde\heatl_t f)\le
    \|f\|_\infty\quad\forevery t>0.
  \end{equation}
\end{theorem}
\begin{proof}
  Let us consider two bounded Lipschitz functions $f,\,\varphi$
  with $\varphi$ nonnegative, and let us set
  \begin{equation}
    \label{eq:27}
    G(s):=\int \big(\heatl_{t-s} f\big)^2\,\heatl_s \varphi\,\d\mm,\quad
    G(0)=\int \big(\heatl_{t} f\big)^2\, \varphi\,\d\mm,\quad
    G(t)=\int f^2\heatl_t\varphi\,\d\mm.
  \end{equation}
  It is easy to check that $G$ is of class $C^1$
  and, evaluating the derivative of $G$, we obtain
  thanks to \eqref{eq:fuku1}
  \begin{align*}
    G'(s)&=-\mathcal E\big((\heatl_{t-s} f)^2,\heatl_s
    \varphi\big)-2\int \heatl_{t-s} f\, \Deltam \heatl_{t-s} f\,
    \heatl_{s} \varphi\,\d\mm \\&= -\mathcal E\big((\heatl_{t-s}
    f)^2,\heatl_s \varphi\big)+2\mathcal E \bigl(\heatl_{t-s} f,
    \heatl_{t-s} f\, \heatl_{s} \varphi\bigr)\,\d\mm=
    2\int
    \weakgrad{(\heatl_{t-s} f)}^2 \heatl_s \varphi\,\d\mm.
  \end{align*}
  Using the fact that $\heatl_t$ is
  selfadjoint and applying the Bakry-Emery estimate \eqref{eq:bakryemery}
  we get
  \begin{align*}
    G'(s)&={\color{blue}2}\int \heatl_s\Big(\weakgrad{(\heatl_{t-s} f)}^2\Big)
    \varphi\,\d\mm \ge 2\,\rme^{2K s}\int \weakgrad{(\heatl_{t}
      f)}^2\,\varphi\,\d\mm
  \end{align*}
  and an integration in time yields
  \begin{equation}
    \label{eq:29}
    \int \biggl(\heatl_t f^2-\big(\heatl_t f\big)^2-{\color{blue}2}\,\mathrm I_{2K}(t)
    \weakgrad{(\heatl_{t} f)}^2\biggr)\varphi\,\d\mm\ge0.
  \end{equation}
  Since $\varphi$ is arbitrary nonnegative, neglecting the term $(\heatl_t f)^2$ we get
  the bound \eqref{eq:19}.
  We can now use
  Theorem~\ref{thm:bakryemery} to obtain \eqref{eq:28}.
\end{proof}

By duality one immediately gets:

\begin{corollary} [$W_1$-$L^1$ regularization]
  For every $x,\,y\in\supp\mm$ and $t>0$ we have
  \begin{equation}
    \label{eq:16}
    \sqrt{\mathrm I_{2K}(t)}\int \Big|\ked xt(z)-\ked yt(z)\Big|\,\d\mm(z)\le
    \sfd(x,y).
  \end{equation}
  More generally, the map $\sfh_t:\mu\mapsto
  {\d\heatw_t(\mu)}/{\d\mm}$ satisfies
  \begin{equation}
    \label{eq:23}
    \sqrt{\mathrm I_{2K}(t)}\,\|\sfh_t\mu-\sfh_t\nu\|_{L^1(X,\mm)}\le
    W_1(\mu,\nu).
  \end{equation}
\end{corollary}

\subsection{Stability}

Here we prove that the Riemannian Ricci curvature bounds are stable
w.r.t. $\D$-convergence. Notice that we will prove this by showing
that condition $(iii)$ of Theorem~\ref{thm:mainriem}, namely the
$\EVI$ property, is stable w.r.t. $\D$-convergence.

\begin{theorem}[Stability]
Let $(X_n,\sfd_n,\mm_n)\in\X$, $n\in\N$, be $\rcd K\infty$ spaces.
If
\[
\lim_{n\to\infty}\D\big((X_n,\sfd_n,\mm_n),(X,\sfd,\mm)\big)=0,
\]
then $(X,\sfd,\mm)$ is a $\rcd K\infty$ space as well.
\end{theorem}
\begin{proof}
We pass to the limit in  $(iii)$ of Theorem~\ref{thm:mainriem}. By
Proposition~\ref{prop:evipropr} it is sufficient to prove that for
any measure $\mu=\rho\mm$ with $\rho\in L^\infty(X,\mm)$, there
exists a continuous curve $(\mu_t)$ on $[0,\infty)$ starting from
$\mu$ which is locally absolutely continuous on $(0,\infty)$ and
satisfies
\begin{equation}
\label{eq:claimstab} \frac{\rme^{K(s-t)}}2 W_2^2(\mu_s,\nu)-\frac
12W_2^2(\mu_t,\nu)+ \mathrm I_K(s-t)\entr{\mu_s}{\mm} \leq \mathrm
I_K(s-t)\entr\nu{\mm}\qquad\forall t\leq s,
\end{equation}
for any $\nu\in\probt{X}$ with bounded density. Let
$C:=\|\rho\|_{\infty}$, choose optimal couplings $({\bf
d}_n,\ggamma_n)\in\opt{(\sfd,\mm)}{(\sfd_n,\mm_n)}$ and define
$\mu^n:=(\ggamma_n)_\sharp\mu\in\probt{X_n}$. Since
$(X_n,\sfd_n,\mm_n)$ is a $\rcd K\infty$ space, we know that there
exists a curve $t\mapsto\mu^n_t\in\probt{X_n}$ starting from $\mu^n$
such that
\begin{equation}
\label{eq:eviappr} \frac{\rme^{K(s-t)}}2 W_2^2(\mu^n_s,\nu^n)-\frac
12W_2^2(\mu^n_t,\nu^n)+ \mathrm I_K(s-t)\,\entr{\mu^n_s}{\mm} \leq
\mathrm I_K(s-t)\,\entr{\nu^n}{\mm}\qquad\forall t\leq s,
\end{equation}
where $\nu^n:=(\ggamma_n)_\sharp\nu$. By the maximum principle
(Proposition \ref{prop:proprheat}) we get $\mu^n_t\leq C\mm_n$ for
any $n,\,t$. Also, the energy dissipation equality \eqref{eq:ede}
yields that
\begin{equation}
\label{eq:equicont} \frac12\int_t^s|\dot{\mu^n_r}|^2dr\leq
\entr{\mu^n}{\mm^n}\leq C\log C,
\end{equation}
so that the curves $(\mu^n_t)$ are equi-absolutely continuous.

Now, define
$\tilde\mu^n_t:=(\ggamma_n)^{-1}_\sharp\mu^n_t\in\probt{X}$ for any
$n,\,t$ and notice that by $(i)$ of
Proposition~\ref{prop:basegammasharp} we have $\tilde\mu^n_t\leq
C\mm$ for any $n,\,t$.

We claim that the set of measures in $\probt X$ which are absolutely
continuous w.r.t. $\mm$ and with density bounded bounded above by
$C$ is compact w.r.t. $W_2$. Indeed the measure $\mm$ is tight and,
since it has finite second moment, also 2-uniformly integrable. Thus
the same is true for the set of measures less than $C\mm$, which is
therefore compact (see \cite{Ambrosio-Gigli-Savare08} Section 5.1
for the relevant definitions and properties).

By a diagonal argument we obtain a subsequence $n_k\uparrow\infty$
such that $\tilde\mu^{n_k}_t\to\mu_t$ in $(\probt X,W_2)$ as
$k\to\infty$ for any $t\in \Q\cap[0,\infty)$ and some
$\mu_t\in\probt X$. The equi-absolute continuity of the
$(\mu^n_t)$'s granted by \eqref{eq:equicont}, the uniform bound on
the densities and the equicontinuity of $(\ggamma_n)^{-1}_\sharp$
($(ii)$ of Proposition~\ref{prop:basegammasharp}) grant that there
is convergence for all times to a limit curve $(\mu_t)\subset\probt
X$ which is absolutely continuous as well.

To conclude, notice that by \eqref{eq:boundlimitato} we have
$W_2(\mu^{n_k}_t,\nu^{n_k})\to W_2(\mu_t,\nu)$ for any
$t\in[0,\infty)$, that lower semicontinuity and marginal
monotonicity of the entropy yield
\[
\entr{\mu_t}\mm\leq\liminf_{k\to\infty}\entr{\tilde\mu^{n_k}_t}\mm\leq\entr{\mu^{n_k}_t}{\mm_{n_k}}
\]
and $\entr{\nu^n}{\mm_n}\leq\entr\nu\mm$. Thus, we can pass to the
limit in \eqref{eq:eviappr} to get \eqref{eq:claimstab}.
\end{proof}

We remark that it looks much harder to pass to the limit in $(ii)$
of Theorem~\ref{thm:mainriem}, because in general we gain no
information about convergence of Cheeger's energies by the
$\D$-convergence of the spaces. To see why, just observe that in
\cite{Sturm06I} it has been proved that any space
$(X,\sfd,\mm)\in\X$ can be $\D$-approximated by a sequence of finite
spaces and that in these spaces Cheeger's energy is trivially null.

\subsection{Tensorization}

In this section we shall prove the following tensorization property
of $\rcd K\infty$ spaces:

\begin{theorem}[Tensorization]\label{thm:tensor}
Let $(X,\sfd_X,\mm_X)$, $(Y,\sfd_Y,\mm_Y)\in\X$ and define the
product space $(Z,\sfd,\mm)\in\X$ as $Z:=X\times Y$,
$\mm:=\mm_X\times\mm_Y$ and
\[
\sfd\big((x,y),(x',y')\big):=\sqrt{\sfd_X^2(x,x')+\sfd_Y^2(y,y')}.
\]
Assume that both $(X,\sfd_X,\mm_X)$ and $(Y,\sfd_Y,\mm_Y)$ are $\rcd
K\infty$ and nonbranching. Then $(Z,\sfd,\mm)$ is $\rcd K\infty$ and
non branching as well.
\end{theorem}

The proof of this result is not elementary. Before turning to the
details, we comment on the statement of the theorem: the non
branching assumption is needed in particular because, up to now, it
is not known whether the $CD(K,\infty)$ tensorizes or not: what is
known is that the product of two \emph{nonbranching} $CD(K,\infty)$
spaces is $CD(K,\infty)$ \cite[Proposition~4.16]{Sturm06I}. Thus,
the result follows combining this tensorization property with
another tensorization property at the level of Cheeger's energies,
proved in Theorem~\ref{thm:tensorquadratic}, that ensures that
Cheeger's energy in $Z$ is a quadratic form. Finally we use the
nonbranching assumption once more to show that $(Z,\sfd)$ is
nonbranching as well and therefore strong $CD(K,\infty)$ holds.

Throughout this section we assume that the base spaces
$(X,\sfd_X,\mm_X)$, $(Y,\sfd_Y,\mm_Y)$ are $\rcd K\infty$, even
though for the proof some intermediate results suffice weaker
assumptions.

Keeping the notation of Theorem~\ref{thm:tensor} in mind, given
$f:Z\to\R$ we shall denote $f^x$ the function $f(x,\cdot)$ and by
$f^y$ the function $f(\cdot,y)$. Having in mind Beppo-Levi's
pioneering paper \cite{BeppoLevi}, we denote by
$BL^{1,2}(Z,\sfd,\mm)$ the space of functions $f\in L^2(Z,\mm)$
satisfying:
\begin{itemize}
\item[(a)] $f^x\in D(\C^Y)$ for $\mm_X$-a.e.~$x\in X$ and $f^y\in
D(\C^X)$ for $\mm_Y$-a.e.~$y\in X$.
\item[(b)] $\weakgrad{f^y}^2(x)\in L^1(Z,\mm)$ and $\weakgrad{f^x}^2(y)\in L^1(Z,\mm)$.
\end{itemize}
For any $f\in BL^{1,2}(Z,\sfd,\mm)$ the \emph{cartesian} gradient
$$
\cartgrad{f}(x,y):=\sqrt{\weakgrad{f^y}^2(x)+\weakgrad{f^x}^2(y)}
$$
is well defined and belongs to $L^2(Z,\mm)$.

Accordingly, we shall denote by $\C^c:L^2(Z,\mm)\to [0,\infty)$ the
quadratic form associated to $\cartgrad{f}$, namely
$$
\C^c (f):= \int\C^X(f^y)\,\d\mm(y)+\int\C^Y(f^x)\,\d\mm(x)=
\frac12\int \cartgrad{f}^2(x,y)\,\d\mm(x,y),
$$
if $f\in BL^{1,2}(Z,\sfd,\mm)$, $+\infty$ otherwise. It is not hard
to show that the two terms which define $\C^c$ are
$L^2(Z,\mm)$-lower semicontinuous, which implies in particular that
$\C^c$ is lower semicontinuous: indeed, considering for instance
$\int\C^Y(f^x)\,\d\mm(x)$, suffices to check the lower
semicontinuity on (fastly) converging sequences satisfying
$\sum_n\|f_n-f\|_2^2<\infty$. By Fubini's theorem these sequences
satisfy $\sum_n\|f_n^x-f^x\|_2^2\,\d\mm(x)<\infty$, so that
$f^n_x\to f^x$ in $L^2(Y,\mm_Y)$ for $\mm$-a.e.~$x\in X$; then, the
lower semicontinuity of Cheeger's functional $\C^Y$ in the base
space $Y$ and Fatou's lemma provide the lower semicontinuity (the
same argument applies to $\int\C^X(f^y)\,\d\mm(y)$).

\begin{lemma}\label{lem:Levico1}
If $f\in {\rm Lip}(Z)$ then $\weakgrad{f}\leq\sqrt{|\nabla
f^x|^2+|\nabla f^y|^2}$ $\mm$-a.e.~in $Z$. In particular
\begin{equation}\label{eq:Levico6}
\weakgrad{f}\leq \sqrt{g_1^2+g_2^2}\qquad\text{$\mm$-a.e.~in $Z$}
\end{equation}
whenever $g_1,\,g_2:Z\to\R$ are bounded Borel functions such that
$g_1(x,\cdot)$ is a upper semicontinuous upper gradient of $f^x$ and
$g_2(\cdot,y)$ is a upper semicontinuous upper gradient of $f^y$.
\end{lemma}
\begin{proof} We will prove that the cartesian slope $\sqrt{|\nabla
f^x|^2+|\nabla f^y|^2}$ is a weak upper gradient for Lipschitz
functions $f$. If $\gamma=(\gamma^X,\gamma^Y)\in \AC2{[0,1]}Z$ we
need to prove that for a.e.~$t$ the inequality
\begin{equation}\label{eq:Levico1}
|\frac{\d}{\d t} (f\circ\gamma)|(t)\leq \sqrt{|\nabla
f^{\gamma^X_t}|^2(\gamma^Y_t)+|\nabla f^{\gamma^Y_t}|^2(\gamma^X_t)}
\sqrt{|\dot\gamma^X_t|^2+|\dot\gamma^Y_t|^2}
\end{equation}
holds. A pointwise proof of this inequality seems not to be easy, on
the other hand, working at the level of distributional derivatives,
in \cite[Lemma~4.3.4]{Ambrosio-Gigli-Savare08} it is proved that
a.e.~in $[0,1]$ it holds
$$
|\frac{\d}{\d t} (f\circ\gamma)|(t)\leq\limsup_{h\downarrow
0}\frac{|f(\gamma^X_{t-h},\gamma^Y_t)-f(\gamma^X_t,\gamma^Y_t)|}{h}+
\limsup_{h\downarrow
0}\frac{|f(\gamma^X_t,\gamma^Y_{t+h})-f(\gamma^X_t,\gamma^Y_t)|}{h},
$$
so that
$$
|\frac{\d}{\d t} (f\circ\gamma)|(t)\leq |\nabla
f^{\gamma^Y_t}|(\gamma^X_t)|\dot\gamma^X_t|+ |\nabla
f^{\gamma^X_t}|(\gamma^Y_t)|\dot\gamma^Y_t|\qquad\text{a.e.~in
$[0,1]$,}
$$
from which \eqref{eq:Levico1} readily follows. The estimate
\eqref{eq:Levico6} follows noticing that any upper semicontinuous
upper gradient bounds the slope from below.
\end{proof}
In the next lemma we will improve the inequality
$\weakgrad{f}\leq\sqrt{|\nabla f^x|^2+|\nabla f^y|^2}$ obtaining
$\cartgrad{f}$ in the right hand side. To this aim we consider, as a
regularizing operator, the product semigroup $\tilde\heatl_t^c$ in
$L^2(Z,\mm)$, pointwise defined by
\begin{equation}\label{eq:Levico10}
\tilde\heatl_t^c f(x,y):=\int\int
f(x',y')\rho_t^X[x](x')\rho_t^Y[y](y')\,\d\mm_X(x')\,\d\mm_Y(y')
\end{equation}
where $\rho_t^X[x](x')$ and $\rho_t^Y[y](y')$ are the transition
probability densities in the base spaces (see also
\eqref{eq:Levico7} below for an equivalent description in terms of
iterated operators). It is easy to show that $\tilde\heatl_t$
retains the same properties of its ``factors'' $\tilde\heatl_t^X$,
$\tilde\heatl_t^Y$ in the base spaces, in particular it is mass
preserving, self-adjoint, satisfies the maximum principle,
regularizes from $L^\infty(Z,\mm)$ to $\Cb(Z)$ and leaves ${\rm
Lip}(Z)$ invariant. In addition, $\tilde\heatl_t$ can also be viewed
as the $L^2(Z,\mm)$-gradient flow of $\C^c$, namely the solution to
\begin{equation}\label{eq:Prato2}
\frac{\d}{\d t}f_t=\Deltamc f_t
\end{equation}
where the linear operator $\Deltamc$ is defined in terms of the
Laplacians in the base spaces $\Delta_X$, $\Delta_Y$ by $\Deltamc
f(x,y):=\Delta_X f^y(x)+\Delta_Y f^x(y)$.

\begin{lemma}\label{lem:Levico2}
For all $f\in{\rm Lip}(Z)$ it holds $\weakgrad{f}\leq\cartgrad{f}$
$\mm$-a.e.~in $Z$.
\end{lemma}
\begin{proof} Set $F_1(x,y):=\weakgrad{f^y}(x)$ and
$F_2(x,y):=\weakgrad{f^x}(y)$. We consider the regularization
$f_t:=\tilde\heatl^c_t f$ of $f$. Writing
\begin{equation}\label{eq:Levico7}
f_t(x,y)=\tilde\heatl_t^X G(\cdot,y)(x)
\end{equation}
with $G(x',y):=\tilde\heatl_t^Y f(x',\cdot)(y)$, we can use first
Theorem~\ref{thm:bakryemery} and then the convexity of
$g\mapsto\weakgrad{g}$ to get
\begin{eqnarray*}
|\nabla f_t^y|(x)&\leq&
\rme^{-Kt}\bigl(\tilde\heatl_t^X\weakgrad{G(\cdot,y)}^2\bigr)^{1/2}(x)
\\&\leq&
\rme^{-Kt}\bigl(\tilde\heatl_t^X\tilde\heatl_t^YF_1^2\bigr)^{1/2}(x,y)\\
&=&\rme^{-Kt}\bigl(\tilde\heatl^c_t F_1^2\bigr)^{1/2}(x,y).
\end{eqnarray*}
Analogously, reversing the role of the variables we get
$$
|\nabla f_t^x|(y)\leq \rme^{-Kt}\bigl(\tilde\heatl^c_t
F_2^2\bigr)^{1/2}.
$$
So, we may take $g_1(x,y):=\rme^{-Kt}(\tilde\heatl_t^c F_1^2)^{1/2}$
and $g_2:=\rme^{-Kt}(\tilde\heatl_t^c F_2^2)^{1/2}$ in
Lemma~\ref{lem:Levico1} to get
$$
\weakgrad{f_t}^2\leq \rme^{-2Kt}\tilde\heatl_t^c
(F_1^2+F_2^2)=\tilde\heatl_t^c\cartgrad{f}^2 \qquad\text{$\mm$-a.e.
in $Z$.}
$$
Letting $t\downarrow 0$ the stability property of weak upper
gradients and the strong continuity of the semigroup provide the
result.
\end{proof}

The proof of the converse inequality is more involved. It rests
mainly in an improvement in product spaces of the Hamilton-Jacobi
inequality satisfied by the Hopf-Lax semigroup (see
Lemma~\ref{lem:hlimproved} below) and on its consequence, an
improved metric derivative that we obtain in
Lemma~\ref{le:keyimproved} along solutions to the
$L^2(Z,\mm)$-gradient flow of $\C^c$ defined in \eqref{eq:Levico10}
or, equivalently, in \eqref{eq:Prato2}.

In \cite[Section~3]{Ambrosio-Gigli-Savare11}, a very detailed
analysis of the differentiability properties of the Hopf-Lax
semigroup
\begin{equation}\label{eq:Prato6}
Q_tg(w):=\inf_{w'\in W} g(w')+\frac{1}{2t}\sfd_W^2(w',w)
\end{equation}
in a metric space $(W,\sfd_W)$ has been made. The analysis is based
on the quantities
$$
D^+_g(w,t):=\sup\limsup_{n\to\infty} \sfd_W(w,w_n'),\qquad
D^-_g(w,t):=\inf\liminf_{n\to\infty} \sfd_W(w,w_n),
$$
where the supremum and the infimum run among all minimizing
sequences $(w_n)$ in \eqref{eq:Prato6}. These quantities reduce
respectively to the maximum and minimum distance from $w$ of
minimizers in the locally compact case. Confining for simplicity our
discussion to the case of bounded functions, which suffices for our
purposes, it has been shown that $D^+_g$ and $D^-_g$ are
respectively upper and lower semicontinuous in $W\times (0,\infty)$,
that $D^-_g(\cdot,t)/t$ is an upper gradient of $Q_tg$ and that the
following pointwise equality holds:
\begin{equation}\label{eq:Prato4}
\frac{\d^+}{\d t}Q_tg(w)+\frac{(D^+_g(w,t))^2}{2t^2}=0,
\end{equation}
where we recall that $\d^+/\d t$ stands for right derivative (part
of the statement is its existence at every point). Notice that,
since $D^+_g(\cdot,t)/t\geq D^-_g(\cdot,t)/t$ is an upper
semicontinuous upper gradient of $Q_t g$, \eqref{eq:Prato4} implies
the Hamilton-Jacobi subsolution property $\tfrac{\d^+}{\d
t}Q_tg+|\nabla Q_tg|^2/2\leq 0$, but in the sequel we shall need the
sharper form \eqref{eq:Prato4}.

\begin{lemma}\label{lem:hlimproved}
Let $g:Z\to\R$ be a bounded function. Then, for all $t>0$ the
function $Q_tg$ satisfies
\begin{equation}\label{eq:Prato1}
\frac{\d^+}{\dt}Q_tg+\frac{1}{2}\cartgrad{Q_t g}^2\leq
0\qquad\text{$\mm$-a.e.~in $Z$.}
\end{equation}
\end{lemma}
\begin{proof} Taking \eqref{eq:Prato4} into account and the definition of
$\cartgrad{Q_tg}$ (recall the notation $f^x(y)=f(x,y)=f^y(x)$),
suffices to show that for all $t>0$ it holds
\begin{equation}\label{eq:Prato5}
\frac{[D^+_g((x,y),t)]^2}{t^2}\geq
\weakgrad{(Q_tg)^y}^2(x)+\weakgrad{(Q_tg)^y}^2(x) \qquad
\text{$\mm$-a.e.~in $Z$.}
\end{equation}
In order to prove \eqref{eq:Prato5}, notice that we can minimize
first in one variable and then in the other one to get
\begin{equation}\label{eq:Prato7}
(Q_tg)^y(x)=Q_t^X(L_{t,y})(x),\qquad (Q_tg)^x(y)=Q_t^Y(R_{t,x})(y),
\end{equation}
where $L_{t,y}(x'):=Q_t^Yg(x',\cdot)(y)$ and
$R_{t,x}(y'):=Q_t^Xg(\cdot,y')(x)$. Since $D^-(\cdot,t)/t$ is is an
upper gradient, we see that \eqref{eq:Prato5} is a consequence of
the pointwise inequality
\begin{equation}\label{eq:Prato8}
[D^+_g((x,y),t)]^2\geq
[D^-_{L_{t,y}}(x,t)]^2+[D^-_{R_{t,x}}(y,t)]^2.
\end{equation}
In order to prove \eqref{eq:Prato8}, let us consider a minimizing
sequence $(x_n,y_n)$ for $Q_tg (x,y)$; since
\begin{eqnarray*}
Q_tg(x,y)&=&\lim_{n\to\infty}
g(x_n,y_n)+\frac{1}{2t}\sfd_Y^2(y_n,y)+\frac{1}{2t}\sfd_X^2(x_n,x)\\&\geq&
\liminf_{n\to\infty}L^y_t(x_n)+\frac{1}{2t}\sfd_X^2(x_n,x)\geq
Q^X_t(L_{t,y})(x)
\end{eqnarray*}
we can use \eqref{eq:Prato7} to obtain that all inequalities are
equalities: this implies that the liminf is a limit and that and
that $(x_n)$ is a minimizing sequence for $Q_t^X\varphi(x)$, with
$\varphi(x)=L_{t,y}(x)$. Analogously, $(y_n)$ is a minimizing
sequence for $Q_t^Y\psi(y)$, where $\psi(y)=R_{t,x}(y)$. Taking into
account the definitions of $D^\pm$, this yields \eqref{eq:Prato8}.
\end{proof}

\begin{lemma}[Kuwada's lemma in product spaces]\label{le:keyimproved}
Let $f\in L^\infty(Z,\mm)$ be a probability density and let $f_t$ be
the solution of the $L^2$-gradient flow of $\C^c$ starting from $f$.
Then $\mu_t=f_t\mm\in\probt{X}$ for all $t\geq 0$ and
\begin{equation}\label{eq:Levico11}
|\dot\mu_t|^2\leq
\int_{\{f_t>0\}}\frac{\cartgrad{f_t}^2}{f_t}\,\d\mm\qquad\text{for
a.e.~$t>0$.}
\end{equation}
\end{lemma}
\begin{proof} The proof can be achieved following \emph{verbatim}
the proof of the analogous result
\cite[Lemma~6.1]{Ambrosio-Gigli-Savare11}, this time working with
$\cartgrad{f_t}$ in place of $\weakgrad{f_t}$: this replacement is
possible in view of the improved Hamilton-Jacobi inequality
\eqref{eq:Prato1} and of the calculus rules
\begin{equation}\label{eq:Prato3}
-\int g\Deltamc
f\,\d\mm\leq\int\cartgrad{f}\cartgrad{g}\,\d\mm,\qquad -\int
\phi(f)\Deltamc f\,\d\mm=\int\phi'(f)\cartgrad{f}^2\,\d\mm,
\end{equation}
which follow immediately by the analogous properties of the partial
Laplacians.
\end{proof}

\begin{proposition}\label{prop:Levico1}
We have $D(\C)\subset BL^{1,2}(Z,\sfd,\mm)$. In addition, for all
$f\in D(\C)$ there exist $f_n\in D(\C^c)$ converging to $f$ in
$L^2(Z,\mm)$ and satisfying
\begin{equation}\label{eq:Levico3}
\limsup_{n\to\infty}\C^c(f_n)\leq\C(f).
\end{equation}
\end{proposition}
\begin{proof} We argue exactly as in
\cite[Theorem~6.2]{Ambrosio-Gigli-Savare11}, where we identify weak
upper gradients and relaxed gradients, the only difference being the
use of the gradient flow of $\C^c$ and the improved estimate
\eqref{eq:Levico11}.

Pick $f\in D(\C)$. With a truncation argument, we can assume that
$c^{-1}\ge f\ge c>0$ $\mm$-almost everywhere in $Z$ with $\int
f^2\,\d\mm=1$. We consider the gradient flow $(h_t)$ of $\C^c$ with
initial datum $h:=f^2$, setting $\mu_t=h_t\mm$, and we apply
Lemma~\ref{le:keyimproved}. The maximum principle yields $c^{-1}\geq
f_t\geq c$ and a standard argument based on \eqref{eq:Prato2} and
\eqref{eq:Prato3} yields the energy dissipation identity
\begin{equation}\label{eq:Levico12}
\frac{\d}{\d t}\int f_t\log
f_t\,\d\mm=-\int\frac{\cartgrad{f_t}^2}{f_t}\,\,\d\mm.
\end{equation}

Let $g=h^{-1}\weakgrad h$, notice that by the chain rule we know
that $\log h$ is Sobolev along almost every curve and use the same
argument of \cite[Theorem~6.2]{Ambrosio-Gigli-Savare11} to get
$$
\int\big(h\log h-h_t \log h_t\big)\,\d\mm\leq\int \log
h(h-h_t)\,\d\mm\leq \Big(\int_0^t\int g^2 h_s\,\d\mm\,\d s
\Big)^{1/2}\Big(
  \int_0^t|\dot \mu_s|^2\,\d s\Big)^{1/2}.
$$
Now, inequality \eqref{eq:Levico11} gives
\begin{eqnarray*}
  \int\big(h\log h-h_t \log h_t\big)\,\d\mm&\le&
  \frac 12 \int_0^t\int  g^2 h_s\,\d\mm\,\d s+
  \frac 12 \int_0^t |\dot\mu_s|^2\,\d s
  \\&\le&
  \frac 12 \int_0^t\int  g^2 h_s\,\d\mm\,\d s+
  \frac 12 \int_0^t \int\frac{\cartgrad{h_s}^2}{h_s}\,\d\mm\,\d s.
\end{eqnarray*}
Recalling the entropy dissipation formula \eqref{eq:Levico12} we
obtain
$$
  \int_0^t \int\frac {\cartgrad{h_s}^2}{h_s}\,\d\mm\,\d s
  \le \int_0^t\int g^2h_s\,\d\mm\,\d s.
$$
Now, the chain rule and the identity $g=2f^{-1}\weakgrad f$ give
$\int_0^t\C^c(\sqrt{h_s})\,\d s\leq\int_0^t\int\weakgrad f^2f^{-2}
h_s\,\d\mm\,\d s$, so that dividing by $t$ and passing to the limit
as $t\downarrow0$ we get \eqref{eq:Levico3}, since $\sqrt{h_s}$ are
equibounded and converge strongly to $f$ in $L^2(Z,\mm)$ as
$s\downarrow0$.
\end{proof}
\begin{theorem}\label{thm:tensorquadratic}
Let $f\in L^2(Z,\mm)$. Then $f\in D(\C)$ if and only if $f\in
D(\C^c)$ and $\weakgrad{f}=\cartgrad{f}$ $\mm$-a.e.~in $Z$. In
particular $\C=\C^c$ is a quadratic form.
\end{theorem}
\begin{proof} By Proposition~\ref{prop:Levico1} we obtain that $f\in D(\C)$
implies $f\in D(\C^c)$ and $\C^c(f)\leq\C(f)$. If $f\in {\rm
Lip}(Z)$, Lemma \ref{lem:Levico2} yields $\weakgrad f\leq\cartgrad
f$ $\mm$-a.e.~in $Z$ and the converse inequality $\C(f)\leq\C^c(f)$.
It follows that the functionals and the gradients coincide in ${\rm
Lip}(Z)$. Since ${\rm Lip}(Z)$ is a $L^2(Z,\mm)$-dense and invariant
subset for $\heatl_t^c$, for all $f\in D(\C^c)$ we can apply
Lemma~\ref{lem:invariant} to obtain $(f_n)\subset {\rm Lip}(Z)$
satisfying $\C^c(f-f_n)\to 0$ and we can pass to the limit as
$n\to\infty$ in the inequality $\weakgrad{f_n}\leq\cartgrad{f_n}$ to
get $\weakgrad{f}\leq\cartgrad{f}$. Hence, $\C(f)=\C^c(f)$ and the
respective gradients coincide.
\end{proof}

\begin{proof} (\emph{of Theorem~\ref{thm:tensor}}) By \cite[Proposition~4.16]{Sturm06I}
we know that $(Z,\sfd,\mm)$ is $CD(K,\infty)$, while
Theorem~\ref{thm:tensorquadratic} ensures that Cheeger's energy in
this space is a quadratic form.

The proof that $(Z,\sfd,\mm)$ is nonbranching is simple, and we just
sketch the argument. It is immediately seen that the non branching
property is implied by the stability of constant speed geodesics
under projections, namely if $\gamma=(\gamma^X,\gamma^Y)\in\geo(Z)$,
then $\gamma^X\in\geo(X)$ and $\gamma^Y\in\geo(Y)$. This stability
property can be shown as follows: in any metric space, constant
speed geodesics are characterized by
$$
\int_0^1|\dot\gamma_t|^2\,\d t= \sfd^2(\gamma_0,\gamma_1),
$$
while for all other curves the inequality $\geq$ holds. Since
$|\dot\gamma_t|^2=|\dot\gamma^X_t|^2+|\dot\gamma^Y_t|^2$ wherever
the metric derivatives of the components exist, we obtain
$\int_0^1|\dot\gamma_t^X|^2\,\d t= \sfd^2_X(\gamma^X_0,\gamma^X_1)$
and $\int_0^1|\dot\gamma^Y_t|^2\,\d t=
\sfd^2_Y(\gamma^Y_0,\gamma^Y_1)$, so that both $\gamma^X$ and
$\gamma^Y$ are constant speed geodesics.

Finally, we prove the strong $CD(K,\infty)$ property. Since the
space is nonbranching, by Remark~\ref{re:nonbr} it is sufficient to
prove that it is a $CD(K,\infty)$ space. To prove this, we argue
exactly as in \cite[Lemma 4.7 and Proposition 4.16]{Sturm06I},
taking into account the tightness of the sublevels of $\entv$ to
remove the compactness assumption. We omit the details.\end{proof}

\subsection{Locality}

Here we study the locality properties of $\rcd K\infty$ spaces. As
for the tensorization, we will adopt the point of view of the
definition coming from the Dirichlet form, rather than the ones
coming from the properties of the heat flow. The reason is simple.
On one side, the heat flow does not localize at all: even on $\R^d$
to know how the heat flow behaves on the whole space gives little
information about the behavior of the flow on a bounded region (we
recall that, with our definitions, the heat flow that we consider
reduces to the classical one with homogeneous Neumann boundary
condition). On the other hand, Cheeger's energy comes out as a local
object, and we will see that the analysis carried out in
Section~\ref{sec:quadratic} and Section~\ref{sub8} will allow us to
quickly derive the locality properties we are looking for.

There are two questions we want to answer. The first one is: say
that we have a $\rcd K\infty$ space and a convex subregion, can we
say that this subregion - endowed with the restricted distance and
measure - is a $\rcd K\infty$ space as well? The second one is:
suppose that a space is covered by subregions, each one being a
$\rcd K\infty$ space, can we say that the whole space is $\rcd
K\infty$?

The first question has a simple answer: yes. The second one is more
delicate, the problem coming from proving the convexity of the
entropy. The analogous question for $CD(K,\infty)$ spaces has, as of
today, two different answers. On one side there is Sturm's result
\cite[Theorem~4.17]{Sturm06I} saying that this local-to-global
property holds if the space is nonbranching and the domain of the
entropy is geodesically convex. On the other side there is Villani's
result \cite[Theorem~30.42]{Villani09}) which still requires the
space to be nonbranching, but replaces the global convexity of the
domain on the entropy, with a local one one, roughly speaking
``$(X,\sfd,\mm)$ is finite-dimensional near to every point'' (in a
sense which we won't specify).

Our answer to the local-to-global question in the $\rcd K\infty$
setting will be based on the following assumptions, besides the
obvious one that the covering subregions are $\rcd K\infty$: the
space is nonbranching and $CD(K,\infty)$, so that independently from
the approach one has at disposal to prove the local to global for
$CD(K,\infty)$, as soon as the space is nonbranching, $\rcd K\infty$
globalizes as well.

We say that a subset $Y$ of a metric space $(X,\sfd)$ is
\emph{convex} if, for any $x,\,y\in Y$, there exists a geodesic
$\gamma$ connecting $x$ to $y$ is contained in $Y$.

\begin{theorem}[Global to Local]
Let $(X,\sfd,\mm)\in\X$ be a $\rcd K\infty$ space and let $Y\subset
X$ be a closed convex set such that $\mm(\partial Y)=0$ and
$\mm(Y)>0$. Then $(Y,\sfd,\mm_Y)$ is a $\rcd K\infty$ space as well,
where $\mm_Y:=(\mm(Y))^{-1}\mm\res Y$.
\end{theorem}
\begin{proof} Since $Y$ is closed, $(Y,\sfd,\mm_Y)\in\X$. Let us first remark
that for every $\mu\in \probt X$
\begin{equation}
  \label{eq:32}
  \begin{gathered}
    \entr\mu{\mm_Y}<\infty\qquad\Leftrightarrow\qquad \supp\mu\subset
    Y,\quad  \ent\mu<\infty,
 \end{gathered}
\end{equation}
and in this case $\entr\mu{\mm_Y}=c_Y +\ent\mu$, where
$c_Y=\log(\mm(Y))$. Therefore, thanks to the $RCD(K,\infty)$
property of $(X,\sfd,\mm)$, the functional $\mathrm{Ent}_{\mm_Y}$ is
$K$-geodesically convex on any Wasserstein geodesic $(\mu_s)$ with
$\supp\mu_s\subset Y$ for all $s\in [0,1]$. Such a geodesic exists
since $(Y,\sfd)$ and thus $(\probt Y,W_2)$ are geodesic spaces: in
particular, $(Y,\sfd_Y,\mm_Y)$ is a strong $CD(K,\infty)$ space.
Thus, to conclude we simply apply $(iii)$ of
Theorem~\ref{thm:gradristretti}.
\end{proof}

The previous result is similar to the following lower Ricci
curvature bound for weighted spaces:

\begin{proposition}[Weighted spaces]
  \label{prop:weighted}
  Let $(X,\sfd,\mm)\in\X$ be a $RCD(K,\infty)$ space and let
  $V:X\mapsto \R$ be a continuous $H$-geodesically convex function
  bounded from below with $\int \rme^{-V}\,\d\mm=1$. Then
  $(X,\sfd,\rme^{-V}\mm)$ is a $RCD(K+H,\infty)$ space.
\end{proposition}
The proof follows by the same arguments, applying
\cite[Proposition~4.14]{Sturm06I} (showing that
$(X,\sfd,\rme^{-V}\mm)$ is a strong $CD(K+H,\infty)$ space),
\cite[Lemma~4.11]{Ambrosio-Gigli-Savare11} for the invariance of
weak gradients with respect to the multiplicative perturbation, and
$(iii)$ of Theorem~\ref{thm:tangente}.

We conclude this section with the globalization result.
\begin{theorem}[Local to Global]
Let $(X,\sfd,\mm)\in\X$ and let $\{Y_i\}_{i\in I}$ be a cover of $X$
made of finitely or countably many closed sets of positive
$\mm$-measure, with $\mm_i:=[\mm(Y_i)]^{-1}\mm\res{Y_i}$. Assume
that the Cheeger functional associated to $(Y_i,\sfd,\mm_i)$ is
quadratic (in particular when $(Y_i,\sfd,\mm_i)$ is a $\rcd K\infty$
space) for every $i\in I$. Assume also that $(X,\sfd,\mm)$ is
nonbranching and $CD(K,\infty)$. Then $(X,\sfd,\mm)$ is a $\rcd
K\infty$ space.
\end{theorem}
\begin{proof}
We start proving that $\C$ is a quadratic form. Notice that it holds
\[
2\C(f)=\int_X\weakgrad f^2\,\d\mm=\sum_i \int_{X_i}\weakgrad
f^2\,\d\mm
\]
where $X_i:=Y_i\setminus\cup_{j<i}Y_j$. Let $f_i:=f\restr{Y_i}$ and
recall that by Theorem~\ref{thm:gradristretti} we know that
$\weakgrado{f_i}{Y_i}=\weakgrad f$ $\mm$-a.e.~on $Y_i$. Also, by
Theorem~\ref{thm:tangente} we have that for $i$ and any Borel subset
$A$ of $Y_i$, the map $f\mapsto\int_A\weakgrado f{Y_i}^2\,\d\mm$ is
quadratic. Choosing $A=X_i$ the conclusion follows.

The fact that $(X,\sfd,\mm)$ is a strong $CD(K,\infty)$ space
follows from the fact that it is nonbranching and $CD(K,\infty)$, as
in Remark~\ref{re:nonbr}.
\end{proof}

\def\cprime{$'$}

% \bibliographystyle{siam}
% \bibliography{bibliografia2011bis}
\end{document}